\documentclass[aap]{imsart}

\RequirePackage{amsthm,amsmath,amsfonts,amssymb}
\RequirePackage[numbers]{natbib}
\RequirePackage{graphicx}%% uncomment this for including figures

\startlocaldefs
\numberwithin{equation}{section}
\theoremstyle{plain}
\newtheorem{theorem}{Theorem}[section]
\newtheorem*{theorem*}{Theorem}
\newtheorem*{openproblem*}{Open problem}

\newtheorem{proposition}[theorem]{Proposition}
\newtheorem{lemma}[theorem]{Lemma}
\newtheorem{corollary}[theorem]{Corollary}
\newtheorem*{corollary*}{Corollary}
\theoremstyle{remark}

\newtheorem{assumption}{Assumption}
\newtheorem{remark}[theorem]{Remark}
\usepackage{amsmath,amssymb,amsthm}

\usepackage{caption}
\usepackage{subcaption}

\usepackage{hyperref}
\hypersetup{colorlinks,
            breaklinks,
            linkcolor=blue,
            urlcolor=blue,
            anchorcolor=blue,
            citecolor=blue}
\usepackage[capitalize,nameinlink]{cleveref}
\crefname{assumption}{Assumption}{Assumptions}

\usepackage{todonotes}
\newcommand{\R}{\mathbb{R}}
\newcommand{\E}{\mathbb{E}}
\newcommand{\Z}{\mathbb{Z}}
\renewcommand{\P}{\mathbb{P}}
\newcommand{\N}{\mathbb{N}}

\newcommand{\eps}{\varepsilon}
\renewcommand{\epsilon}{\varepsilon}

\DeclareMathOperator{\dist}{dist}
\newcommand{\norm}[1]{\left\|#1\right\|}
\newcommand{\abs}[1]{\left|#1\right|}
\newcommand{\st}{\,:\,}
\newcommand{\de}{\,\mathrm{d}}
\renewcommand{\d}{\,\mathrm{d}}

\DeclareMathOperator{\supp}{supp}
\newcommand{\g}{\, | \, }
\newcommand{\gm}{\rev \, \middle| \, \nc}
\newcommand{\F}{{\mathcal F}}

\newcommand{\comment}[1]{}

\numberwithin{equation}{section}

\newcommand{\probSpace}{\Omega}
\newcommand{\sigAlg}{\mathcal{F}}
\newcommand{\probMeasure}{\mathbb{P}}
\newcommand{\Exp}[1]{\E\left[#1\right]}
\newcommand{\Prob}[1]{\probMeasure\left(#1\right)}
\newcommand{\Set}[1]{\left\lbrace#1\right\rbrace}

\newcommand{\closure}[1]{\overline{#1}}
\newcommand{\domain}{\Omega}
\newcommand{\constr}{\mathcal{O}}

\newcommand{\nlscale}{\eps}

\newcommand{\grad}{\nabla}

\newcommand{\dprime}{{\prime\prime}}
\newcommand{\tprime}{{\prime\prime\prime}}
\newcommand{\qprime}{{\prime\prime\prime\prime}}
\newcommand{\scale}{\eps}
\newcommand{\homscale}{\tau}

\renewcommand{\vert}{\,|\,}
\DeclareFontFamily{U}{mathx}{\hyphenchar\font45}
\DeclareFontShape{U}{mathx}{m}{n}{
      <5> <6> <7> <8> <9> <10>
      <10.95> <12> <14.4> <17.28> <20.74> <24.88>
      mathx10
      }{}
\DeclareSymbolFont{mathx}{U}{mathx}{m}{n}
\DeclareMathSymbol{\bigtimes}{1}{mathx}{"91}
\def\journal@name{}
\newcommand{\dprimeS}[2]{d_{#1, #2}}

\newcommand{\nc}{\normalcolor}

\DeclareMathOperator{\Lip}{Lip}
\newcommand{\rev}{}%\color{magenta}}

\endlocaldefs

\begin{document}

\begin{frontmatter}
\title{Ratio convergence rates for Euclidean first-passage percolation: Applications to the graph infinity Laplacian}
%\title{A sample article title with some additional note\thanksref{T1}}
\runtitle{Ratio convergence rates for first-passage percolation}
%\thankstext{T1}{A sample of additional note to the title.}

\begin{aug}
%%%%%%%%%%%%%%%%%%%%%%%%%%%%%%%%%%%%%%%%%%%%%%%
%% Only one address is permitted per author. %%
%% Only division, organization and e-mail is %%
%% included in the address.                  %%
%% Additional information can be included in %%
\author[A]{\fnms{Leon}~\snm{Bungert}\ead[label=e1]{leon.bungert@uni-wuerzburg.de}\orcid{0000-0002-6554-9892}},
\author[B]{\fnms{Jeff}~\snm{Calder}\ead[label=e2]{jwcalder@umn.edu}\orcid{0000-0002-9829-4128}}
\and
\author[C]{\fnms{Tim}~\snm{Roith}\ead[label=e3]{tim.roith@desy.de}\orcid{0000-0001-8440-2928}}
%%%%%%%%%%%%%%%%%%%%%%%%%%%%%%%%%%%%%%%%%%%%%%
%% Addresses                                %%
%%%%%%%%%%%%%%%%%%%%%%%%%%%%%%%%%%%%%%%%%%%%%%
\address[A]{Institute of Mathematics,
  University of Würzburg, Emil-Fischer-Str. 40, Würzburg, Germany\printead[presep={,\ }]{e1}}

\address[B]{School of Mathematics, University of Minnesota, 127 Vincent Hall, 206 Church St. S.E., Minneapolis, MN 55455, USA\printead[presep={,\ }]{e2}}

\address[C]{Helmholtz Imaging, 
Deutsches Elektronen-Synchrotron DESY, Notkestr. 85, 22607 Hamburg, Germany\printead[presep={,\ }]{e3}}
\end{aug}

\begin{abstract}
    In this paper we prove the first quantitative convergence rates for the graph infinity Laplace equation for length scales at the connectivity threshold.
    In the graph-based semi-supervised learning community this equation is also known as Lipschitz learning.
    The graph infinity Laplace equation is characterized by the metric on the underlying space, and convergence rates follow from convergence rates for graph distances. At the connectivity threshold, this problem is related to Euclidean first passage percolation, which is concerned with the Euclidean distance function $d_{h}(x,y)$ on a homogeneous Poisson point process on $\mathbb{R}^d$, where admissible paths have step size at most $h>0$. 
    Using a suitable regularization of the distance function and subadditivity we prove that ${d_{h_s}(0,se_1)}/ s \to \sigma$ as $s\to\infty$ almost surely where $\sigma \geq 1$ is a dimensional constant and $h_s\gtrsim \log(s)^{1/d}$. A convergence rate is not available due to a lack of approximate superadditivity when $h_s\to \infty$. Instead, we prove convergence rates for the \emph{ratio} $\frac{d_{h}(0,se_1)}{d_{h}(0,2se_1)}\to \frac{1}{2}$ when $h$ is frozen and does not depend on $s$.  
    Combining this with the techniques that we developed in (Bungert, Calder, Roith, {\it IMA Journal of Numerical Analysis}, 2022), we show that this notion of ratio convergence is sufficient to establish uniform convergence rates for solutions of the graph infinity Laplace equation at percolation length scales.
\end{abstract}



\end{frontmatter}
%%%%%%%%%%%%%%%%%%%%%%%%%%%%%%%%%%%%%%%%%%%%%%
%% Please use \tableofcontents for articles %%
%% with 50 pages and more                   %%
%%%%%%%%%%%%%%%%%%%%%%%%%%%%%%%%%%%%%%%%%%%%%%
\tableofcontents

%%%%%%%%%%%%%%%%%%%%%%%%%%%%%%%%%%%%%%%%%%%%%%
%%%% Main text entry area:

\section{Introduction}

In this paper we will use techniques from first-passage percolation to prove new and strong results in the field of partial differential equations on graphs. 
In more detail, we will exploit stochastic homogenization effects in Euclidean first-passage percolation on Poisson point clouds to derive uniform convergence rates for the infinity Laplacian equation on a random geometric graph with $n$ vertices in $\R^d$ whose connectivity length scale $\eps_n$ is proportional to the connectivity threshold, i.e.,
\begin{align*}
    \eps_n \sim \left(\frac{\log n}{n}\right)^\frac{1}{d}.
\end{align*}
Our approach is based on the insight from our previous work \cite{bungert2022uniform} that convergence rates for the graph distance function translate to convergence rates for solutions of the graph infinity Laplace equation which can be regarded as a generalized finite difference method and which, in the context of semi-supervised learning, is also known as Lipschitz learning.

While the fields of percolation theory and partial differential equations (PDEs) on graphs (including finite difference methods  and semi-supervised learning) are very well developed, there are relatively few results that connect them, such as \cite{braides2022asymptotic,caroccia2022compactness} which deals with Gamma-convergence of discrete Dirichlet energies on Poisson clouds or \cite{groisman2022nonhomogeneous} on distance learning from a Poisson cloud on an unknown manifold.
In the following we give a brief overview of first-passage percolation and graph PDEs.

\textbf{First-passage percolation:}

First-passage percolation was introduced in \cite{broadbent1957percolation,hammersley1965first} as a model for the propagation of fluid through a random medium.
In mathematical terms, the set-up is a graph $G=(V,E)$ whose edges are equipped with passage times $t(e)\in[0,\infty]$ and one would like to understand the graph distance function between vertices~$x,y\in V$:
\begin{align}
    T(x,y) := \inf\left\lbrace \sum_{i=1}^m t(e_i) \st m\in\N,\; (e_1,\dots,e_m) \text{ connects $x$ and $y$}\right\rbrace.
\end{align}
Typical questions address properties of geodesics, shape theorems, size of connected components of the graph, and large scale asymptotics of the graph distance. 

Stochasticity can enter the model in different ways.  In the simplest set-up the graph consists of the square lattice $\Z^d$ and the passage times $t(e)$ are \emph{i.i.d.}~random variables.
This setting is well-understood (see the incomplete list of results \cite{kesten1986aspects,kesten1993speed,cox1981some,alexander1993note,alexander2011subgaussian} and the surveys \cite{kesten1982percolation,auffinger201750,smythe2006first}). Another way that randomness can enter the percolation model is through the vertices of the graph instead of its edge weights. 
This setting is known as Euclidean first-passage percolation and typically the vertices are assumed to constitute a Poisson point process $X$ in $\R^d$, which possesses convenient isotropy properties \cite{kingman1992poisson}.
The connectivity of the graph can be modelled in different ways but is typically assumed to follow deterministic rules (once the vertices are given).

In the works \cite{howard1997euclidean,howard2001geodesics} a fully connected graph together with power weighted passage times is considered, i.e., $t(e) = \abs{x-y}^\alpha$ where $e=(x,y)$ represents an edge in the graph and $\alpha\geq 1$ is a parameter.
For $\alpha=1$ long hops are possible and the corresponding graph distance $T(x,y)$ equal the Euclidean one $\abs{x-y}$.
To prevent this trivial behavior and enforce short hops, in almost all results it is assumed that $\alpha>1$.
More recent results and applications of this power weighted Euclidean first-passage percolation model can be found, for instance, in \cite{little2022balancing,hwang2016shortest}.
It is also possible to replace the fully connected graph by a Delaunay triangulation subordinate to the Poisson point process, see, e.g., \cite{serafini1997first,pimentel2011asymptotics,hirsch2015first}.

Most relevant for us will be the setting of a random geometric graph. 
Here, the connectivity relies on some parameter $h>0$ and admissible paths in the definition of the distance $T(x,y)$ cannot have hops of length larger than $h$.
Such models were previously considered but much less is known, as compared to lattice percolation or Euclidean percolation with power weights.
High probability bounds between the graph and Euclidean distance were proved in \cite{friedrich2013diameter,diaz2016relation} and large deviation results for the graph distance and a shape theorem were established in \cite{yao2011large}.
The central difficulty of this model is that the distance function is a random variable with infinite expectation with respect to the realizations of the Poisson point process.
This makes standard techniques from subadditive ergodic theory inapplicable. 
Furthermore, establishing quantitative large deviation bounds for this graph distance is very challenging due to the fact that feasible paths on different scales $h$ cannot be straightforwardly combined into a feasible path. \rev In essence, this means that the stochastic processes, while still subadditive, do not readily admit any type of approximate superadditivity \emph{across length scales}, which is needed to establish convergence rates.\footnote{As we show in this paper, approximate superadditivity does hold when the length scale $h$ is fixed, which is sufficient for the ratio convergence results in this paper, but not for establishing convergence rates for the scaling limit.} This issue does not arise in lattice percolation \cite{kesten1993speed}, power weighted percolation \cite{howard2001geodesics}, or related problems like the longest chain problem \cite{bollobas1992height}, since in these cases the connectivity structure does not involve a length scale $h$, and so approximate superadditivity is readily available. For additional convergence rate results in lattice percolation we also refer to \cite{alexander1993note,alexander2011subgaussian} \nc

Let us mention that there is a history of ideas from percolation theory (e.g., subadditivity and concentration inequalities) finding important applications in the theory of PDEs. Recent results on stochastic homogenization theory for PDEs make use of subadditive quantities \cite{armstrong2017additive,armstrong2016quantitative,armstrong2019quantitative,armstrong2018stochastic}, including homogenization of elliptic PDEs on percolation clusters \cite{armstrong2018elliptic,dario2021quantitative}. Subadditivity and concentration inequalities are also key tools in the convergence of data peeling processes to solutions of continuum PDEs  \cite{calder2020limit,cook2022rates}.
%\todo{Look into the papers of Alexander "A note on some rates of convergence in first-passage percolation"
%(Ann. Appl. Probab. 1993), "Subgaussian rates of convergence of means in directed
%first passage percolation" (ArXiv 2011)}

\textbf{Graph PDEs, finite difference methods, and semi-supervised learning:}
Recent years have seen a surge of interest and results in the field of PDEs and variational problems on graphs.
This is based on the observations that, on one hand, PDEs on graphs generalize finite difference methods for the numerical solution of PDEs and, on the other hand, constitute efficient and mathematically well-understood tools for solving problems in machine learning, including data clustering, semi-supervised learning, and regression problems, to name a few. 

The first observation is easily understood, noting that any grid in $\R^d$ with neighbor relations---for instance, the rectangular regular grid $\scale\Z^d$ where every point $x_0\in\scale\Z^d$ is connected to its $2d$ nearest neighbors $x_0\pm \scale e_i$ for $i=1,\dots,d$ and their connection is weighted by their Euclidean distance $\scale$---is a special case of a weighted graph.
The Laplacian operator of a smooth function, for instance, can be approximated as
\begin{align}\label{eq:laplace_fd}
    \Delta u(x_0) \approx \frac{1}{\eps^2}\sum_{i=1}^d \big(u(x_0+\scale e_i) - 2 u(x_0) + u(x_0-\scale e_i)\big)
\end{align}
It is important to remark that graphs allow for richer models. For instance, if the points $\{x_1,\dots,x_n\}$ are \emph{i.i.d.}~samples from a probability density $\rho$, then the graph Laplacian offers an approximation of a density weighted Laplacian with high probability (see e.g, \cite{calder2018game} and the references therein):
\begin{align}\label{eq:laplace_graph}
    \frac{1}{\rho(x_0)}\operatorname{div}\left(\rho(x_0)^2\grad u(x_0)\right) 
    \approx 
    \frac{1}{n\scale^{d+2}}\sum_{\substack{1\leq i\leq n\\\abs{x_i-x_0}\leq\scale}} \big(u(x_i) - u(x_0)\big).
\end{align}
Furthermore, as opposed to standard finite difference methods, random graphs can possess an increased approximation and convergence behavior due to stochastic homogenization effects.
In the context of the \emph{infinity} Laplace operator, this is a key finding of the present paper.

The convergence analysis of finite difference methods for nonlinear PDEs like the $p$-Laplace and the infinity Laplace equations was revolutionized by the seminal work of Barles and Souganidis \cite{barles1991convergence} on convergence of monotone schemes to viscosity solutions and sparked results like \cite{oberman2005convergent,oberman2013finite,Ober06}.
Furthermore, the dynamic programming principles and mean value formulas gave rise to new finite difference methods for $p$-Laplace equations \cite{del2022convergence,del2022finite}.

There are also close connections between graph PDEs and semi-supervised learning (SSL). In SSL one is typically confronted with a relatively large collection of $n\in\N$ data points $\domain_n$, only few of which carry a label.
The points with labels constitute the small subset $\constr_n\subset\domain_n$ (which can but does not have to depend on $n$).
A prototypical example for this is the field of medical imaging where obtaining data is cheap but obtaining labels is expensive. 
Based on pairwise similarity or proximity of the data points, the whole data set is then turned into a weighted graph structure and one seeks to extend the label information by solving a ``boundary'' value problem on this graph, where the boundary data is given by the labels on the small labeled set. 
The abstract problem consists of finding a function $u_n:\domain_n\to\R$ that solves the graph PDE
\[\left\{
\begin{aligned}
 \mathcal{L}_n u_n(x) &= 0,&&  \text{for all }x\in X_n \setminus \constr_n,\\
u_n(x) &= g(x),&& \text{for all }x\in\constr_n.
\end{aligned}
\right.\]
%\begin{align*}
%    \begin{cases}
%    \mathcal{L}_n u_n(x) = 0,\quad&\text{for all }x\in X_n \setminus \constr_n,\\
%    u_n(x) = g(x),\quad&\text{for all }x\in\constr_n,
%    \end{cases}
%\end{align*}
where $\mathcal{L}_n$ is a suitable differential operator on a graph, e.g., a version of the graph Laplacian \cite{calder2020properly,calder2020rates,zhu03}, the graph $p$-Laplacian for $p\in(1,\infty)$ \cite{GarcSlep16,Slep19,calder2018game,hafiene2019continuum}, the graph infinity Laplacian \cite{calder2019consistency,roith2022continuum,bungert2022uniform}, a Poisson operator \cite{calder2020poisson,calder2020lipschitz}, or an eikonal-type operator \cite{calder2022hamilton,dunbar2022models,fadili2021limits}.

Both in the context of finite difference methods and in graph-based semi-supervised learning two main questions arise:
\begin{enumerate}
    \item Under which conditions on the graph and the discrete operators do solutions converge to solutions of the respective continuum PDE?
    \item What is the rate of convergence?
\end{enumerate}
The answers to these questions, if they exist, typically involve two important parameters: The graph resolution $\delta_n$, which describes how well the graph approximates the continuum domain in the Hausdorff distance, and the graph length scale $\scale_n$, which encodes the maximum distance between neighbors in the graph.
Note that for $n$ \emph{i.i.d.}~samples from a positive distribution $\delta_n\sim\left(\log n / n\right)^\frac{1}{d}$ with high probability whereas for a regular grid $\delta_n \sim \left(1/n\right)^\frac{1}{d}$.
The finite difference approximation of the Laplacian on a regular grid \labelcref{eq:laplace_fd} where $\scale_n\sim\delta_n$ is consistent with the Laplacian, \rev where with consistency we mean that the application of the discrete operator to a smooth function converges to the application of the limiting operator to the same function. \nc 
However, already for the graph Laplacian \labelcref{eq:laplace_graph} \rev on general point clouds \nc or for nonlinear differential operators like the game theoretic $p$-Laplacian or the infinity Laplacian, one has to choose $\scale_n$ significantly larger than $\delta_n$ to ensure that the discrete operators are consistent with the continuum one, e.g.,
$\scale_n \gg \delta_n^\frac{d}{d+2}$ \rev for the Laplacian \cite[Theorem 5]{calder2018game}, $\scale_n \gg \delta_n^\frac{2}{3}$ for the $p$-Laplacian \cite[Theorem 1.1]{del2022finite}, and \cite[Lemma 15, Theorem 17]{calder2019consistency} for the infinity Laplacian. \nc

Note that convergence rates can be proved for solutions of the graph Laplace equation by combing consistency with maximum principles, see, e.g., \cite{calder2020rates}, and spectral convergence rates for eigenfunctions are also available, see \cite{calder2022improved} and the references therein.
Furthermore, in the consistent regime of the infinity Laplacian, rates of convergence were proved for $\scale_n\gg\delta_n^\frac{1}{2}$ and a very restrictive setting in \cite{smart2010infinity} and, recently, for general unstructured grids but very large length scales $\scale_n\sim\delta_n^\frac{1}{4}$ in \cite{li2022convergent}. In \cite{bungert2022uniform} we established convergence rates in a general setting whenever $\scale_n\gg\delta_n$.

As our result in \cite{bungert2022uniform} shows, overcoming the lower bounds imposed by consistency \rev of the operator \nc is clearly possible is some cases. For instance, when working with variational methods like Gamma-convergence, convergence can typically be established in the regime $\scale_n\gg\delta_n$ \cite{GarcSlep16,Slep19,roith2022continuum}, however, proving convergence rates is difficult due to the asymptotic nature of Gamma-convergence.

In our previous work \cite{bungert2022uniform} we proposed an entirely new approach based on ideas from homogenization theory. We defined a new homogenized length scale $\homscale_n$ that is significantly larger then the graph length scale $\scale_n$, i.e., one has $\delta_n\ll\scale_n\ll\homscale_n$.
We showed that solutions of the graph infinity Laplace equation
\begin{align}
\begin{split}
%\mathcal L_n^\infty u(x_i):=
&\max_{1\leq j \leq n}\eta(\abs{x_i-x_j}/\scale_n)(u(x_j)-u(x_i))
\\
&\qquad+\min_{1\leq j \leq n}\eta(\abs{x_i-x_j}/\scale_n)(u(x_j)-u(x_i))
=0,
\end{split}
\end{align}
where $\eta:(0,\infty)\to(0,\infty)$ is a decreasing function satisfying $\supp\eta\subset[0,1]$ and some other mild conditions, give rise to approximate sub- and super-solutions of a non-local homogenized infinity Laplace equation for the operator
\begin{align}
    \Delta_\infty^{\homscale_n} u(x)
    :=
    \frac{1}{\homscale_n^2}\left(\sup_{y\in B(x;\homscale_n)}(u(y)-u(x))
    +
    \inf_{y\in B(x;\homscale_n)}(u(y)-u(x))
    \right).
\end{align}
Loosely speaking the larger length scale $\homscale_n$ can then be used to ensure consistency with the infinity Laplacian $\Delta_\infty u := \langle\grad u,\grad^2\grad u\rangle$ while at the same time allowing $\scale_n$ to arbitrarily close to $\delta_n$ as long as $\scale_n\gg\delta_n$ is satisfied. The rate is then given by the optimal choice of $\homscale_n$ in terms of $\scale_n$ and $\delta_n$. The convergence rates obtained in our previous work \cite{bungert2022uniform} depend on quantities like the ratio $\frac{\delta_n}{\epsilon_n}$, and are degenerate at the connectivity scaling $\epsilon_n \sim \delta_n$. Establishing convergence rates at the length scale $\epsilon_n \sim \delta_n$ is the main focus of this paper.

\textbf{Structure of this paper:}
The rest of the paper is organized as follows: 
In \cref{sec:setup} we explain our precise setup and our main results for Euclidean first-passage percolation \cref{thm:main_result_percolation} and the graph infinity Laplacian \cref{thm:gen_rates}.
We also discuss some open problems and extensions. 
\cref{sec:expectation,sec:concentration,sec:ratio} are devoted to proving the percolation results, by first establishing asymptotics for the expected value of a regularized graph distance (which has finite expectation and coincides with the original distance with high probability), proving concentration of measure for this distance, and establishing quantitative convergence rates for the ratio of two distance functions.
Finally, we apply our findings to get convergence rates for the graph infinity Laplace equation in \cref{sec:application}.

The appendix collects important statements regarding (approximate) sub- and superadditive functions, an abstract concentration statement for martingale difference sequences, some auxiliary estimates, and numerical illustrations.

\section{Setup and main results}
\label{sec:setup}

In this section we introduce the different distance functions on Poisson point processes that we shall use in the course of the paper.

In large parts of this paper we let $X$ be a Poisson point process on $\R^d$ with unit intensity. This means that $X$ is a random at most countable collection of points such that the number of points in $X\cap A$, for a Borel set $A$, is a Poisson random variable with mean $\abs{A}$, which denotes the Lebesgue measure of $A$. That is
\begin{align}\label{eq:poisson_density}
    \P(\#(A\cap X) = k) = \frac{\abs{A}^k}{k!}\exp(-\abs{A}).
\end{align}
The Poisson process has the important property that for any $A\subset \R^d$, the intersection $X\cap A$ is also a Poisson point process with intensity function $1_A$, or rather, a unit intensity Poisson point process on $A$ \cite{kingman1992poisson}. 
This is not true for \emph{i.i.d.}~sequences, restrictions of which to subsets are, in fact, Binomial point processes.

\subsection{Paths and distances}

Given a set of points $P\subset\R^d$, $x,y\in \R^d$, and a length scale $h>0$, we denote the set of paths in $P$, connecting $x,y\in\R^d$ with steps of size less than or equal to $h$, by
\begin{align}\label{eq:paths}
\begin{split}
    % \Pi_{h,P}(x,y) &:= \Big\{p\in P^m \st m\in\N,\, |x-p_1|\leq \tfrac{h}{2},\, |y-p_m|\leq \tfrac{h}{2} \\
    % &\qquad\qquad\qquad\text{ and } |p_i-p_{i+1}|\leq h\;\forall i=\rev 2 \nc,\dots,m-1\Big\}.
    \Pi_{h,P}(x,y) &:= 
    \rev 
    \Big\{p\in P^m \st m\in\N,\, p_1\in\pi_P(x),\,p_m\in\pi_P(y), \\
    &\quad
    \rev 
    |x-p_1|\leq h/2,\,
    |y-p_m|\leq h/2,\,
    \text{ and } |p_i-p_{i+1}|\leq h\;\forall i= 1,\dots,m\Big\},\nc
\end{split}
\end{align}
\rev 
where for $x\in\R^d$ the set $\pi_P(x) := \Set{\hat x\in P\st \abs{x-\hat{x}}=\min_{y\in P}\abs{x-y}}$ is the set of all closest points in $P$. \nc 
For any path $p\in \Pi_{h,P}(x,y)$ with $m$ elements we let
\begin{align}\label{eq:length_of_path}
    L(p) := \sum_{i=1}^{m-1}\abs{p_{i+1}-p_i}
\end{align}
denote the length of the path.
We now define
\begin{align}
    d_{h,P}(x,y) := \inf\Set{L(p)\st p\in \Pi_{h,P}(x,y)},\quad x,y\in P,
\end{align}
to be the length of the shortest path in $\Pi_{h,P}(x,y)$ connecting $x\in\R^d$ and $y\in\R^d$.
Whenever we are referring to the Poisson point process $X$, meaning $P=X$, we obfuscate the dependency on $X$ by using the abbreviations $\Pi_h(x,y)$ and $d_h(x,y)$.

\subsection{Different distance-based random variables}
\label{sec:different_rvs}

Thanks to the spatial homogeneity of the Poisson process it suffices to study the distance between the points $x=0$ and $y=se_1$ where $e_1=(1,0,\dots,0)\in\R^d$ denotes the first unit vector.
This leads to the quantity $d_{h}(0,se_1)$ for $s\geq 0$.
If the length scale $h>0$ is fixed, most distances will be infinite with high probability when $s$ is large.
Therefore, we consider length scales $h\equiv h_s$ which depend on the distance.

Our main object of study is the random variable 
\begin{align}\label{eq:T_s}
    T_s :=  d_{h_s}(0,se_1) = \inf\Set{L(p) \st p\in \Pi_{h_s}(0,se_1)},\quad s\geq 0.
\end{align}
For properly chosen length scales $h_s$, roughly satisfying $\log(s)^\frac{1}{d}\lesssim h_s \ll s$, we show that  $s - h_s\leq T_s \leq C_d\,s$ with high probability.
Here $C_d>0$ is a suitable dimensional constant, to be specified later.
However, with small but positive probability there are no feasible paths and $T_s$ is infinite which makes it meaningless to study its expectation $\Exp{T_s}$ and fluctuations around the expectation.

% For this reason we also introduce the random variable
% \begin{align}\label{eq:T_s'}
%     T_s^\prime := \min(T_s, C_d\,s),\quad s\geq 0,
% \end{align}
% which is always finite and has similar properties like $T_s$.
% In particular, it has finite expectation $\Exp{T_s^\prime}$ which we shall prove to be asymptotically linear, i.e.,
% \begin{align*}
%     \lim_{s\to\infty}\frac{\Exp{T_s^\prime}}{s} \quad\text{exists in }(0,\infty).
% \end{align*}
% For proving concentration of measure, i.e, high probability bounds for the fluctuations around the expected value, $T_s^\prime$ is not well suited.
% This is due to the fact that it does not have feasible paths if $T_s>C_d\,s$ and the minimum in the definition of $T_s^\prime$ kicks in.

Therefore, we construct yet another distance function which always has feasible paths.
For this, let us fix $s>0$ and cover $\R^d$ with closed boxes $\{B_k\}_{k\in\N}$ of side length $\delta_s/C_d$.
Here $\delta_s\sim \log(s)^\frac{1}{d}$ will be specified later and $C_d>0$ is a dimensional constant, sufficiently large such that the maximum distance of two points in two touching boxes is at most~$\delta_s/2$.
A possible choice is
\begin{align}\label{eq:C_d}
    C_d := 2\sqrt{d}.
\end{align}
The probability that all boxes $B_k$ contain at least one point in $X$ is zero and therefore we define an at most countable index set $\mathcal{I}_s$ such that $B_i \cap X = \emptyset$ for all $i\in\mathcal{I}_s$.
For every $i\in\mathcal{I}_s$ we then add a point $x_i\in B_{i}$, for instance the center of the box, to the Poisson process $X$ which leads to the enriched set of points
\begin{align}
    \mathcal{X}_s := X \cup \bigcup_{i\in\mathcal{I}_s}\{x_i\}.
\end{align}
\begin{figure}[!t]
\centering
\includegraphics[width=.7\textwidth]{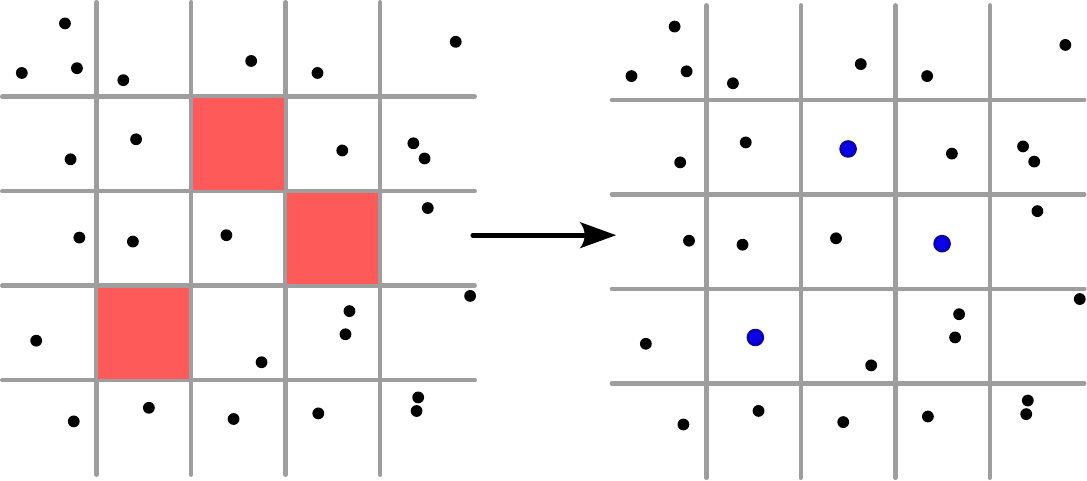}
\caption{The covering of the Poisson process $X$ on the left yields empty boxes in red.
We define the enriched process $\mathcal{X}_s$ on the right by adding the points in blue.}
\label{fig:enriched_Poisson}
\end{figure}
With the notation $\mathcal{X}_s$ we emphasize that this enriched Poisson point process depends on~$s$.
We can hence consider the following distance function on the enriched Poisson process which we define for all scalings $h\geq\delta_s$ (which can but do not have to depend on $s$)
\begin{align}\label{eq:distance_prime}
\dprimeS{h}{\mathcal{X}_s}(x,y),\qquad x,y\in\R^d,
\end{align}
and we define $T_s^\prime$ as
\begin{align}\label{eq:T_s'}
T_s^\prime :=
\dprimeS{h_s}{\mathcal{X}_s}(0,se_1).
\end{align}
Later we shall express $T_s^\prime-\Exp{T_s^\prime}$ as a sum over a \emph{martingale difference sequence} with bounded increments. 
This will allow us to prove concentration of measure for $T_s^\prime$.
% Along the way we will also have to work with the random variable 
% \begin{align}\label{eq:T_s^k}
%     T_s^{(k)} := d_{h_s}(0,se_1)[\mathcal{X}_s\setminus B_k]
% \end{align}
% which is the distance function on the enriched Poisson point process that cannot use the $k$-th box.

Finally, we will synthesis these different results (see \cref{tab:notation} for an overview of the different definitions) by utilizing that with high probability the random variables $T_s$ and $T_s^\prime$ coincide.

\begin{table}[htb]
    \centering
    \begin{tabular}{c|l|l}
         \textbf{Symbol}& \textbf{Meaning} &\textbf{Definition}  \\
         \hline 
         $T_s$ & graph distance on Poisson process & $d_{h_s,X}(0,s e_1)$ \\
         %$T_s^\prime$ & truncated graph distance on Poisson process & $\min(T_s, C_d\,s)$ \\
         $T_s^\prime$ & graph distance on enriched Poisson process & $\dprimeS{h_s}{\mathcal{X}_s}(0,s e_1)$ \\
         --- & graph distance on enriched Poisson process with fixed step size & $\dprimeS{h}{\mathcal{X}_s}(x,y)$
         %$d_{h}(x,y)[\mathcal{X}_s]$ %\\
         %$T_s^{(k)}$ & graph distance on enriched Poisson process w/o $k$-th box & $d_{h_s}(0,s e_1)[\mathcal{X}_s\setminus B_k]$
    \end{tabular}
    \caption{Different random variables used in this work. $X$ denotes a unit intensity Poisson process, $\mathcal{X}_s$ an enrichment with additional points.}
    \label{tab:notation}
\end{table}

\subsection{Constants and symbols}

We will encounter many constants in the paper, most of which are dimensional, i.e., they depend on $d\in\N$. 
Out of all these constants, only $C_d$ (which was already introduced above), $C_d^\prime$ (which shall be introduced in \cref{sec:T_s_T_s'_same}), and $\sigma$ (which will arise as $\sigma=\lim_{s\to\infty}\frac{T_s}{s}\in[1,C_d]$) will keep their meaning throughout the whole paper. 
However, we do not claim that their values are \rev optimal or \nc analytically known.
In many estimates and probabilities other (mostly dimensional) constants will appear and we number them as $C_1,C_2,C_3$, etc.
Note, however, that their values change between the individual lemmas and theorems they appear in and also sometimes change in proofs, which we mention in the latter case.
Since these constants are of no importance to us, we refrain from numbering them continuously, as its sometimes done.
Finally, we sometimes write our inequalities in a more compact form by absorbing all constants into the symbols $\lesssim$, $\gtrsim$, or $\sim$, where, for instance, $f(s)\lesssim g(s)$ means $f(s)\leq C\,g(s)$ and $f(s)\sim g(s)$ means $f(s)=Cg(s)$ for a constant $C>0$.
Finally, we will use the symbol $\ll$ to denote
\begin{align*}
    f(s) \ll g(s) 
    \iff
    \lim_{s\to\infty}\frac{f(s)}{g(s)} = 0.
\end{align*}

\subsection{Main results}

In this section we state our most important results for Euclidean first-passage percolation on a unit intensity Poisson process $X\subset\R^d$ and the application to the graph infinity Laplace equation.

Note that the first theorem is stated in a very compact form and some of the results will be proved in a slightly more general setting.
Also the assumptions will be spelled out in a more quantitative form throughout the paper.
The most important parts of this statement are the concentration of measure and the convergence rates for ratio convergence. Even though convergence rates for $d_{h_s}(0,se_1)$ are not available, the ratio convergence rates are sufficient for our application to Lipschitz learning.

\begin{theorem}[Euclidean first-passage percolation]\label{thm:main_result_percolation}
Let $s>1$ and assume that $s\mapsto h_s$ is non-decreasing and satisfies 
\begin{align*}
    \log(s)^\frac{1}{d}
    \lesssim
    h_s
    \ll 
    s.
\end{align*}
There exist dimensional constants $C_1,C_2>0$, not depending on $s$, such that:
\begin{enumerate}
    \item {\bf (Convergence)} There exists a dimensional constant $\sigma\in[1,C_d]$ (depending on the choice of $s\mapsto h_s$) such that
    \begin{align*}
        \lim_{s\to\infty} \frac{\Exp{\dprimeS{h_s}{\mathcal{X}_s}(0,se_1)}}{s} = \sigma
        \quad
        \text{and}
        \quad
        \lim_{s\to\infty} \frac{d_{h_s}(0,se_1)}{s} = \sigma
        \quad
        \text{almost surely}.
    \end{align*}
    % and, furthermore,
    % \begin{align*}
    %     \lim_{s\to\infty} \frac{d_{h_s}(0,se_1)}{s} = \sigma
    %     % \quad
    %     % \text{and}
    %     % \quad
    %     % \lim_{s\to\infty} \frac{\dprimeS{h_s}{\mathcal{X}_s}(0,se_1)}{s} = \sigma,
    %     \quad
    %     \text{almost surely}.
    % \end{align*}
    % \item {\bf (Near subadditivity)} It holds
    % \begin{align*}
    %     \Exp{T_{s+t}^\prime} \leq \Exp{T_{s}^\prime} + \Exp{T_{t}^\prime} + \frac{C_1}{(s+t)^{k-(d+1)}\log(C_d (s+t)/3)^\frac{1}{d}},
    % \end{align*}
    % where $k>0$ can be increased by multiplying $h_s$ with a constant. 
    
    \item\label{it:conc} {\bf (Concentration)} It holds for all $t\geq h_s$
    \begin{align*}
        \Prob{\abs{\dprimeS{h_s}{\mathcal{X}_s}(0,te_1) - \Exp{\dprimeS{h_s}{\mathcal{X}_s}(0,te_1)}}>\lambda \sqrt{\frac{\log(s)^\frac{2}{d}}{h_s}t}} \leq C_1 \exp(-C_2 \lambda)\qquad\forall \lambda \geq 0.
    \end{align*}
    
    % \item\label{it:monoton} {\bf (Near monotonicity)} 
    % For all $0 \leq s^\prime \leq s$ it holds
    % \begin{align*}
    %     \Exp{\dprimeS{h_s}{\mathcal{X}_s}(0,s^\prime e_1)}
    %     \leq 
    %     \Exp{\dprimeS{h_s}{\mathcal{X}_s}(0,s e_1)} 
    %     +
    %     C_1 h_s 
    %     +
    %     C_2 \sqrt{\frac{\log(s)^\frac{2}{d}}{h_s}s}\ \log(s).
    % \end{align*}
    
    \item\label{it:ratio} {\bf (Ratio convergence)}
    It holds for $s>1$ sufficiently large
    \begin{align*}
    \abs{\frac{\Exp{\dprimeS{h_s}{\mathcal{X}_s}(0,se_1)}}{\Exp{\dprimeS{h_s}{\mathcal{X}_s}(0,2se_1)}} - \frac{1}{2}} \leq 
    C_1\frac{h_s}{s}
    +
    C_2 \sqrt{\frac{\log(s)^\frac{2}{d}}{h_s}}\frac{\log(s)}{\sqrt{s}}.
\end{align*}
\end{enumerate}
\end{theorem}
\begin{proof}
The theorem collects results from \cref{thm:concentration_T_s',thm:convergence,prop:ratio_convergence_exp}.
\end{proof}
\begin{remark}
Since $\dprimeS{h}{\mathcal{X}_s}(x,y)=d_h(x,y)$ with high probability, the concentration of measure statement in \cref{it:conc} of \cref{thm:main_result_percolation} implies concentration of the standard distance function $T_s$ around $\Exp{T_s^\prime}$.
Furthermore, using concentration, \cref{it:ratio} has a corresponding high probability versions for both distances.
\end{remark}

Our second main result concerns convergence rates for solutions to the graph infinity Laplace equation.
For this we let $X_n\subset\closure\domain$ be a Poisson point process with density $n\in\N$ in an open and bounded domain~$\domain\subset\R^d$.
For a bandwidth parameter $\scale>0$ and a function $u : X_n \to \R$ we define the graph infinity Laplacian of $u$~as
\begin{align*}
    \mathcal{L}_\infty^\scale u(x) :=
    \sup_{y\in B(x,\scale)\cap X_n}\frac{u(y)-u(x)}{\abs{y-x}}
    +
    \inf_{y\in B(x,\scale)\cap X_n}\frac{u(y)-u(x)}{\abs{y-x}},\qquad x \in X_n.
\end{align*}
The infinity Laplacian operator of a smooth function $u:\domain\to\R$ is defined as
\begin{align*}
    \Delta_\infty u = \sum_{i,j=1}^d\partial_i u\partial_j u \partial_{ij}^2 u
    =
    \langle\grad u, \grad^2u\grad u\rangle.
\end{align*}
The following theorem states quantitative high probability convergence rates of solutions to the equation $\mathcal{L}_\infty^\scale u_n = 0$ to solutions of $\Delta_\infty u = 0$.
\rev 
Note that the theorem considers the boundary value problem associated with the infinity Laplace operator whereas in our previous work \cite{bungert2022uniform} we considered the setting where function values are prescribed in a very general closed set $\mathcal{O}\subset\closure\Omega$.
While this is much more realistic in the context of semi-supervised learning, the corresponding convergence proof requires precise control of graph distance functions close to the boundary of the domain. 
Achieving this control in the percolation setting is far beyond the scope of this paper since it would essentially require percolation results on Poisson point processes on half spaces together with suitable flattening techniques. 
Therefore, we focus on the setting of a boundary value problem, where boundary values for the discrete equation are prescribed in a tube around the boundary.
This is in line with previous work for the linear Laplacian operator, e.g., \cite{braides2022asymptotic,calder2020rates}.
\nc 
\begin{theorem}[Convergence rates]\label{thm:gen_rates}
Let $\domain\subset\R^d$ be an open and bounded domain.
Let $g:\closure\domain\to\R$ be a Lipschitz function and $u:\domain\to\R$ be the unique viscosity solution of
\[\left\{
\begin{aligned}
\Delta_\infty u &= 0,&& \text{in } \domain,\\
u &= g,&& \text{on }\partial \domain .
\end{aligned}
\right.\]
%\begin{align*}
%\begin{cases}
%    \Delta_\infty u = 0\quad&\text{in }\domain,\\
%    u = g\quad&\text{on }\partial\domain.
%\end{cases}
%\end{align*}
Let $X_n$ be a Poisson point process in $\R^d$ with density $n\in\N$, let $\scale>0$ and $\homscale>0$ satisfy
\begin{align*}
    K \left(\frac{\log n}{n}\right)^\frac{1}{d}
    \leq 
    \scale
    \rev
    \leq 
    \frac{1}{K}
    \homscale,
    \qquad
    0 < \homscale <1,    
\end{align*}
and let
\begin{align*}
    \constr_n := \left\{x\in X_n\cap\closure\domain\st \dist(x,\partial\domain) \leq \scale\right\}
    % K \left(\frac{\log n}{n}\right)^\frac{1}{d}\right\}.
\end{align*}
Let $u_n:X_n\to\R$ be the unique solution of
\[\left\{
\begin{aligned}
\mathcal{L}_\infty^\scale u_n&= 0,&& \text{in }\domain\cap X_n\setminus \constr_n \\
u_n &= g,&& \text{on } \constr_n.
\end{aligned}
\right.\]
%\begin{align*}
%    \begin{cases}
%        \mathcal{L}_\infty^\scale u_n = 0\quad&\text{on }X_n\setminus \constr_n, \\
%        u_n = g&\text{on }\constr_n.
%    \end{cases}
%\end{align*}
There exist dimensional constants $C_1,C_2,C_3,C_4,C_5>0$ such that for $n\in\N$,  \rev for all $\lambda\geq 0$, and for $K\geq 8$ \nc sufficiently large it holds
\begin{align*}
    &\Prob{
    \sup_{x\in X_n}\abs{u(x)-u_n(x)} \lesssim
    \homscale
    +
    \sqrt[3]{
    (\log n + \lambda)
    \left(\frac{\log n}{n}\right)^\frac{1}{d}
    \frac{1}{\sqrt{\homscale^3\scale}}
    +
    \frac{\scale}{\homscale^2}
    }}
    \\
    &\hspace{2in}\geq 
    1 - C_1\exp(-C_2 K^d \log n) - C_3 \exp(-C_4\lambda+C_5\log n).
\end{align*}
\end{theorem}

An important special case of \cref{thm:gen_rates} is the choice of $\scale_n \sim \left(\frac{\log n}{n}\right)^\frac{1}{d}$.
\begin{corollary}\label[corollary]{cor:small_length_scales}
Under the conditions of \cref{thm:gen_rates} and for $\scale = \scale_n = K\left(\frac{\log n}{n}\right)^\frac{1}{d}$ with $K$ sufficiently large it holds \rev for all $\lambda\geq 0$ that\nc
\begin{align*}
    &\Prob{
    \sup_{x\in X_n}\abs{u(x)-u_n(x)} \lesssim
    (\log n + \lambda)^\frac{2}{9}\left(\frac{\log n}{n}\right)^\frac{1}{9d}
    }
    \\
    &\hspace{2in}\geq 
    1 - C_1\exp(-C_2 K^d \log n) - C_3 \exp(-C_4\lambda+C_5\log n). 
\end{align*}
\end{corollary}
\begin{proof}
For this choice of $\scale=\scale_n$ it holds that 
\begin{align*}
    \left(\frac{\log n}{n}\right)^\frac{1}{d}\frac{1}{\sqrt{\homscale^3\scale}}
    \gtrsim
    \frac{\scale}{\homscale^2}
\end{align*}
so we can ignore the second term under the root in \cref{thm:gen_rates}.
Optimizing the resulting error term over $\homscale$ yields the optimal choice of $\homscale_n := (\log n + \rev\lambda\nc)^\frac{2}{9}\left(\frac{\log n}{n}\right)^\frac{1}{9d}$.
For this choice both terms scale in the same way. 
\end{proof}
\begin{remark}\label[remark]{rem:rates}
\cref{cor:small_length_scales} shows that we get a convergence rate of $\left(\tfrac{\log n}{n}\right)^{\frac{1}{9d}}$ (up to the log factor) at the connectivity scale $\scale_n \sim \left(\tfrac{\log n}{n}\right)^{\frac{1}{d}}$. Interestingly, this rate coincides with the best rate \rev achievable using the techniques \nc from our previous paper \cite{bungert2022uniform}, though in that work we had to choose a much larger length scale $\scale_n\sim\left(\tfrac{\log n}{n}\right)^\frac{5}{9d}$ to obtain the rate.
\rev 
In any case, judging from our numerical experiments and simple examples we do not expect these rates to be optimal.
In particular, it would be interesting to understand the degree of suboptimality which our techniques introduce when passing from rates of distance functions (or their ratio) to rates for the infinity Laplace equation.
\nc 
\end{remark}

We can obtain almost sure convergence rates by letting $\lambda$ depend on $n$.
\begin{corollary}
Under the conditions of \cref{cor:small_length_scales} and for $K>0$ sufficiently large it holds 
\begin{align*}
    \limsup_{n\to\infty}\frac{\sup_{x\in X_n}\abs{u(x)-u_n(x)}}{(\log n)^\frac{2}{9}\left(\frac{\log n}{n}\right)^\frac{1}{9d}}
    <\infty
    \qquad
    \text{almost surely}.
\end{align*}
\end{corollary}
\begin{proof}
%\todo{I am surprised that $\lambda=\log n$ works. {\color{red}I am not:)}}
For $\lambda_n = C\log n$ with a large constant $C>0$ and for $K>0$ sufficiently large we can use the Borel--Cantelli lemma to conclude.
\end{proof}

While we have stated our results for Poisson point processes, it is straightforward to de-Poissonize and obtain the same results for \emph{i.i.d.}~sequences. 
\begin{corollary}
Assume the conditions of \cref{thm:gen_rates}, except that $X_n$ is defined instead as an \emph{i.i.d.}~sample of size $n$ uniformly distributed on $\domain$.  There exist dimensional constants $C_1,C_2,C_3,C_4,C_5>0$ such that for $n\in\N$ and $K>0$ sufficiently large it holds
\begin{align*}
    &\Prob{
    \sup_{x\in X_n}\abs{u(x)-u_n(x)} \lesssim
    \homscale
    +
    \sqrt[3]{
    (\log n + \lambda)
    \left(\frac{\log n}{n}\right)^\frac{1}{d}
    \frac{1}{\sqrt{\homscale^3\scale}}
    +
    \frac{\scale}{\homscale^2}
    }}
    \\
    &\hspace{1in}\geq 
    1 - e^{\frac{1}{12}}C_1\exp(-C_2 K^d \log n + \tfrac{1}{2}\log(n)) - C_3 \exp(-C_4\lambda+C_5\log n)
    .
\end{align*}
\end{corollary}
\begin{proof}
Let $\widetilde{X}_n$ be a Poisson point process on $\R^d$ with intensity $\frac{n}{|\domain|}$. Conditioned on $\#( \widetilde{X}_n \cap \domain) = n$, both $X_n$ and $\widetilde{X}_n\cap \domain$ have the same distribution. By conditioning on $\#(\widetilde{X}_n \cap \domain) = n$ and using  \cref{thm:gen_rates} the probability of the event 
\[\sup_{x\in X_n}\abs{u(x)-u_n(x)} \gtrsim
    \homscale
    +
    \sqrt[3]{
    (\log n + \lambda)
    \left(\frac{\log n}{n}\right)^\frac{1}{d}
    \frac{1}{\sqrt{\homscale^3\scale}}
    +
    \frac{\scale}{\homscale^2}
    }\]
is bounded by
\[\Prob{\# (\widetilde{X}_n \cap \domain) = n}^{-1}\left(C_1\exp(-C_2 K^d \log n) + C_3 \exp(-C_4\lambda+C_5\log n)\right).\]
By Stirling's formula we have
\[\Prob{\# (\widetilde{X}_n \cap \domain) = n}^{-1} = \frac{n! e^{n}}{n^n} \leq e^{\frac{1}{12}} \sqrt{n}. \]
Upon adjusting the values of $C_3$ and $C_5$, the proof is complete.
\end{proof}

\subsection{Outlook}
\label{sec:outlook}

There two central directions of future research that originate from this paper, namely further strengthening and generalizing our percolation results, and applying the techniques from this paper to prove convergence rates for other graph PDEs, like for instance the $p$-Laplace equation. 
With respect to the first direction, the ultimate goal would be to prove a strong approximate superadditivity result of the form \labelcref{ineq:superadditivity_conj} which in combination with the concentration of measure from \cref{thm:concentration_T_s'} immediately yields convergence rates for the almost sure convergences $T_s^\prime / s \to \sigma$ and $T_s/s \to \sigma$, as shown in \cite{kesten1993speed}.
Therefore, we formulate the following open problem:
\begin{openproblem*}
Does there exist a function $s\mapsto g(s)$, satisfying $\int_1^\infty g(s)s^{-2}\de s < \infty$, such that%\todo{Look into Cox and Kesten, “On the continuity of the time constant of first-passage
%percolation” (J. Appl. Probab. 1981) and explain the differences: lattice, continuity wrt distribution, lack of rates}
\begin{align*}
    \Exp{T_{2s}^\prime} \geq 2\Exp{T_s^\prime} - g(s),\qquad s>1?
\end{align*}
\end{openproblem*}
This form of strong super-additivity is implied and roughly equivalent to establishing a modulus of continuity of the distance function with respect to the length scale, i.e., for the function
\begin{align*}
    h \mapsto \Exp{\dprimeS{h}{\mathcal{X}_s}(0,se_1)}.
\end{align*}
\rev This problem is related to continuity of the time constant in first passage percolation, which was established for lattice percolation in \cite{cox1981continuity,cox1980time}. However, the notion of continuity in \cite{cox1981continuity,cox1980time} is non-quantitative, and taken with respect to the distribution of the \emph{i.i.d.}~edge weights, whereas in our setting we seek a \emph{quantitative} continuity statement with respect to the \emph{length scale} $h$ that defines the connectivity structure. It seems that different techniques are required here.  \nc

Having this continuity at hand, it would be straightforward to extend the arguments of \cref{sec:application} to \emph{inhomogeneous} Poisson point processes with intensity $n\rho$ where $n\in\N$ and $\rho$ is a probability density with some regularity.
Blowing up around a point shows that the graph distance can be bounded from above and from below with distances $d_{h_i}^\prime(0,se_1)$ on a unit intensity process, albeit with two different but close length scales~$h_1,h_2>0$.

It would be desirable to extend the percolation results to weighted distances of the form
\begin{align}\label{eq:graph_distances}
    d_h(x,y) := \inf\left\lbrace
    \sum_{i=1}^m \frac{h}{\eta(\abs{p_i-p_{i-1}}/h)} \st p \in \Pi_h(x,y)
    \right\rbrace.
\end{align}
For $\eta(t) := \frac{1}{t}$ this reduces to the distance that we considered here but it allows to generate a large class of commonly known graph distances where the weight of an edge $(x,y)$ is given by $h^{-1}\eta(\abs{x-y}/h)$.
Most notably, if $\eta(t)=1$ for $0\leq t\leq 1$ and $\eta(t)=0$ for $t>1$ one obtains the hop counting distance, scaled with $h$. 
The analysis of \labelcref{eq:graph_distances} is complicated by the fact that they do not obey the triangle inequality and, furthermore, are inaccurate if $\abs{x-y}\ll h$.
Still, we expect that our results can be generalized to these distances relatively easily.

The question of whether and how percolation techniques can be applied to other graph PDEs (e.g., the Laplace or $p$-Laplace equations) seems much harder. 
Recent results in two dimensions show that at least Dirichlet energies Gamma-converge for percolation length scales \cite{braides2022asymptotic,caroccia2022compactness}.
Combining quantitative versions of these arguments with the techniques from \cite[Section 5.5]{calder2020calculus} can potentially produce convergence rates.

\section{Convergence in Expectation}
\label{sec:expectation}

In this section we prove that $T_s^\prime$ satisfies 
\begin{align}\label{eq:linear_growth}
    \lim_{s\to\infty}\frac{\Exp{T_s^\prime}}{s} = \sigma \in (0,\infty).
\end{align}
For this we use the subadditivity techniques from \cref{sec:sub_super_additivity}.
It will turn out that $s\mapsto\Exp{T_s^\prime}$ is only nearly subadditive which, however, is enough to establish \labelcref{eq:linear_growth}.
Note that we cannot hope for an analogous statement for $T_s$ since $\Exp{T_s}=\infty$ for all $s>0$.

\subsection{Bounds}

First we prove coarse lower and upper bounds for $T_s$ and $T_s^\prime$ which will be used to prove that, if the limit in \labelcref{eq:linear_growth} exists, then $0<\sigma<\infty$ has to hold.

We start with a trivial lower bound which is true for any distance function, independent of the set of points which is used to construct it.
\begin{lemma}[Lower bound]\label[lemma]{lem:lower_bound}
For any set of points $P\subset\R^d$, $x,y\in\R^d$, and $h>0$ it holds
\begin{align*}
    d_{h,P}(x,y) \geq \abs{x-y} \rev - \dist(x,P) - \dist(y,P) \geq \abs{x-y}-h,\nc
\end{align*}
and, in particular, \rev for all $s\geq 0$\nc
\begin{gather*}
    T_s \geq s \rev- h_s, 
    \qquad\nc
    T_s^\prime \geq s \rev- h_s.\nc
\end{gather*}
\end{lemma}
\begin{proof}
We can assume that $d_{h,P}(x,y)<\infty$ since otherwise the inequality is trivially true.
Let therefore $p\in\Pi_{h,P}(x,y)$ be a path with $m\in\N$ elements in $X$, the length of which realizes $d_{h,P}(x,y)$.
Then it holds 
\begin{align*}
    d_{h,P}(x,y) 
    &\geq 
    \sum_{i=1}^{m-1}\abs{p_{i+1}-p_i} 
    \geq 
    \abs{\sum_{i=1}^{m-1}(p_{i+1}-p_i)}
    = 
    \abs{p_{m}-p_1} 
    \\
    &\geq
    \abs{y-x + p_{m} - y - (p_1 - x)}
    \\
    &\geq 
    \abs{x-y} - \abs{p_{m}-y} - \abs{p_1-x}
    % \\ 
    % &=
    % \abs{x-y} \rev- \abs{\pi_P(y)-y}- \abs{\pi_P(x)-x}
    \\
    &=
    \abs{x-y} \rev- \dist(x,P) - \dist(y,P)
    \\
    &\geq 
    \rev 
    \abs{x-y} - h,
\end{align*}
\rev
using that the existence of a feasible path implies $\dist(x,P),\dist(y,P)\leq h/2$. \nc
The statements for $T_s$ and $T_s^\prime$ follow from their definition as distances on $P:=X$ and $P:=\mathcal{X}_s$, respectively.
\end{proof}

Now we prove a high probability upper bound for the distance function on the Poisson point process which we will apply to $T_s$.

\begin{lemma}[Upper bound 1]\label[lemma]{lem:upper_bound}
\rev
\rev For all $x,y\in\R^d$ and $h>0$ it holds\nc 
\begin{align*}
    \P\left(d_h(x,y) \leq C_d \abs{x-y}\rev+h\nc\right) 
    &\geq 
    \rev 
    \P\left(d_h(x,y) \leq C_d \abs{x-y}+\dist(x,X)+\dist(y,X)\right) 
    \\
    &\geq 1 -  \exp\left(-\left(\frac{h}{C_d}\right)^d + \log\left(\frac{C_d \abs{x-y}}{h}\right)\right),
\end{align*}
and, in particular, \rev for all $s\geq 0$\nc
\begin{align*}
    \P\left(T_s \leq C_d\,s\rev+h_s\nc\right)
    \geq 
    1 -  \exp\left(-\left(\frac{h_s}{C_d}\right)^d + \log\left(\frac{C_d\,s}{h_s}\right)\right),\qquad\forall s>0.
\end{align*}
Here the constant $C_d$ is defined in \labelcref{eq:C_d}.
% \begin{align*}
%     \P\left(\big\{d_h(x,y) \leq C_d \abs{x-y}\}\right) \geq 1 -  \exp\left(-\left(\frac{h}{C_d}\right)^d + \log\left(\frac{C_d \abs{x-y}}{h}\right)\right)
% \end{align*}
% and consequently for any $s\geq 0$ it holds
% \begin{align*}
%     \P\left(\big\{D_{s,h} \leq C_d\,s\}\right) \geq 1 -  \exp\left(-\left(\frac{ h}{C_d }\right)^d + \log\left(\frac{C_d\,s}{h}\right)\right).
% \end{align*}
% for all $\lambda \geq 0$
% \begin{align*}
%     \P\left(\big\{d_h(x,y) \leq \lambda + C_d \abs{x-y}\}\right) \geq 1 -  \exp\left(-\left(\frac{\abs{x-y} h}{\lambda + C_d \abs{x-y}}\right)^d + \log\left(\frac{\lambda + C_d \abs{x-y}}{h}\right)\right)
% \end{align*}
% and consequently for any $s\geq 0$ it holds
% \begin{align*}
%     \P\left(\big\{D_{s,h} \leq \lambda + C_d\,s\}\right) \geq 1 -  \exp\left(-\left(\frac{s h}{\lambda + C_d\,s}\right)^d + \log\left(\frac{\lambda + C_d\,s}{h}\right)\right).
% \end{align*}
\end{lemma}
\begin{remark}
The probability for this upper bound deteriorates for large distances if the step size $h>0$ is fixed.
Therefore, we have to use $h=h_s$ which shall be chosen as $h_s\sim\log(s)^\frac{1}{d}$ later.
\end{remark}
\begin{proof}[Proof of \cref{lem:upper_bound}]
Because of the spatial invariance of the Poisson process, it suffices to proof the statement for $d_h(0,se_1)$.
\begin{figure}[ht]
\centering
\includegraphics[width=\textwidth,trim={0cm 1.6cm 0cm 1.8cm},clip]{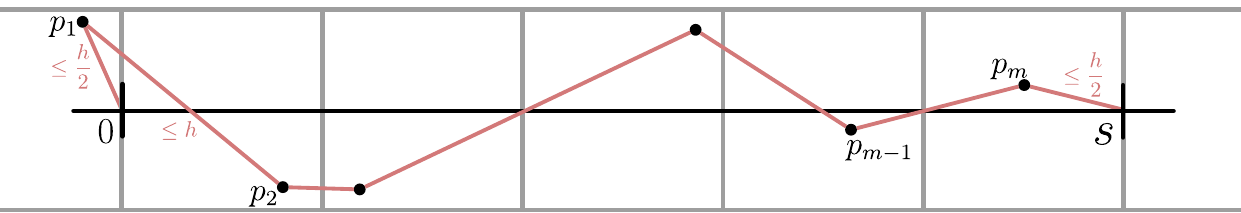}
\caption{Boxes covering the line segment between $0$ and $se_1$.}
\label{fig:boxes_line}
\end{figure}
We cover the line segment connecting $0$ and $s e_1$ by $M\in\N$ boxes $B_i:=\{\frac{2i - Mr}{2M} e_1\}\oplus[-r,r]^{d\rev-1}$, $i=1,\dots,M$, see \cref{fig:boxes_line}.
The side length $r>0$ is given by 
\begin{align}\label{eq:side_of_boxes}
    r = \frac{h}{C_d},    
\end{align}
where $C_d$ given by \labelcref{eq:C_d} assures that the maximal distance of two points in two adjacent boxes $B_i,B_{i+1}$
%, given by $r\sqrt{d+1}$, 
is bounded by $h$, 
and the maximal distance between $0$ and the points in the first box $B_1$ and $se_1$ and the points in the last box $B_M$
%, given by $\frac{r}{2}\sqrt{d+3}$, 
is bounded by $h/2$.
Consequently, the number of boxes is
\begin{align}\label{eq:number_of_boxes}
    % M = 
    % \left\lceil\frac{s}{h}\sqrt{d+3}\right\rceil.
    % M = \left\lfloor \tfrac{\lambda + C_d\,s}{h}\right\rfloor,
    M = \frac{s}{r} = \frac{C_d\,s}{h}.
\end{align}
If each box contains a point from the Poisson cloud $X$, we can construct a valid path $p\in\Pi_h(0,s e_1)$ 
% of length at most \rev$M+2$\nc, 
% leading to $d_{h}(0,se_1)\leq \rev(M+2) h = C_d\,s+2h$.
\rev which satisfies
\begin{align*}
    d_{h}(0,se_1) 
    &\leq 
    \dist(0,X)+\frac{h}{2}+(M-1)h+\frac{h}{2}+\dist(se_1,X)
    \\
    &=
    C_d s + \dist(0,X) + \dist(se_1,X).
\end{align*}
Here, we used the triangle inequality to estimate $\abs{p_1-p_2}\leq\abs{p_1-0}+\abs{0-p_2}\leq\dist(0,X)+h/2$ and similarly for the last term.
\nc 
% leading to $d_{h_s}(0,se_1)\leq M h \leq \lambda + C_d\,s$
Furthermore, the probabilility of this event is
\begin{align*}
    % \Prob{{d_h(0,se_1)\leq C_d\,s\rev+h\nc}}
    % &\geq
    \Prob{{B_i\cap X \neq \emptyset \; \forall i\in\{1,\dots,M\}}}
    =
    1 - \Prob{\bigcup_{i=1}^M\Set{B_i\cap X = \emptyset}}.
    % \P(\{D_{s,h} \leq \lambda + C_d  s\})
    % &=
    % \P\left(\left\{d_h(0,se_1)\leq \lambda + C_d\,s\right\}\right)
    % \\
    % &\geq
    % \P\left(
    % \{B_i\cap X \neq \emptyset \; \forall i\in\{1,\dots,M\}\}
    % \right)
    % =
    % 1 - \P\left(\bigcup_{i=1}^M\{B_i\cap X = \emptyset\}\right).
\end{align*}
Using a union bound, and \labelcref{eq:poisson_density} with $k=0$ we obtain
\begin{align*}
    \P\left(\bigcup_{i=1}^M\{B_i\cap X = \emptyset\}\right) 
    &\leq
    \sum_{i=1}^M \P(B_i\cap X = \emptyset)
    = 
    \sum_{i=1}^M \exp(-\mu(B_i))
    \\
    &= M \exp(-r^d)
    = \frac{s}{r}\exp(-r^d)
    =
    \exp\left(-r^d + \log\left(\frac{s}{r}\right)\right).
\end{align*}
Furthermore, the definition of $r$ in \labelcref{eq:side_of_boxes} implies
% \begin{align*}
%     r = \frac{s}{\left\lfloor\tfrac{\lambda+ C_d\,s}{h}\right\rfloor} \geq \frac{s h}{\lambda + C_d\,s}
% \end{align*}
% and hence
\begin{align*}
    \Prob{\bigcup_{i=1}^M\Set{B_i\cap X = \emptyset}}
    \leq 
    \exp\left(-\left(\frac{h}{C_d}\right)^d + \log\left(\frac{C_d\,s}{h}\right)\right).
    % \P\left(\bigcup_{i=1}^M\{B_i\cap X = \emptyset\}\right)  \leq \exp\left(-\left(\frac{s h}{\lambda + C_d\,s}\right)^d + \log\left(\frac{\lambda + C_d\,s}{h}\right)\right).
\end{align*}
% The statement for $T_s$ follows using that $T_s = d_{h_s}(0,se_1)$.
\end{proof}

\begin{lemma}[Upper bound 2]\label[lemma]{lem:upper_bound_T_s'}
For the constant $C_d$, defined in \labelcref{eq:C_d}, for any $s>1$, and for $h \geq \delta_s$ it holds almost surely
\begin{align*}
    \dprimeS{h}{\mathcal{X}_s}(x,y) \leq C_d \abs{x-y}\rev + h\nc,\qquad
    \forall x,y\in\R^d,
\end{align*}
and in particular for $h_s\geq\delta_s$
\begin{align*}
    T_s^\prime \leq C_d\,s \rev + h_s\nc,\qquad\forall s>0.
\end{align*}
\end{lemma}
\begin{proof}
The proof is the same as the one of \cref{lem:upper_bound} with the only difference being that the path which is constructed there uses the non-empty boxes from the definition of~$\mathcal{X}_s$.
\end{proof}

\begin{remark}[Better upper bounds]\label[remark]{rem:upper_bounds}
It is important to remark that the upper bounds $d_{h}(x,y),\,\dprimeS{h}{\mathcal{X}_s}(x,y)\leq C_d \abs{x-y}\rev+h\nc$ are quite coarse.
\rev Using the more careful strategy from \cite[Lemma 5.5]{bungert2022uniform} one can obtain the (high probability) bounds
\begin{align*}
    d_{h}(x,y),\,\dprimeS{h}{\mathcal{X}_s}(x,y) \leq \left(1 + C\frac{\delta_s}{h}\right)\abs{x-y}\rev+h\nc,
\end{align*}
where $C$ is a dimensional constant.
Since in our regime $h \sim \delta_s$ the constant in front of $\abs{x-y}$ does not converge to $1$ anyway, there is no need for us to use these improved bounds.
\end{remark}

\subsection{The distances coincide with high probability}\label{sec:T_s_T_s'_same}

It turns out that the two distance functions $d_{h}(\cdot,\cdot)$ and $\dprimeS{h}{\mathcal{X}_s}(\cdot,\cdot)$ coincide with high probability.
For this, we first show localization, i.e., that optimal paths for the former distance lie in a sufficiently large ball with high probability.

\begin{lemma}\label[lemma]{lem:path_T_s_in_box}
There exists a dimensional constant $C_d^\prime\geq 1$ such that for $0< h \leq \abs{x-y}/2$ with probability at least $1 -  \exp\left(-\left(\frac{h}{C_d}\right)^d + \log\left(\frac{C_d\,\abs{x-y}}{h}\right)\right)$ any optimal path of $d_h(x,y)$ lies in $B(x,C_d^\prime \abs{x-y})$. 
%\todo{Did we define $[x,y]^d$ for $x,y\in \R^d$? Would it be simpler to use the ball $B(x,C_d'|x-y|)$?
% \red 
% No we didn't define it and we could use the ball instead. This will only change the proof of \cref{lem:T_s-T_s'} a little but it will simplify the proof of the infinity Laplacian consistency.
% Changed it.
% }
\end{lemma}
\begin{proof}
Without loss of generality we assume $x=0$ and $y=s e_1$.
By \cref{lem:upper_bound}, with probability at least $1 -  \exp\left(-\left(\frac{h}{C_d}\right)^d + \log\left(\frac{C_d\,s}{h}\right)\right)$ there exists an optimal path for $d_h(0,se_1)$ and it holds $d_h(0,se_1) \leq C_d\,s\rev+h\leq (C_d+1/2)s$.
Let $p$ be such an optimal path with $m:=\abs{p}$ elements.
If $p$ contained a point $p_i$ outside $B(0,C_d^\prime s)$ its length would satisfy 
\begin{align*}
    L(p) 
    &\geq 
    \abs{p_1 - p_i} + \abs{p_i - p_m}
    \geq 
    2\abs{p_i} - \abs{p_1} - \abs{p_m}
    \geq 
    2\abs{p_i} - \frac{h}{2} - \frac{h}{2} - \abs{s e_1}
    \\
    &\geq 
    2C_d^\prime s - h - s
    =
    (2 C_d^\prime-1) s - h.
    \\
    &=
    (2 C_d^\prime-1) s \left(1 - \frac{h}{s}\right).
\end{align*}
By the assumption $h \leq s/2$ we get that the brackets are larger or equal than $\tfrac{1}{2}$. 
Hence, if we choose $C_d^\prime\geq C_d+\rev 3/2\nc$ we get that
\begin{align*}
    d_h(0,se_1) = L(p)\geq 
    \left(C_d+\rev 1\nc\right) s
\end{align*}
which is a contradiction.
\end{proof}

An analogous statement is satisfied by $\dprimeS{h}{\mathcal{X}_s}(\cdot,\cdot)$, using the upper bound established in \cref{lem:upper_bound_T_s'}.
\begin{lemma}\label[lemma]{lem:path_T_s'_in_box}
Assume that $\delta_s\leq h \leq \abs{x-y}/2$.
Then any optimal path of $\dprimeS{h}{\mathcal{X}_s}(x,y)$ lies in $B(x,C_d^\prime\abs{x-y})$.
\end{lemma}
\begin{proof}
Using \cref{lem:upper_bound_T_s'}, the proof works exactly as the one of the previous lemma.
\end{proof}

Thanks to these two lemmata for any $x,y\in\R^d$ the distance $\dprimeS{h}{\mathcal{X}_s}(x,y)$ in fact only depends \rev on \nc points in a compact set.
Using properties of the Poisson process we can argue that the small boxes $B_k$ from the definition of $\dprimeS{h}{\mathcal{X}_s}(\cdot,\cdot)$ which fall into this compact set all contain a Poisson point with high probability.
This then implies that $d_h(x,y)=\dprimeS{h}{\mathcal{X}_s}(x,y)$ since no point has to be added to $X$.

\begin{lemma}\label[lemma]{lem:T_s-T_s'}
Let $x,y\in\R^d$ and $\delta_s \leq h \leq \abs{x-y}/2$.
Then it holds that 
\begin{align*}
    \Prob{d_h(x,y)=\dprimeS{h}{\mathcal{X}_s}(x,y)} \geq 
    1 - 2\exp\left(-\left(\frac{h}{C_d}\right)^d + d\log\left(\frac{2 C_d C_d^\prime\,\abs{x-y}}{\delta_s}\right)\right)
    .
\end{align*}
\end{lemma}
\begin{proof}
Again it suffices to prove the statement for $x=0$ and $y=se_1$.
Let $E_s$ be the event any optimal path of $d_h(0,se_1)$ lies within $B(0,C_d^\prime s)$.
Then \cref{lem:path_T_s_in_box} shows
\begin{align}\label{eq:prob_E_s}
    \Prob{E_s} \geq 1 - \exp\left(-\left(\frac{h_s}{C_d}\right)^d + \log\left(\frac{C_d\,s}{h}\right)\right)
\end{align}
After possibly enlarging $C_d^\prime$ a little we can assume that the box of side length $2C_d^\prime s$ which contains $B(0,C_d^\prime s)$ coincides with the union of $M\in\N$ boxes $B_k$ which have a side length of $\delta_s/C_d$.
Here $M = \left(\frac{2 C_d C_d^\prime\,s}{\delta_s}\right)^d$.
As in the proof of \cref{lem:upper_bound}, using \labelcref{eq:poisson_density} and a union bound shows that the probability that all of these boxes contain a point from $X$ is at least
\begin{align*}
    1 - M \exp\left(-\left(\frac{h}{C_d}\right)^d\right)
    &= 
    1 - \exp\left(-\left(\frac{h}{C_d}\right)^d + \log\left(\left(\frac{2 C_d C_d^\prime\,s}{\delta_s}\right)^d\right)\right) 
    \\
    &=
    1 - \exp\left(-\left(\frac{h}{C_d}\right)^d + d\log\left(\frac{2 C_d C_d^\prime\,s}{\delta_s}\right)\right).
\end{align*}
We call this event $F_s$ and obtain
\begin{align}\label{eq:prob_F_s}
    \Prob{F_s} \geq 1 - \exp\left(-\left(\frac{h}{C_d}\right)^d + d\log\left(\frac{2 C_d C_d^\prime\,s}{\delta_s}\right)\right).
\end{align}
Since according to \cref{lem:path_T_s'_in_box} it holds $d_h(0,se_1) = \dprimeS{h}{\mathcal{X}_s}(0,se_1)$ if all boxes contain a point from $X$, we obtain $E_s\cap F_s \subset\Set{d_h(0,se_1) = \dprimeS{h}{\mathcal{X}_s}(0,se_1)}$.
Hence, using \labelcref{eq:prob_E_s,eq:prob_F_s} and a union bound we get
\begin{align*}
    \Prob{{d_h(0,se_1) = \dprimeS{h}{\mathcal{X}_s}(0,se_1)}} 
    &\geq 
    \Prob{E_s\cap F_s}
    =
    1-\Prob{E_s^c \cup F_s^c}
    \geq 
    1 - \Prob{E_s^c} - \Prob{F_s^c}
    \\
    &\geq 
    1 - \exp\left(-\left(\frac{h}{C_d}\right)^d + \log\left(\frac{C_d\,s}{h}\right)\right)\\
    &\qquad-\exp\left(-\left(\frac{h}{C_d}\right)^d + d\log\left(\frac{2 C_d C_d^\prime\,s}{\delta_s}\right)\right)
    \\
    &\geq 
    1 - 2\exp\left(-\left(\frac{h}{C_d}\right)^d + d\log\left(\frac{2 C_d C_d^\prime\,s}{\delta_s}\right)\right).
\end{align*}
Here we also used that $d\geq 1$ and $2C_d^\prime/\delta_s \geq 1/h$.
\end{proof}

\subsection{Approximate spatial invariance}

A main benefit of using distance functions over homogeneous Poisson point processes is their invariance with respect to isometric transformations like shifts, rotations, etc., which preserve the Lebesgue measure. 

Using that the distance functions $d_{h,\mathcal{X}_s}(x,y)$ and $d_h(x,y)$ coincide with high probability, we can show that this invariance of $d_h(x,y)$ translates to $d_{h,\mathcal{X}_s}(x,y)$.
In fact, we will need the slightly more general statement of the following lemma.

\begin{lemma}\label[lemma]{lem:invariance}
%\todo{please check \red Looks good}
Let $M\in\N$ and $x_i,y_i\in\R^d$ be points satisfying $\abs{x_i-y_i}=\Delta$ for all $i=1,\dots,M$ and $\delta_s \leq h \leq \Delta/2$.
Let furthermore $\Phi : \R^d \to \R^d$ be an isometry. 
Then it holds
\begin{align*}
    &\abs{\Exp{\min_{1\leq i \leq M}d_{h,\mathcal{X}_s}(\Phi(x_i),\Phi(y_i))} - \Exp{\min_{1\leq i \leq M}d_{h,\mathcal{X}_s}(x_i,y_i)}} 
    \\
    &\qquad
    \leq \exp\left(-\left(\frac{h}{C_d}\right)^d
    +
    (d+1)
    \log\left(\rev\max\left\lbrace 2 C_d C_d^\prime,4C_d+2\right\rbrace\nc\Delta\right)
    +
    \log M
    -d\log(\delta_s)\right).
\end{align*}
\end{lemma}
\begin{proof}
Using that $\Phi$ is an isometry and applying \cref{lem:T_s-T_s'} and a union bound, yields that the event\rev
\begin{align*}
    A := \Set{d_{h,\mathcal{X}_s}(\Phi(x_i),\Phi(y_i))=d_{h}(\Phi(x_i),\Phi(y_i))\text{ and }
    d_{h,\mathcal{X}_s}(x_i,y_i)=d_{h}(x_i,y_i)
    \;\forall
    i=1,\dots,n
    }
    \nc
\end{align*} 
satisfies\nc
\begin{align*}
    \rev
    \Prob{A}
    \geq \nc
    1 - 2M\exp\left(-\left(\frac{h}{C_d}\right)^d + d\log\left(\frac{2 C_d C_d^\prime\,\Delta}{\delta_s}\right)\right).
\end{align*}
Hence, we can use the invariance of the distance function on the Poisson process $X$ to get
\begin{align*}
    \Exp{\min_{1\leq i \leq M}d_{h}(\Phi(x_i),\Phi(y_i))\vert A }
    =\Exp{\min_{1\leq i \leq M}d_{h}(x_i,y_i)\vert A}.
\end{align*}
Therefore, we obtain
\begin{align*}
    &\phantom{{}={}}
    \Exp{\min_{1\leq i \leq M}d_{h,\mathcal{X}_s}(\Phi(x_i),\Phi(y_i))}
    \\
    &=
    \Exp{\min_{1\leq i \leq M}d_{h,\mathcal{X}_s}(\Phi(x_i),\Phi(y_i)) \gm A}
    \Prob{A}
    +
    \Exp{\min_{1\leq i \leq M}d_{h,\mathcal{X}_s}(\Phi(x_i),\Phi(y_i))\gm A^c}
    \Prob{A^c}
    \\
    &=
    \Exp{\min_{1\leq i \leq M}d_{h}(\Phi(x_i),\Phi(y_i)) \gm A}
    \Prob{A}
    +
    \Exp{\min_{1\leq i \leq M}d_{h,\mathcal{X}_s}(\Phi(x_i),\Phi(y_i)) \gm A^c}
    \Prob{A^c}
    \\
    &=
    \Exp{\min_{1\leq i \leq M}d_{h}(x_i,y_i)\gm A}
    \Prob{A}
    +
    \Exp{\min_{1\leq i \leq M}d_{h,\mathcal{X}_s}(\Phi(x_i),\Phi(y_i)) \gm A^c}
    \Prob{A^c}
    \\
    &=
    \Exp{\min_{1\leq i \leq M}d_{h,\mathcal{X}_s}(x_i,y_i)\gm A}
    \Prob{A}
    +
    \Exp{\min_{1\leq i \leq M}d_{h,\mathcal{X}_s}(\Phi(x_i),\Phi(y_i)) \gm A^c}
    \Prob{A^c}
    \\
    &=
    \Exp{\min_{1\leq i \leq M}d_{h,\mathcal{X}_s}(x_i,y_i)}
    \\&\qquad+
    \left(
    \Exp{\min_{1\leq i \leq M}d_{h,\mathcal{X}_s}(\Phi(x_i),\Phi(y_i)) \gm A^c}
    -
    \Exp{\min_{1\leq i \leq M}d_{h,\mathcal{X}_s}(x_i,y_i) \gm A^c}
    \right)
    \Prob{A^c}.
\end{align*}
Reordering and trivially estimating $d_{h,\mathcal{X}_s}(\cdot,\cdot)$ using \cref{lem:upper_bound_T_s'} we obtain
\begin{align*}
    &\phantom{{}={}}
    \abs{\Exp{\min_{1\leq i \leq M}d_{h,\mathcal{X}_s}(\Phi(x_i),\Phi(y_i))}
    -
    \Exp{\min_{1\leq i \leq M}d_{h,\mathcal{X}_s}(x_i,y_i)}
    }
    \\
    &\leq 
    4 M \left(C_d\Delta+h\right)
    \exp\left(-\left(\frac{h}{C_d}\right)^d+d\log\left(\frac{2 C_d C_d^\prime\Delta}{\delta_s}\right)\right)
    \\
    &= 
    \exp\left(-\left(\frac{h}{C_d}\right)^d+d\log\left(\frac{2 C_d C_d^\prime\Delta}{\delta_s}\right)+
    \log(4 M (C_d\Delta+h))\right)
    \\
    &\leq 
    \exp\left(-\left(\frac{h}{C_d}\right)^d+d\log\left({2 C_d C_d^\prime\Delta}\right) + \log\left(\left(4C_d+2\right)\Delta\right) + \log M
    -d\log(\delta_s)\right)
    \\
    &\leq 
    \rev
    \exp\left(-\left(\frac{h}{C_d}\right)^d
    +
    (d+1)
    \log\left(\max\left\lbrace 2 C_d C_d^\prime,4C_d+2\right\rbrace\Delta\right)
    +
    \log M
    -d\log(\delta_s)\right)
\end{align*}
where we used the isometry of $\Phi$ \rev and that $h\leq\Delta/2$.
\end{proof}

\subsection{Near subadditivity}

% \todo[inline]{Is it really true 
% that $d_{h_t}(se_1,(s+t)e_1)$ and $d_{h_t}(0,te_1)$ have the same distribution? 
% Cf. \cite[p.313]{kesten1993speed}: ``Because $v(m, \xi)$ may differ from $v(n + m,\xi) - v(n,\xi)$, it is not necessarily true that an $a_{n+m}(\xi)$ has the same distribution as $a_m(\xi)$.''
% \blue I don't see any issue for us. I believe Kesten's issue is that he makes a specific, but artibrary, choice of the closest lattice point, and that choice may not be spatially invariant. Maybe he could have worked harder to make it invariant, but he has a simple fix for the problem anyway, so he probably didn't care. Since we are instead allowing that point to be determined in the shortest path optimization it is clearly spatially invariant at the level of the probability law. I don't see this being an issue for us.}

In this section we prove \rev an approximate \nc triangle inequality for the distance $d_h(\cdot,\cdot)$ which will then allow us to prove an \emph{approximate subadditivity} property for $\Exp{T_s^\prime}$.
Old results, which go back to Erd\H os and others, will then allow us to deduce \labelcref{eq:linear_growth}.

First, we prove a general approximate triangle inequality for the distance function on an arbitrary set of points and different values of the length scale $h$.
\begin{lemma}\label[lemma]{lem:triangle_ineq}
Let $P\subset\R^d$ be a set of points.
Let $h_1,h_2>0$ and $h_3 \geq \max(h_1,h_2)$.
Then it holds
\begin{align}
    d_{h_3,P}(x,y) \leq d_{h_1,P}(x,z) + d_{h_2,P}(z,y) \rev + h_3\nc\quad\forall x,y,z\in \R^d.
\end{align}
\end{lemma}
\begin{proof}
The statement follows from the simple observation that if $p\in\Pi_{h_1,P}(x,z)$ and $q\in\Pi_{h_2,P}(z,y)$ are optimal paths which realize $d_{h_1,P}(x,z)$ and $d_{h_2,P}(z,y)$ then $r := (p,q)$ is a path in $\Pi_{h_3,P}(x,y)$.
\rev To see this, note that the last point in $p$ has a distance of at most $h_1/2$ to $z$ and the first point in $q$ has a distance of at most $h_2/2$ to $z$. 
Using the triangle inequality the distance between the those two points is at most $h_1/2+h_2/2\leq h_3$ \nc and consequently
\begin{align*}
    d_{h_3,P}(x,y) \leq L(r) \rev \leq \nc L(p) + L(q) \rev + h_3 \nc =  d_{h_1,P}(x,z) + d_{h_2,P}(z,y) \rev + h_3\nc.
\end{align*}
\end{proof}

A straightforward consequence of \cref{lem:triangle_ineq} would be that $s\mapsto\Exp{T_s}$ is \rev near \nc subadditive which by means of \cref{lem:debruijn} implies that the limit $\lim_{s\to\infty}\frac{\Exp{T_s}}{s}$ exists.
However, since there is a small but non-zero probability that $T_s=d_{h_s}(0,se_1)=\infty$, the expected value $\Exp{T_s}$ and this limit is infinite.
Therefore, we investigate $T_s^\prime$ defined in \labelcref{eq:T_s'}.

% Note that we have the triangle inequality
% \begin{align}
%     d_{h,\mathcal{X}_s}(x,y) 
%     \leq
%     d_{h,\mathcal{X}_s}(x,z)
%     +
%     d_{h,\mathcal{X}_s}(z,y),\qquad\forall x,y,z\in\R^d,
% \end{align}
% whereas for $T_s^\prime$ one does \emph{not} have a triangle inequality since the point clouds in its definition are different:
% \begin{align}
%     T^\prime_{s+t}
%     =
%     d_{h_s,\mathcal X_{s+t}}(0,(s+t)e_1)
%     \not\leq
%     d_{h_s,\mathcal X_s}(0,s e_1)
%     +
%     d_{h_t,\mathcal X_t}(s e_1, (s+t) e_1).
% \end{align}
From \cref{lem:T_s-T_s'} we know that $T_s^\prime = T_s$ with high probability and, furthermore, $T_s^\prime$ is always finite and satisfies $T_s^\prime \leq T_s$.
We introduce the error term $E_s := T_s - T_s^\prime \geq 0$.
For estimating it we now specify the choice of $\delta_s$, the width of the boxes in the definition of $T_s^\prime$ in \labelcref{eq:T_s'}.
We shall choose it in such a way that the error $E_s$ is zero with high probability as $s\to\infty$.
\begin{assumption}\label{ass:delta}
For a constant $k>0$ and for $\rev C_d^\dprime:=\max\left\lbrace 2 C_d C_d^\prime,4C_d+2\right\rbrace$ we choose
\begin{align*}
    \delta_s = C_d(k\log(C_d^\dprime\,s))^\frac{1}{d}.
\end{align*}
\end{assumption}
At this point we also fix the assumptions on the step size $h_s$:
\begin{assumption}\label{ass:h}
Let $s\mapsto h_s$ be non-increasing and satisfy
\begin{align*}
    \delta_s \leq h_s \ll s.
\end{align*}
\end{assumption}
For these assumptions on $\delta_s$ and $h_s$ (note that we are mainly interested in the case $h_s=\delta_s$) one can simplify the following term, which appears in a lot of probabilities:
\begin{align}    \label{ineq:probability_scaling}
    \exp\left(-\left(\frac{h_s}{C_d}\right)^d + d\log\left(\frac{2 C_d C_d^\prime\,s}{\delta_s}\right)\right)
    \leq 
    \frac{1}{\delta_s^d}
    \left(\frac{1}{2 C_d C_d^\prime s}\right)^{k-d}
\end{align}
and similarly for the error term in \cref{lem:invariance} with $M=1$ and $s\geq\Delta/2$ we have
\begin{align}\label{ineq:error_invariance}
\begin{split}
    % \exp\left(-\left(\frac{h}{C_d}\right)^d
    % +d\log\left({2 C_d C_d^\prime\abs{x-y}}\right)
    % + \log\left(\left(4C_d+2\right)\abs{x-y}\right)
    % -d\log(\delta_s)\right)
    \exp\left(-\left(\frac{h}{C_d}\right)^d
    +
    (d+1)
    \log\left(\rev C_d^\dprime\nc\Delta\right)
    -d\log(\delta_s)\right)
    \leq 
    \frac{2^k}{\delta_s^d}\left(\frac{1}{\rev C_d^\dprime \nc \Delta}\right)^{k-(d+1)},
\end{split}    
\end{align}
which is dominating \labelcref{ineq:probability_scaling}.
Using \labelcref{ineq:probability_scaling}, the statement of \cref{lem:T_s-T_s'} can be reformulated as follows
\begin{align}\label{eq:probability_large_error}
     \Prob{{E_s > 0}}
     \leq
    %  \frac{(C_d\,s)^{1-k}}{C_d(k\log(C_d\,s))^\frac{1}{d}}.
    \frac{2}{\delta_s^d}\left(\frac{1}{2 C_d C_d^\prime\,s}\right)^{k-d}.
\end{align}
Utilizing that the error $E_s$ is zero with high probability and that we have the \rev approximate \nc triangle inequality from \cref{lem:triangle_ineq} we can show that $\Exp{T_s^\prime}$ is nearly subadditive.
\begin{proposition}\label[proposition]{prop:subadditivity}
% \todo[inline]{I think for our new definition of $T_s^\prime$ this has no error{\red I think there is no error in subadditivity for $T_s$, if you didn't have the infinite expectation to deal with. I thought the issue here is that the enriched point clouds are different for $s$, $t$ and $s+t$, so you only have subadditivity on the event where there is no enriching in any of these 3 cases, and the error terms come from controlling the probability of this event. Maybe I am missing something, what has changed in the new definition of  $T_s'$? }}
Under \cref{ass:delta,ass:h} \rev and for $k\geq d+1$ \nc there exists a constant $C=C(d)>0$ such that for all $s>0$ sufficiently large and all $s\leq t \leq 2s$ it holds
\begin{align*}
    \Exp{T_{s+t}^\prime} \leq \Exp{T_{s}^\prime} + \Exp{T_{t}^\prime} \rev + C\,h_{s+t}.\nc %+ \frac{C}{(s+t)^{k-(d+1)}\log(C_d (s+t)/3)}.
\end{align*}
\end{proposition}
\begin{proof}
We define the translation $\Phi_s : \R^d\to\R^d$, $x\mapsto x - s e_1$ and note that it is a probability measure preserving transformation of $\R^d$.
We define the event
\begin{align}\label{eq:set_A}
A:= \Set{T_{s+t}^\prime=T_{s+t}} \cap \Set{T_s^\prime=T_s} \cap \Set{T_t^\prime\circ\Phi_s=T_t\circ\Phi_s},
\end{align}
abbreviating $T_t \circ \Phi_s := d_{h_t}(s e_1, (s+t)e_1)$ and analogously $T_t^\prime \circ \Phi_s := \dprimeS{h_t}{\mathcal{X}_t}(s e_1, (s+t)e_1)$.

Using the conditional expectation we obtain the following formula of total probability
\begin{align}\label{eq:total_prob}
    \Exp{T_{s+t}^\prime} = \Exp{T_{s+t}^\prime\vert A}\Prob{A} + \Exp{T_{s+t}^\prime\vert A^c}\Prob{A^c}.
\end{align}
By definition of the $T^\prime$ random variable and the event $A$, and using the approximate triangle inequality from \cref{lem:triangle_ineq}, we have
\begin{align}\nonumber
    \Exp{T_{s+t}^\prime\vert A} 
    &=
    \Exp{T_{s+t}\vert A}
    =
    \Exp{d_{h_{s+t}}(0,(s+t)e_1)\vert A}    
    \\
    \nonumber
    &\leq 
    \Exp{d_{h_{s}}(0,s e_1)\vert A}
    +
    \Exp{d_{h_{t}}(s e_1,(s+t) e_1)\vert A}
    \rev 
    + h_{s+t}
    \nonumber 
    \\
    &=
    \Exp{T_s \vert A}
    +
    \Exp{T_t \circ \Phi_s\vert A}
    \rev 
    + h_{s+t}
    \nonumber
    \\
    \label{eq:conditional_exp}
    &=
    \Exp{T_{s}^\prime\vert A}
    +
    \Exp{T_t^\prime \circ \Phi_s\vert A} \rev 
    + h_{s+t} \nc.
\end{align}
Combining \labelcref{eq:total_prob,eq:conditional_exp}, \rev estimating $\Prob{A}\leq 1$\nc, and using also \cref{lem:lower_bound,lem:upper_bound_T_s'} and the almost translation invariance of $\Exp{T_s^\prime}$ from \cref{lem:invariance} in the case $M=1$ together with \labelcref{ineq:error_invariance} we get
\begin{align}
    \nonumber
    \Exp{T_{s+t}^\prime}
    &\leq 
    \Big(
    \Exp{T_{s}^\prime\vert A}
    +
    \Exp{T_t^\prime \circ \Phi_s\vert A}
    \rev 
    + h_{s+t} 
    \nc
    \Big)
    \Prob{A}
    +
    \Exp{T_{s+t}^\prime\vert A^c}\Prob{A^c}
    \\
    \nonumber
    &=
    \Exp{T_s^\prime} + \Exp{T_t^\prime\rev\circ \Phi_s\nc}
    \rev + h_{s+t}\nc
    \\
    \nonumber
    &\qquad
    + \left(\Exp{T_{s+t}^\prime\vert A^c} - \Exp{T_{s}^\prime\vert A^c}-\Exp{T_t^\prime \circ \Phi_s\vert A^c}\right)\Prob{A^c}
    \\
    \nonumber
    &\leq
    \Exp{T_s^\prime} + \Exp{T_t^\prime\circ \Phi_s}
    \rev + h_{s+t} \nc
    + \left(C_d(s+t) \rev + h_s\nc - (s-h_s)-(t-h_t)\right)\Prob{A^c}.
    \\
    \nonumber
    &\leq
    \Exp{T_s^\prime} + \Exp{T_t^\prime} 
    \rev + h_{s+t} \nc
    +\frac{2^k}{\delta_t^d}\left(\frac{1}{\rev C_d^\dprime \nc t}\right)^{k-(d+1)}
    \\\label{ineq:appr_subadd_proof}
    &\qquad
    +
    \Big[(C_d-1)(s+t) + \rev 2\nc h_s + \rev 2h_t\nc\Big]\Prob{A^c}.
\end{align}
Using \cref{lem:T_s-T_s',eq:probability_large_error,ass:delta} and the fact that $0\leq s\leq t\leq 2s$ we obtain that
\begin{align*}
    \Prob{A^c}
    &\leq 
    \Prob{{E_{s+t} > 0}}
    +
    \Prob{{E_{s} > 0}}
    +
    \Prob{{E_t > 0}}
    \\
    &\leq 
    C \left(\frac{(s+t)^{d-k}}{\log(C_d (s+t))}
    +
    \frac{s^{d-k}}{\log(C_d\,s)}
    +
    \frac{t^{d-k}}{\log(C_d t)}
    \right)
    \\
    &\leq 
    C\frac{(s+t)^{d-k}}{\log(C_d (s+t)/3)},
\end{align*}
where the constant $C$ is dimensional and changes its value.
Plugging this estimate into \labelcref{ineq:appr_subadd_proof} and using \cref{ass:h}, we obtain that
\begin{align*}
    \Exp{T_{s+t}^\prime} 
    &\leq 
    \Exp{T_s^\prime} + \Exp{T_t^\prime} \rev + h_{s+t} \nc + C\frac{(s+t)^{d+1-k}}{\log(C_d (s+t)/3)},
\end{align*}
where $C$ again changed its value.
\rev For $k\geq d+1$ and using \cref{ass:delta} we can absorb the second error term into the first one.
Changing $C$ again \nc concludes the proof.
\end{proof}

\subsection{Convergence}

Utilizing the bounds and the near subadditivity we obtain the following result:
\begin{proposition}\label[proposition]{prop:convergence_exp_T_s'}
% Let $\delta_s$ be chosen as in \cref{ass:delta} with $k>d$ and the $h_s$ satisfy \cref{ass:h}.
\rev 
Assume that $\delta_s$ satisfies \cref{ass:delta} with $k\geq d+1$ and $h_s$ satisfies \cref{ass:h} with the additional requirement that for $s$ sufficiently large it holds $h_s \leq C s^\alpha$ for some constant $C>0$ and some $\alpha\in(0,1)$. \nc
Then the limit
\begin{align}\label{eq:limit_distance_mod}
    \sigma := \lim_{s\to\infty}\frac{\Exp{T_s^\prime}}{s}
\end{align}
exists and satisfies $\sigma\in [1,C_d]$.
\end{proposition}
\begin{proof}
\rev For $\alpha\in(0,1)$ the function $g(z) := C z^\alpha$ \nc%$\frac{C}{z^{k-(d+1)}\log(C_d z/3)}$ 
satisfies $\int_{z_0}^\infty g(z)z^{-2}<\infty$ for $z_0>0$.
Hence \cref{prop:subadditivity,lem:debruijn} imply that $\sigma := \lim_{s\to\infty}\frac{\Exp{T_s^\prime}}{s}$ exists.
By \cref{lem:lower_bound,lem:upper_bound_T_s'} the random variable $T_s^\prime$ satisfies the deterministic bounds
\begin{align*}
   s - h_s \leq T_s^\prime \leq C_d\,s\rev +h_s.
\end{align*}
Taking the expectation, dividing by $s$, using $h_s\ll s$ and $C_d\geq 1$ shows that $\sigma\in[1,C_d]$.
\end{proof}
\begin{remark}[The constant $\sigma$]
As already pointed out in \cref{rem:upper_bounds}, the constant $\sigma$ can be brought arbitrarily close to $1$ by multiplying $h_s$ with a large constant.
Note that it is not required to demand a quicker growth than $h_s \sim \log(s)^\frac{1}{d}$.
\end{remark}

\section{Concentration of measure and almost sure convergence}
\label{sec:concentration}

In this section we will prove concentration of measure for $T_s^\prime$ around its expectation. 
This will have two important consequences: First, it will allow us to show that
\begin{align*}
    \lim_{s\to\infty} \frac{T_s}{s}
    =
    \sigma
    \qquad
    \text{almost surely},
\end{align*}
where $\sigma$ is the constant from \labelcref{eq:linear_growth} and \cref{prop:convergence_exp_T_s'}.
Second, we will use concentration to prove an approximate superadditivity property which will be the key to convergence rates.

\subsection{Concentration of measure}

In this section we prove concentration of measure for $T_s^\prime$.
In fact, we will prove a slightly more general statement, namely concentration of measure for the distance function $\dprimeS{h_s}{\mathcal{X}_s}(x,y)$.
Since $T_s^\prime = \dprimeS{h_s}{\mathcal{X}_s}(0,se_1)$ concentration for $T_s^\prime$ will be a special case.
 
\begin{theorem}[Concentration of measure for $\dprimeS{h}{\mathcal{X}_s}(\cdot,\cdot)$]\label{thm:concentration_T_s'}
There exist dimensional constants $C_1,C_2>0$ such that for all $s>0$ and all $x,y\in\R^d$ with $\abs{x-y}\geq h \geq \delta_s$ it holds
\begin{align*}
    \Prob{\abs{\dprimeS{h}{\mathcal{X}_s}(x,y) - \Exp{{\dprimeS{h}{\mathcal{X}_s}(x,y)}}}>\lambda \sqrt{\frac{\delta_s^2}{h}\abs{x-y}}} \leq C_1 \exp(-C_2 \lambda)\qquad\forall \lambda \geq 0.
\end{align*}
\end{theorem}
\begin{remark}
We will be mostly interested in the regime where $h = h_s = \delta_s = C_d(k\log(2C_d C_d^\prime\,s))^\frac{1}{d}$, in which case we get \rev for $\abs{x-y} \geq h_s$:\nc
\begin{align*}
\Prob{\abs{\dprimeS{h_s}{\mathcal{X}_s}(x,y) - \Exp{\dprimeS{h_s}{\mathcal{X}_s}(x,y)}}>\lambda \sqrt{h_s \abs{x-y}}} \leq C_1 \exp(-C_2 \lambda)\quad\forall \lambda \geq 0,
\end{align*}
This means that with high probability the fluctuations of $\dprimeS{h_s}{\mathcal{X}_s}(x,y)$ around its expectation are of order $\sqrt{\rev\abs{x-y}\nc}$ modulo a $\log$ factor.
However, the general result from \cref{thm:concentration_T_s'} can also be applied to larger length scales $h_s \gg \delta_s$ in which case the fluctuations \rev are smaller\nc. %mitigated.\todo{explain}
\end{remark}
\begin{proof}
The proof relies on an application of the abstract martingale estimate from \cref{lem:abstract_concentration} in the appendix.
We follow the proofs of \cite[Theorem 1]{kesten1993speed} or \cite[Theorem 2.1]{howard2001geodesics}.
For a more compact notation we use the following abbreviation:
\begin{align*}
    T := \dprimeS{h}{\mathcal{X}_s}(x,y).
\end{align*}

\textbf{Step 1}:
We define a filtration $\mathbb{F}:=\{\F_k\}_{k\in\N_0}$ of the probability space by setting $\F_0:=\{\emptyset,\Omega\}$, $\F_k := \F(B_1\cup\dots\cup B_k)$ for $k\geq 1$.
By $\F(A)$ we refer to the $\sigma$-subfield of $\F:=\F(\R^d)$ which is generated by events of the form $\{X \cap A \neq \emptyset\}$ for Borel sets $A\subset\R^d$.
We also define the martingale $M_k:=\Exp{T \vert \F_k} - \Exp{T}$ with $M_0=0$ and we define $\Delta_k := M_k - M_{k-1} = \Exp{T \vert \F_k} - \Exp{T \vert \F_{k-1}}$.

\textbf{Step 2}:
We want to compute a constant $c>0$ for which $\abs{\Delta_k}\leq c$.

Let us define the random variable
\begin{align*}
    T^{(k)}:= d_{h, \mathcal{X}_s \setminus B_k}(x,y)
\end{align*}
as the graph distance on the enriched Poisson point process without the $k$-th box.
Using that $\Exp{T^{(k)}\vert\F_{k-1}}=\Exp{T^{(k)}\vert\F_{k}}$ and trivially $T\leq T^{(k)}$ we have
\begin{align*}
    \Delta_k 
    &= 
    \Exp{T \vert\F_k} - \Exp{T \g\F_{k-1}}
    =
    \Exp{T  - T^{(k)}\gm \F_k}
    +
    \Exp{T^{(k)}-T\gm \F_{k-1}}
    \\
    &\leq
    \Exp{T^{(k)}-T\gm \F_{k-1}}
\end{align*}
and analogously $-\Delta_k\leq\Exp{T^{(k)}-T\vert\F_{k}}$.
This implies
\begin{align*}
    \abs{\Delta_k}
    \leq 
    \Exp{T^{(k)}-T\gm\F_{k-1}}
    \vee 
    \Exp{T^{(k)}-T\gm\F_{k}}
\end{align*}
For bounding $T^{(k)}-T$ we argue as follows: 
By definition we know that there exist feasible paths of finite length for $T$ and hence also optimal paths.
Let $F_k$ be the event that the optimal path $p=(p_1,\dots,p_m)$ for $T$ contains at least one point $p_i \in B_k$ for $i\in\{1,\dots,m\}$.
On $F_k^c$ it obviously holds $T^{(k)}=T$.
On $F_k$ we can---using \rev that \nc all boxes contain a Poisson point---construct an alternative path around the no-go box $B_k$ which is at most $C_d^\tprime \delta_s$ longer than $T$ and is feasible for $T^{(k)}$.
Here $C_d^\tprime$ is a suitable dimensional constant, depending on $C_d$.
In either case we have $T^{(k)} - T \leq C_d^\tprime \delta_s$ and therefore also
\begin{align}\label{ineq:Delta_bound}
    \abs{\Delta_k} \leq C_d^\tprime \delta_s =: c.
\end{align}

\textbf{Step 3}:

In this step we want to find a sequence of positive $\F$-measurable random variables $\{U_k\}_{k\in\N}$ such that $\Exp{\Delta_k^2\g\F_{k-1}} \leq \Exp{U_k\g\F_{k-1}}$.

For any two random variables $X,Y$ with $Y$ measurable with respect to a $\sigma$-field $\mathcal{G}$ the projection identity \rev (see, e.g., \cite[Prop. 1.26]{weinan2021applied})\nc
\begin{align*}
    \Exp{(X-\Exp{X\vert\mathcal{G}})^2\vert\mathcal{G}} \leq \Exp{(X-Y)^2\vert\mathcal{G}}
\end{align*}
holds.
We shall use this with $X:=\Exp{T\vert\F_k}$, $Y:=\Exp{T^{(k)}\vert\F_{k-1}}=\Exp{T^{(k)}\vert\F_{k}}$, and $\mathcal{G}:=\F_{k-1}$.
Note that we have the tower property $\Exp{X \vert \mathcal{G}} = \Exp{\Exp{T \vert \F_{k}} \vert \F_{k-1}} = \Exp{T \vert \F_{k-1}} $.
Hence, we can compute
\begin{align*}
    \Exp{\Delta_k^2\gm\F_{k-1}}
    &=
    \Exp{\left(\Exp{T \gm \F_k} - \Exp{T \gm \F_{k-1}}\right)^2 \gm\F_{k-1}}
    \\
    &=
    \Exp{\left(X - \Exp{X \gm \mathcal{G}}\right)^2 \gm\F_{k-1}}
    \\
    &\leq 
    \Exp{(X - Y)^2 \gm\F_{k-1}}
    \\
    &=
    \Exp{\left(\Exp{T \gm \F_k} - \Exp{T^{(k)}\gm\F_{k}}\right)^2 \gm \F_{k-1}}
    \\
    &=
    \Exp{\Exp{T  - T^{(k)}\gm\F_{k}}^2 \gm\F_{k-1}}
    \\
    &\leq
    \Exp{\Exp{\left(T  - T^{(k)}\right)^2\gm\F_{k}} \gm\F_{k-1}}
    \\
    &=
    \Exp{\left(T  - T^{(k)}\right)^2\gm\F_{k-1}}
    \\
    &=
    \Exp{U_k\gm\F_{k-1}},
\end{align*}
using Jensen's inequality and again the tower property for the last two steps and defining the $\F$-measurable random variable $U_k:=(T-T^{(k)})^2$.

\textbf{Step 4}:

Here we want to find a constant $\lambda_0 \geq \frac{c^2}{4e}$ such that for all $K\in\N$
\begin{align}\label{ineq:as_bound_distance}
    S_K := \sum_{k=1}^K U_k \leq \lambda_0\quad\text{almost surely}.
\end{align}
For this we first note that $U_k$ equals zero whenever there exists an optimal path for $T$ which does not use a point in the box $B_k$.
We now fix an optimal path with the smallest number of elements, denote it by $p$, and abbreviate its number of points by $\abs{p}$.
We define the index set $\mathcal{K}:=\{k\in\{1,\dots,K\} \st \exists i \in \{1,\dots,\abs{p}\},\;p_i \in B_k\}$ and get
\begin{align*}
    S_K = \sum_{k\in\mathcal{K}} (T^{(k)}-T)^2.
\end{align*}
Using that $T^{(k)}-T\leq C_d^\tprime \delta_s$ and that the cardinality of $\mathcal{K}$ is at most $\abs{p}$, we obtain the bound 
\begin{align}\label{ineq:bound_S_1}
    S_K \leq (C_d^\tprime \delta_s)^2 \abs{p}.
\end{align}
Our next goal is to upper-bound $\abs{p}$ by a constant times $T / h$ for which we basically want to argue that most of the hops in the path $p$ have a length of order $h$.

Let us abbreviate $m := \abs{p}$.
Our first claim is that for every $\gamma>0$ and $j\in\{1,\dots,m-2\}$ we have:
\begin{align}\label{eq:hops}
    \abs{p_{j+1}-p_j}\leq \gamma
    \implies
    \abs{p_{j+2}-p_{j+1}}>h - \gamma.
\end{align}
If this were not the case, then it would hold
\begin{align*}
    \abs{p_{j+2} - p_j} \leq \abs{p_{j+1}-p_j} + \abs{p_{j+2} - p_{j+1}} \leq \gamma + h - \gamma = h.
\end{align*}
Hence, the path $q:=(p_1,\dots,p_{j},p_{j+2},\dots,p_m)$ would be feasible for $T$, and would satisfy $L(q)\leq L(p)$ as well as $\abs{q}=\abs{p}-1$.
Since $p$ is optimal and shortest, this is a contradiction.

With this at hand we define the index sets
\begin{align*}
    I_\gamma &:= \{j \in \{1,\dots,m-2\} \st \abs{p_{j+1}-p_j}\leq \gamma\},\\
    I_\gamma^\prime &:= \{j+1 \st j \in I_{\gamma}\}, \\
    I_\gamma^\dprime &:= \{1,\dots,m-1\}\setminus(I_\gamma \cup I_\gamma^\prime).
\end{align*}
Note that for $0<\gamma<h/2$, the implication \labelcref{eq:hops} shows that $I_\gamma$ and $I_\gamma^\prime$ are disjoint and their cardinalities coincide.
We abbreviate the latter by $k\in\{0,\dots,m-2\}$ and note that the cardinality of $I_\gamma^\dprime$ equals $m-1 - 2k$.
Using this we can estimate $T$ from below as follows:
\begin{align*}
    T 
    &=
    \sum_{j=1}^{m-1}\abs{p_{j+1}-p_j}
    \\
    &=
    \sum_{j\in I_\gamma}\underbrace{\abs{p_{j+1}-p_j}}_{\geq 0}
    +
    \sum_{j\in I_\gamma^\prime}\underbrace{\abs{p_{j+1}-p_j}}_{\geq h - \gamma}
    +
    \sum_{j\in I_\gamma^\dprime}\underbrace{\abs{p_{j+1}-p_j}}_{>\gamma}
    \\
    &\geq
    k (h-\gamma) + (m-1-2k)\gamma
    \\
    &=
    k(h - 3\gamma) + (m-1)\gamma.
\end{align*}
Choosing $\gamma = h / 6$ the first term is non-negative and we obtain the estimate
\begin{align*}
    m \leq 6\frac{T}{h} + 1.
\end{align*}
Plugging this into our previous bound \labelcref{ineq:bound_S_1} for $S$, we obtain 
\begin{align}\label{ineq:bound_S_2}
    S_K \leq (C_d^\tprime \delta_s)^2\left(6 \frac{T}{h}+1\right)
    =6(C_d^\tprime\delta_s)^2 \frac{T}{h} + (C_d^\tprime \delta_s)^2.
\end{align}
Utilizing that $\delta_s \leq h \leq \abs{x-y}$ and, according to \cref{lem:upper_bound_T_s'}, also $T \leq C_d\abs{x-y}\rev+h\nc$, we get for a suitable constant $C_d^\qprime>0$ that
\begin{align}\label{ineq:bound_S_3}
    S_K \leq C_d^\qprime \frac{\delta_s^2}{h} \abs{x-y}.
\end{align}
Using that $h \leq \abs{x-y}$ we can possibly enlarge $C_d^\qprime$ a little such that 
\begin{align}\label{eq:lambda_0_distances}
    \lambda_0 :=  C_d^\qprime \frac{\delta_s^2}{h} \abs{x-y}
\end{align}
satisfies $\lambda_0 \geq \frac{c^2}{4e}$ where $c = C_d^\tprime \delta_s$ was defined in \labelcref{ineq:Delta_bound}.
Hence, we have established the almost sure bound~\labelcref{ineq:as_bound_distance}.

\textbf{Step 5}

We have check all assumptions for \cref{lem:abstract_concentration} which lets us conclude
\begin{align*}
    \Prob{T - \Exp{T}>\eps} \leq C\exp\left(-\frac{1}{2\sqrt{e \lambda_0}}\eps\right)\qquad\forall \eps\geq 0.
\end{align*}
Plugging in $\lambda_0 =C_d^\qprime \frac{\delta_s^2}{h} \abs{x-y}$ yields
\begin{align*}
    \Prob{T - \Exp{T}>\eps} &\leq C\exp\left(-\frac{1}{2\sqrt{e C_d^\qprime \frac{\delta_s^2}{h}\abs{x-y}} }\eps\right)
    \\&=
    C\exp\left(-\frac{\sqrt{h}}{2\sqrt{e C_d^\qprime \abs{x-y}} \delta_s}\eps\right)
    \qquad\forall \eps\geq 0.
\end{align*}
For the choice $\eps = \lambda \sqrt{\frac{\delta_s^2}{h}\abs{x-y}}$ we get
\begin{align*}
    \Prob{T - \Exp{T}>\lambda \sqrt{\frac{\delta_s^2}{h}}\sqrt{\abs{x-y}}} \leq C\exp\left(-C_2 \lambda\right)\qquad\forall \lambda\geq 0.
\end{align*}
Repeating the proof verbatim for $\Exp{T}-T$ we finally obtain
\begin{align*}
    \Prob{\abs{T - \Exp{T}}>\lambda \sqrt{\frac{\delta_s^2}{h}}\sqrt{\abs{x-y}}} \leq C_1\exp\left(-C_2 \lambda\right)\qquad\forall \lambda\geq 0,
\end{align*}
where $C_1 := 2C$.
\end{proof}

The fact that with high probability $T_s$ and $T_s^\prime$ coincide allows us to deduce concentration of $T_s$ around $\Exp{T_s^\prime}$ (remember that the expectation of $T_s$ is infinite \rev which \nc is why concentration around it is irrelevant). 

\begin{corollary}[Concentration of measure for $T_s$]\label[corollary]{cor:concentration_of_T_s}
Under the assumptions of \cref{thm:concentration_T_s',lem:T_s-T_s'} it holds for $s>0$ sufficiently large 
\begin{align*}
    \Prob{{\abs{T_s - \Exp{T_s^\prime}} > \lambda \sqrt{\frac{\delta_s^2}{h_s}s}}}
    &\leq 
    2\exp\left(-\left(\frac{h_s}{C_d}\right)^d + d\log\left(\frac{2 C_d C_d^\prime\,s}{\delta_s}\right)\right)
    \\&\quad+
    C_1\exp(-C_2 \lambda),
    \quad
    \forall \lambda\geq 0.
\end{align*}
\end{corollary}
\begin{proof}
Utilizing that $T_s^\prime=\dprimeS{h_s}{\mathcal{X}_s}(0,se_1)$ and using \cref{thm:concentration_T_s',lem:T_s-T_s'}, for all $\lambda\geq 0$ it holds
\begin{align*}
    &\phantom{{}\leq{}}\Prob{{\abs{T_s - \Exp{T_s^\prime}} > \rev\lambda\nc \sqrt{\frac{\delta_s^2}{h_s}}\sqrt{s}}}\\
    &\leq 
    \Prob{{\abs{T_s - T_s^\prime} > \lambda\sqrt{h_s s}}}
    +
    \Prob{{\abs{T_s^\prime - \Exp{T_s^\prime}}>\lambda \frac{\delta_s}{\sqrt{h_s}}\sqrt{s}}}
    \\
    &\leq 
    \Prob{{\abs{T_s - T_s^\prime} > 0}}
    +
    \Prob{{\abs{T_s^\prime - \Exp{T_s^\prime}}>\lambda \frac{\delta_s}{\sqrt{h_s}}\sqrt{s}}}
    \\
    &\leq 
    2\exp\left(-\left(\frac{h_s}{C_d}\right)^d + d\log\left(\frac{2 C_d C_d^\prime\,s}{\delta_s}\right)\right)
    +
    C_1\exp(-C_2 \lambda).
\end{align*}
\end{proof}

% For the scaling assumption $C_d(k\log(2 C C_d\,s))^\frac{1}{d}\leq  h_s \ll s$ we get
% \begin{align*}
%     \Prob{{\abs{T_s - \Exp{T_s^\prime}} > x\sqrt{h_s s}}}
%     \lesssim
%     \frac{1}{s^{k-d}}
%     \frac{1}{\log(C s))}
%     +
%     \exp(-C_2 x),\qquad\forall 0\leq x \leq C_3 h_s s.
% \end{align*}
% \red 
% Choosing $x = \eps\sqrt{\frac{s}{h_s}}$ gives
% \begin{align*}
%     \Prob{{\abs{T_s - \Exp{T_s^\prime}} > s\eps}} \leq \frac{1}{s^{k-d}\log(C s)} + \exp\left(-C_2\eps\sqrt{\frac{s}{h_s}}\right)
% \end{align*}
% and Borel--Cantelli should give a.s. convergence of $\frac{T_s}{s}$ to the limit of $\frac{\Exp{T_s^\prime}}{s}$. 
% \todo[inline]{is this possible although $T_s$ has infinite expectation?}
% \nc 

\subsection{Almost sure convergence}

Combining concentration of measure with the convergence in expectation from \cref{prop:convergence_exp_T_s'}, we can now prove almost sure convergence of $T_s^\prime/s$ and even of $T_s/s$.
Note that Kingman's subadditive ergodic theorem is not applicable in this case since the random variables $T_s$ have infinite expectations, hence, are not in $L^1$.
An additional difficulty arises from $T_s$ and $T_s^\prime$ being stochastic processes with a continuous variable $s\in(0,\infty)$.
We prove all statements for a subsequence of integers and to use Lipschitz regularity to extend to the real line.

\begin{theorem}\label{thm:convergence}
Assume that $\delta_s$ satisfies \cref{ass:delta} with $k>d+1$ and $h_s$ satisfies \cref{ass:h} with the additional requirement that for $s$ sufficiently large it holds $h_s \leq C s^\alpha$ for some constant $C>0$ and some $\alpha\in(0,1)$.
Then it holds
\begin{align*}
    % \lim_{s\to\infty}\frac{T_s^\prime}{s} = \sigma
    % \quad
    % \text{and}
    % \quad
    \lim_{s\to\infty}\frac{T_s}{s} = \sigma
    \quad
    \text{almost surely,}
\end{align*}
where $\sigma$ denotes the constant from \cref{prop:convergence_exp_T_s'}.
% \todo[inline]{Check the new proof. The triangle inequality is not needed for $\lim T_s/s=\sigma$. I think the proof of this is fine. I'm not able to fix the issue with the triangle inequality in a straightforward way for $T_s'$. One option is to remove this part of the result, and just focus on $T_s$. Please check the details.}
\end{theorem}
\begin{remark}
As outlined in \cref{rem:upper_bounds} one can make sure that $\sigma$ is arbitrarily close (but not equal) to $1$ by multiplying $h_s$ with a large constant.
\end{remark}
\begin{proof}
Let $\eps>0$  be arbitrary and choose $\lambda = \eps\sqrt{s}\sqrt{\frac{h_s}{\delta_s^2}}$.
Then \cref{thm:concentration_T_s'} implies
\begin{align*}
    \Prob{{\abs{T_s^\prime - \Exp{T_s^\prime}} > s\eps}} 
    \leq 
    C_1\exp\left(-C_2\eps\sqrt{s}\sqrt{\frac{h_s}{\delta_s^2}}\right).
\end{align*}
Let now $s:= n$ where $n\in\N$ is a natural number.
By assumption we have $\delta_{n}\leq h_{n} \leq C n^\alpha$ which implies that
\begin{align*}
    \exp\left(-C_2\eps\sqrt{n}\sqrt{\frac{h_{n}}{\delta_{n}^2}}\right)
    \leq 
    \exp\left(-C_2\eps\frac{\sqrt{n}}{\sqrt{\delta_{n}}}\right)
    =
    \exp\left(-C_2\eps n^\frac{1-\alpha}{2}\right).
\end{align*}
Now we use that for all $m\in\N$ and $x>0$ it holds
\begin{align*}
    \exp(-x) \leq \frac{m!}{x^m}
\end{align*}
to obtain that 
\begin{align*}
    \exp\left(-C_2\eps\sqrt{n}\sqrt{\frac{h_{n}}{\delta_{n}^2}}\right)
    \leq 
    \frac{m!}{\left(C_2\eps (n)^\frac{1-\alpha}{2}\right)^m}.
\end{align*}
If we choose $m>\frac{2}{1-\alpha}$ we obtain 
\begin{align*}
    \sum_{n=1}^\infty
    \exp\left(-C_2\eps\sqrt{n}\sqrt{\frac{h_{n}}{\delta_{n}^2}}\right)
    \leq 
    \frac{m!}{(C_2\eps)^m}
    \frac{1}{t^\frac{m(1-\alpha)}{2}}
    \sum_{n=1}^\infty
    \frac{1}{n^\frac{m(1-\alpha)}{2}}<\infty.
\end{align*}
Hence, the Borel--Cantelli lemma allows us to conclude that 
\begin{align*}
    \Prob{\limsup_{n\to\infty}\Set{\abs{T_{n}^\prime-\Exp{T_{n}^\prime}}>n\eps}} = 0.
\end{align*}
Since $\eps>0$ was arbitrary, we obtain that $\abs{\frac{T_{n}^\prime}{n}-\frac{\Exp{T_{n}^\prime}}{n}}\to 0$ almost surely as $n\to\infty$.
Together with \cref{prop:convergence_exp_T_s'} this implies
\begin{align}\label{eq:limit_T_s'_integers}
    \lim_{n\to\infty}\frac{T_{n}^\prime}{n} = \sigma\quad\text{almost surely}.
\end{align}
We claim that we also have 
\begin{align}\label{eq:limit_T_s_integers}
    \lim_{n\to\infty}\frac{T_{n}}{n} = \sigma\quad\text{almost surely}.
\end{align}
To see this, let $\eps>0$  be arbitrary and choose $\lambda = \eps\sqrt{s}\sqrt{\frac{h_s}{\delta_s^2}}$.
Then \cref{cor:concentration_of_T_s,ass:delta,ass:h,ineq:probability_scaling} imply
\begin{align*}
    \Prob{{\abs{T_s - \Exp{T_s^\prime}} > s\eps}} 
    \leq 
    \frac{2}{\delta_s^d}\left(\frac{1}{2 C_d C_d^\prime\,s}\right)^{k-d}
    + C_1\exp\left(-C_2\eps\sqrt{s}\sqrt{\frac{h_s}{\delta_s^2}}\right).
\end{align*}
Let again $s:= n$ where $n\in\N$ is a natural number.
Using that $\delta_{n} \geq 1$ for $n$ sufficiently large and that $k>d+1$ we get that
\begin{align*}
    \sum_{n=1}^\infty \frac{2}{\delta_{n}^d}\left(\frac{1}{2 C_d C_d^\prime\,n}\right)^{k-d} < \infty.
\end{align*}
Furthermore, by assumption we have $\delta_{n}\leq h_{n} \leq C n^\alpha$ which implies that
\begin{align*}
    \exp\left(-C_2\eps\sqrt{n}\sqrt{\frac{h_{n}}{\delta_{n}^2}}\right)
    \leq 
    \exp\left(-C_2\eps\frac{\sqrt{n}}{\sqrt{\delta_{n}}}\right)
    =
    \exp\left(-C_2\eps n^\frac{1-\alpha}{2}\right).
\end{align*}
Now we use that for all $m\in\N$ and $x>0$ it holds
\begin{align*}
    \exp(-x) \leq \frac{m!}{x^m}
\end{align*}
to obtain that 
\begin{align*}
    \exp\left(-C_2\eps\sqrt{n}\sqrt{\frac{h_{n}}{\delta_{n}^2}}\right)
    \leq 
    \frac{m!}{\left(C_2\eps (n)^\frac{1-\alpha}{2}\right)^m}.
\end{align*}
If we choose $m>\frac{2}{1-\alpha}$ we obtain 
\begin{align*}
    \sum_{n=1}^\infty
    \exp\left(-C_2\eps\sqrt{n}\sqrt{\frac{h_{n}}{\delta_{n}^2}}\right)
    \leq 
    \frac{m!}{(C_2\eps)^m}
    \frac{1}{t^\frac{m(1-\alpha)}{2}}
    \sum_{n=1}^\infty
    \frac{1}{n^\frac{m(1-\alpha)}{2}}<\infty.
\end{align*}
Hence, the Borel--Cantelli lemma allows us to conclude that 
\begin{align*}
    \Prob{\limsup_{n\to\infty}\Set{\abs{T_{n}-\Exp{T_{n}^\prime}}>n\eps}} = 0.
\end{align*}
Since $\eps>0$ was arbitrary, we obtain that $\abs{\frac{T_{n}}{n}-\frac{\Exp{T_{n}^\prime}}{n}}\to 0$ almost surely as $n\to\infty$.
Together with \cref{prop:convergence_exp_T_s'} this establishes the claim, proving \labelcref{eq:limit_T_s_integers}.

We now extend the limits to hold for real-valued $s\to \infty$. We first show that 
\begin{align}\label{eq:limit_T_s}
    \lim_{s\to\infty}\frac{T_{s}}{s} = \sigma\quad\text{almost surely}.
\end{align}
To see this, we let $B(x,t):=\{y\in\R^d\st\abs{x-y}\leq t\}$ denote the closed ball around $x\in\R^d$ with radius $t>0$ and let $A_n$ denote the event that 
\[X\cap B\left((n+1)e_1,\frac{h_n}{4}\right) \neq \emptyset.\]
By the law of the Poisson point process and the choice of scaling of $h_n$ we have
\begin{equation}\label{eq:AnBC}
\sum_{n=1}^\infty \P(A_n^c) = \sum_{n=1}^\infty \exp\left(-\omega_d\frac{h_n^d}{4^d}\right) < \infty,
\end{equation}
where $\omega_d$ is the volume of the $d$-dimensional unit ball.
%, by Borel-Cantelli we have eventually almost surely that
%\begin{equation}\label{eq:nonempty}
%X\cap B\left((n+1)e_1,\frac{h_n}{4}\right) \neq \emptyset.
%\end{equation}
When $A_n$ occurs, let us denote by $x_n$ any point in the intersection of $X$ with $B\left((n+1)e_1,\frac{h_n}{4}\right)$. We also assume $n$ is large enough so that $\frac{h_n}{4} \geq 1$. 

We now claim that whenever $A_n$ occurs and $T_n$ is finite, we have%\todo{I replaced $3/2$ by $1$ since in our definition of the distance we don't count the first and last `virtual' hop.}
\begin{equation}\label{eq:Tsineq}
T_{n+1} - h_n \leq T_s \leq T_n + h_n \ \ \text{for all }n\leq s \leq n+1.
\end{equation}
To see this, note that any optimal path for $T_n$ must terminate at a point $x$ within distance $\frac{h_n}{2}$ of $ne_1$. Since $A_n$ occurs we can add the point $x_n$ to this path to obtain a feasible path for $T_s$. Indeed, we simply note that $h_n \leq h_s$ and compute%\todo{fixed a little mistake here}
\[|x - x_n| \leq |x - ne_1| + |ne_1 - (n+1)e_1| + \abs{x_n - (n+1)e_1} \leq \frac{h_n}{2} + 1 + \frac{h_n}{4} \leq h_n \leq h_s, \]
and
\[|x_n  - se_1| \leq |x_n - (n+1)e_1| + |(n+1)e_1 - se_1| \leq \frac{h_n}{4} + 1 \leq \frac{h_n}{2} \leq \frac{h_s}{2}.\]
Note that we used that $\frac{h_n}{4}\geq 1$ in both inequalities.
It follows that $T_s$ is finite and 
\[T_s \leq T_n + h_n.\]
To prove the other inequality, we follow a similar argument, taking a path that is optimal for $T_s$, which must terminate at a point $y$ that is within distance $\frac{h_n}{2}$ of $se_1$, and concatenating the point $x_n$ to obtain a feasible path for $T_{n+1}$. This yields the inequality
\[T_{n+1} \leq T_s + h_n,\]
which establishes the claim. The proof of \labelcref{eq:limit_T_s} is completed by dividing by $s$ in \labelcref{eq:Tsineq}, recalling \labelcref{eq:AnBC} and applying Borel--Cantelli.
\nc 
\end{proof}

\section{Near superadditivity and ratio convergence}
\label{sec:ratio}

In this section we prove a type of approximate superadditivity of the distance function with the aim of proving convergence rates.
Ideally, we would like to show that for a slowly increasing function $s\mapsto g(s)$
\begin{align}\label{ineq:superadditivity_conj}
    \Exp{T^\prime_{2s}} \geq 2 \Exp{T^\prime_s} - g(s)
\end{align}
holds true. 
Together with the near subadditivity from \cref{prop:subadditivity}, the convergence from \cref{prop:convergence_exp_T_s'}, and \cref{lem:debruijn} in the appendix, this would directly imply quantitative convergence rates for $\frac{\Exp{T_s^\prime}}{s}$ to the constant $\sigma$. 
The concentration statement from \cref{thm:concentration_T_s'} would then yield almost sure rates for $\frac{T_s^\prime}{s}$ to $\sigma$.

Although we think that \labelcref{ineq:superadditivity_conj} might be true, a proof of this is very difficult since the distance functions in the definition of $T_{2s}^\prime$ and $T_s^\prime$ utilize the different length scales $h_{2s}$ and $h_s$. 
Consequently, a path which realizes $T_{2s}^\prime$ is typically not feasible for the distance in $T_s^\prime$ which makes a construction of a suboptimal path for this distance such that \labelcref{ineq:superadditivity_conj} is satisfied hard.
This is a specific problem of our sparse graph setting and can be avoided using a fully connected graph as, e.g., in \cite{howard2001geodesics}.

Therefore, we shall not work with the random variables $T_s^\prime$ or $T_s$ in the following but rather work with a fixed length scale $h$ and the distance function $\dprimeS{h}{\mathcal{X}_s}(\cdot,\cdot)$ defined in \labelcref{eq:distance_prime}.
That is, we aim to prove near superadditivity of the form
\begin{align}\label{ineq:superadditivity_reality}
    \Exp{\dprimeS{h}{\mathcal{X}_s}(0,2se_1)} \geq 2 \Exp{\dprimeS{h}{\mathcal{X}_s}(0,se_1)} - g(s),
\end{align}
where we emphasize that both distance functions on the left and on the right have the same length scale $h$.

\subsection{Near superadditivity}

We start by proving the following proposition which asserts near superadditivity of the form \labelcref{ineq:superadditivity_reality}. The argument closely follows the proof given in \cite[Lemma 4.1]{howard2001geodesics}.
\begin{proposition}[Near superadditivity]\label[proposition]{prop:superadditivity}
Let $\delta_s$ satisfy assumption \cref{ass:delta} with $k>d+1$.
There exist dimensional constants $C_1,C_2>0$ such that for all $s>1$ sufficiently large with $\delta_s \leq h \leq s$ we have that
\begin{align*}
\Exp{\dprimeS{h}{\mathcal{X}_s}(0,2se_1)} 
    \geq 
    2\Exp{\dprimeS{h}{\mathcal{X}_s}(0,se_1)}
    - C_1 h
    -
    C_2 \sqrt{\frac{\delta_s^2}{h}s}\ \log(s).
\end{align*}
% \begin{align*}
%     \Exp{\dprimeS{h}{\mathcal{X}_s}(0,(s+t)e_1)} 
%     \geq 
%     \Exp{\dprimeS{h}{\mathcal{X}_s}(0,se_1)}
%     +
%     \Exp{\dprimeS{h}{\mathcal{X}_s}(0,te_1)}
%     - C_1 h
%     -
%     C_2 \sqrt{\frac{\delta_s^2}{h}s}\log(s)
%     \qquad h \in \{h_s, h_{s+t}\}.
% \end{align*}
% \todo[inline]{What is the right range of $h$ here?}
\end{proposition}
\begin{remark}
In the case that $h=h_s=\delta_s$ we can subsume the error terms into one and for some dimensional constant $C>0$ we have
\begin{align*}
    % \Exp{\dprimeS{h_s}{\mathcal{X}_s}(0,(s+t)e_1)} 
    % \geq 
    % \Exp{\dprimeS{h_s}{\mathcal{X}_s}(0,se_1)}
    % +
    % \Exp{\dprimeS{h_s}{\mathcal{X}_s}(0,te_1)}
    % - C \sqrt{s}\log(s)^{1+\frac{1}{2d}}.
    \Exp{\dprimeS{h_s}{\mathcal{X}_s}(0,2 s e_1)} 
    \geq 
    2\Exp{\dprimeS{h_s}{\mathcal{X}_s}(0,se_1)}
    - C \rev\log(s)^\frac{1+2d}{2d}\nc\sqrt{s}.
\end{align*}
\end{remark}

\begin{proof}
Let $p_1,\ldots, p_m$ be a path realizing the length $\dprimeS{h}{\mathcal{X}_s}(0,2se_1)$. 
We now consider the balls $B(0, s), B(2s e_1, s)$ and denote by $i_s, i_{2s}$ the indices such that

\begin{alignat*}{3}
p_{i}&\in B(0, s)\quad &&\forall i\leq i_{s},\qquad &&p_{i_s+1}\notin B(0, s),\\
p_{i}&\in B(2e_1, s)\quad &&\forall i\geq i_{2s},\qquad &&p_{i_{2s}-1}\notin B(2se_1, s),
\end{alignat*}
that is, the smallest index $i_s$ after which the path leaves $B(0, s)$ and the largest index $i_{2s}$ before which the path enters $B(2se_1, s)$.
We note that by definition these indices exist and that $i_s< i_{2s}$ holds.
The construction is illustrated in \cref{fig:superadd_paths}.
\begin{figure}[htb]
    \centering
    \includegraphics[width=.7\textwidth]{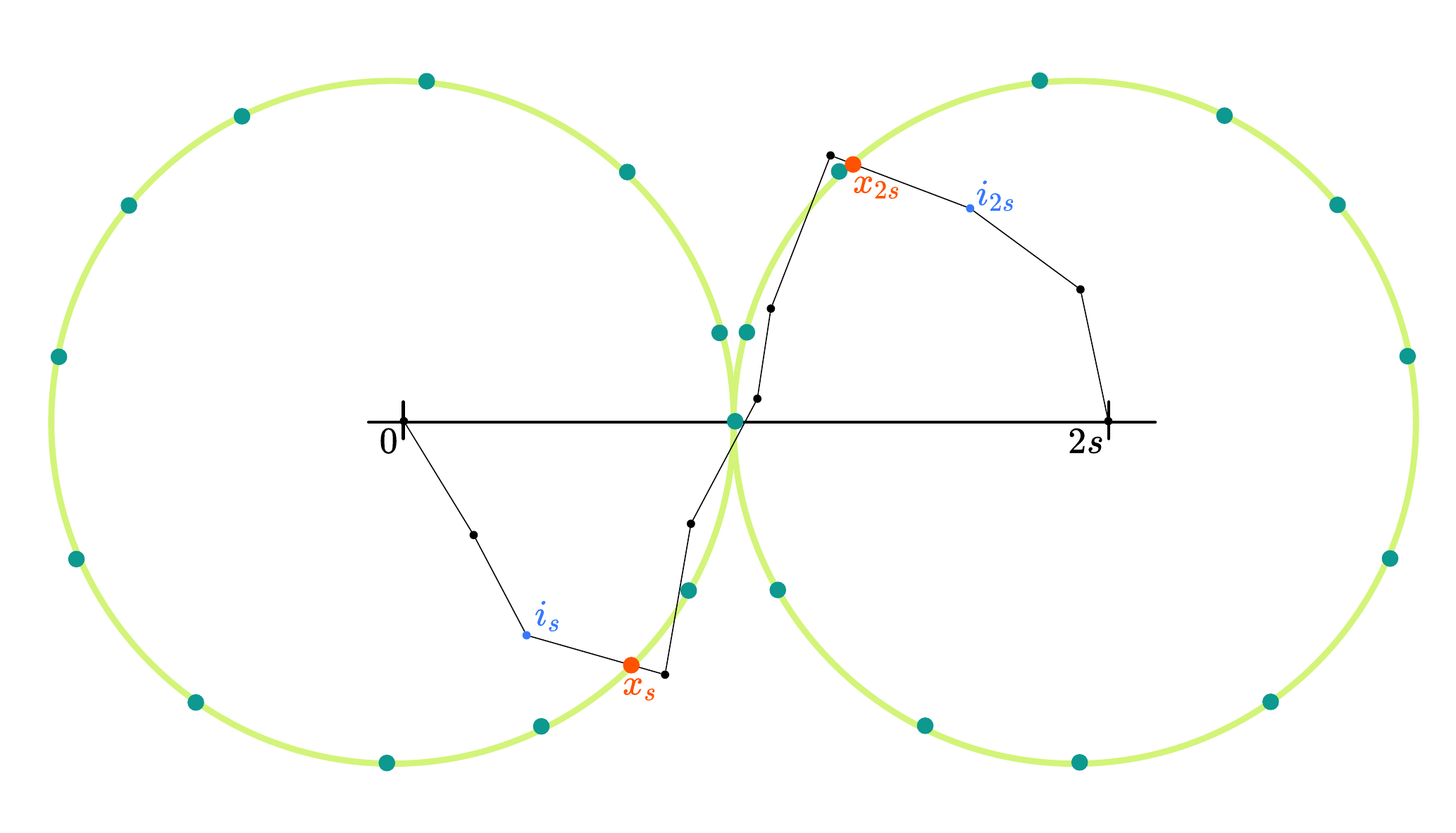}
    \caption{Construction in the proof of \cref{prop:superadditivity}.
    The blue points on the spheres constitute the deterministic coverings $x_i$ and $x_i^\prime$.
    }
    \label{fig:superadd_paths}
\end{figure}
We now take the points where the path intersects the respective spheres,
\begin{align*}
x_s &:= \overline{p_{i_s}p_{i_s+1}}\cap \partial B(0, s),\\
x_{2s} &:= \overline{p_{i_{2s}-1}p_{i_{2s}}}\cap \partial B(2se_1, s),
\end{align*}
for which we have
\begin{align}\nonumber
\dprimeS{h}{\mathcal{X}_s}(0,2se_1) &= \sum_{i=1}^{m-1} \abs{p_{i+1} - p_i}\\ \nonumber
&\geq\sum_{i=1}^{i_s-1} \abs{p_{i+1} - p_i} + \abs{p_{i_s} - \rev x_s\nc} + 
\abs{x_{2s} - p_{i_{2s}}} + 
\sum_{i={i_{2s}}}^{m-1} \abs{p_{i+1} - p_i}\\
&\geq
\label{ineq:lower_bound_superadd}
\dprimeS{h}{\mathcal{X}_s}(0,p_{i_s}) + \dprimeS{h}{\mathcal{X}_s}(p_{i_{2s}}, 2se_1).
\end{align}
We choose a \rev family of $n_s$ points \nc $\{x_i \st 1\leq i\leq n_s\}\subset\partial B(0, s)$ \rev on \nc the sphere $\partial B(0, s)$ such that $x_1 = s e_1$ and the other points are distributed in such a way that for all $x\in\partial B(0, s)$ there exists a point $x_i$ with $\abs{x-x_i}\leq h$.
The construction of these points is straightforward:
Given $\{x_1,\dots,x_k\}$ one chooses
\begin{align*}
    x_{k+1} \in \partial B(0, s)\setminus\bigcup_{i=1}^k B(x_i, h/2).
\end{align*}
Obviously, this process terminates after order $(s/h)^{d-1}$ iterations which means that
\begin{align}\label{ineq:bound_n_s}
    n_s \leq C\left(\frac{s}{h}\right)^{d-1}.
\end{align}
Analogously one defines a covering of $\partial B(2se_1,s)$ into points $\{x'_i \st 1\leq i\leq n_s\}$ by reflecting the points $x_i$ at the point $se_1$.
Note that the value of $n_s$ will turn out to be irrelevant, with the only important thing being that it is at most polynomially large in $s$.

For $i^*\in\{1,\dots,n_s\}$ chosen such that $\abs{\rev x_s\nc-x_{i^*}}\leq h$ it holds that
\begin{align*}
    \abs{p_{i_s}-x_{i^*}} \leq \abs{p_{i_s}-x_s} + \abs{x_s - x_{i^*}} \leq 2h.
\end{align*}
It also holds that 
\begin{align*}
    \dprimeS{h}{\mathcal{X}_s}(p_{i_s}, x_{i^*}) \leq  C_d \abs{p_{i_s}- x_{i^*}}
\end{align*}
which implies that
\begin{align*}
    \dprimeS{h}{\mathcal{X}_s}(0,x_{i^*}) 
    \leq 
    \dprimeS{h}{\mathcal{X}_s}(0,p_{i_s})+
    \rev\dprimeS{h}{\mathcal{X}_s}\nc(p_{i_s}, x_{i^*})
    \leq
    \dprimeS{h}{\mathcal{X}_s}(0,p_{i_s})+ 2 C_d h.
\end{align*}
Analogously, for a suitable $i_*\in\{1,\dots,n_s\}$ one gets
\begin{align*}
    \dprimeS{h}{\mathcal{X}_s}(x'_{i_*},2se_1) 
    \leq
    \dprimeS{h}{\mathcal{X}_s}(p_{i_{2s}},2se_1)+ 2 C_d h.
\end{align*}
Using these two inequalities together with \labelcref{ineq:lower_bound_superadd} we obtain
\begin{align*}
    \dprimeS{h}{\mathcal{X}_s}(0,2 s e_1) 
    &\geq 
    \dprimeS{h}{\mathcal{X}_s}(0,p_{i_s}) + \dprimeS{h}{\mathcal{X}_s}(p_{i_{2s}}, 2 s e_1)
    \geq 
    \dprimeS{h}{\mathcal{X}_s}(0,x_{i^*}) 
    +
    \dprimeS{h}{\mathcal{X}_s}(x'_{i_*},2 s e_1) - 4C_d h
    \\
    &\geq 
    \min_{1\leq i\leq n_s} \dprimeS{h}{\mathcal{X}_s}(0,x_{i}) 
    +
    \min_{1\leq i\leq n_s}
    \dprimeS{h}{\mathcal{X}_s}(x'_{i},2 s e_1) - 4C_d h.
\end{align*}
Taking expectations and using \cref{lem:invariance} with $M=n_s$ as well as \labelcref{ineq:bound_n_s} we get
\begin{align*}
    \Exp{\dprimeS{h}{\mathcal{X}_s}(0,2 s e_1)} 
    &\geq 
    2\Exp{\min_{1\leq i\leq n_s} \dprimeS{h}{\mathcal{X}_s}(0,x_i)}
    - 4C_d h
    -\\
    &\qquad C_1\exp\left(-\left(\frac{h}{C_d}\right)^d+C_2\log\left(\frac{s}{\delta_s}\right)\right)
    \\
    &\geq 
    2\Exp{\min_{1\leq i\leq n_s} \dprimeS{h}{\mathcal{X}_s}(0,x_i)}
    - C_1 h,
\end{align*}
where we used the assumption $\delta_s\leq h \leq s$ and $s>1$ sufficiently large to simplify and absorb the rightmost term into the error term of order $h$. 
The constant $C_1$ changed its value several times.
By adding two zeros and using that because of \cref{lem:invariance} with $M=1$ it holds $\abs{\Exp{\dprimeS{h}{\mathcal{X}_s}(0,x_i)}-\Exp{\dprimeS{h}{\mathcal{X}_s}(0,se_1)}}\leq C_1 h$ for all $i$, we can reorder this inequality in the following way:
\begin{align*}
\begin{split}
    \Exp{\dprimeS{h}{\mathcal{X}_s}(0,2 s e_1)} 
    &\geq 
    2\Exp{\dprimeS{h}{\mathcal{X}_s}(0,se_1)}\\
    &\qquad -2\Exp{\max_{1\leq i\leq n_s} \Big(\Exp{\dprimeS{h}{\mathcal{X}_s}(0,x_i)}-\dprimeS{h}{\mathcal{X}_s}(0,x_i)\Big)}
    -C_1 h,
\end{split}
\end{align*}
where the constant $C_1$ again changed its value.
We shall apply \cref{lem:concentration_to_max} in the appendix to the random variables 
\begin{align*}
Y_i^{(s)} := \frac{1}{\sqrt{\frac{\delta_s^2}{h} s}} \dprimeS{h}{\mathcal{X}_s}(0,x_i), \qquad 1\leq i\leq n_s,
\end{align*}
which satisfy
\begin{align*}
    \Exp{Y_i^{(s)}} 
    &\leq 
    \frac{C_d s}{\sqrt{\frac{\delta_s^2}{h} s}} \leq C_d \sqrt{\frac{h}{\delta_s^2}}\sqrt{s}
    \leq C s 
\end{align*}
for $s>1$ with $s\geq h$ and some constant $C>0$.
Using also \labelcref{ineq:bound_n_s} and the concentration of measure from \cref{thm:concentration_T_s'} we can apply \cref{lem:concentration_to_max} to get that
\begin{align*}
    \Exp{\max_{1\leq i\leq n_s} \Big(\Exp{Y_i^{(s)}} - Y_i^{(s)}\Big)} \leq C_2 \log(s)
\end{align*}
which translates to
\begin{align*}
    \Exp{\max_{1\leq i\leq n_s} \Big(\Exp{\dprimeS{h}{\mathcal{X}_s}(0,x_i)} - \dprimeS{h}{\mathcal{X}_s}(0,x_i)\Big)} \leq C_2 \sqrt{\frac{\delta_s^2}{h}s}\ \log(s).
\end{align*}
Hence, we obtain the desired inequality
\begin{align*}
    \Exp{\dprimeS{h}{\mathcal{X}_s}(0,2 s e_1)} 
    \geq 
    2\Exp{\dprimeS{h}{\mathcal{X}_s}(0,se_1)}
    - C_1 h
    -
    C_2 \sqrt{\frac{\delta_s^2}{h}s}\ \log(s).
\end{align*}
\end{proof}

Similarly, one can prove near monotonicity of the function $s\mapsto\Exp{\dprimeS{h}{\mathcal{X}_s}(0,s e_1)}$.
While we believe that this function should in fact be non-decreasing in $s$, the proof is not obvious. 
However, for our purposes the following approximate monotonicity statement is sufficient.

\begin{proposition}[Near monotonicity]\label[proposition]{prop:monotonicity}
There exist dimensional constants $C_1,C_2>0$ such that for all $s>1$ with $\delta_s\leq h \leq \frac{s}{C_d+\rev 2\nc}$, and $0 \leq s^\prime \leq s$ it holds
\begin{align*}
    \Exp{\dprimeS{h}{\mathcal{X}_s}(0,s^\prime e_1)}
    \leq 
    \Exp{\dprimeS{h}{\mathcal{X}_s}(0,s e_1)} 
    +
    C_1 h 
    +
    C_2 \sqrt{\frac{\delta_s^2}{h}s}\ \log(s).
\end{align*}
\end{proposition}
\begin{proof}
We distinguish two cases, based on whether $s^\prime$ is smaller or larger than $h$.

\textbf{Case 1, $s^\prime \leq h$:}

In this case we can perform trivial estimates:
\begin{align*}
    \Exp{\dprimeS{h}{\mathcal{X}_s}(0,s^\prime e_1)}
    \leq 
    C_d h 
    \rev + h\nc 
    \leq 
    s - h + (C_d+\rev 2\nc )h - s
    \leq 
    s - h 
    \leq 
    \Exp{\dprimeS{h}{\mathcal{X}_s}(0,s e_1)}.
\end{align*}

\textbf{Case 2, $s^\prime \geq h$:}
Using similar notation as in the proof of \cref{prop:superadditivity}, we obtain
\begin{align*}
    \dprimeS{h}{\mathcal{X}_s}(0,s e_1)
    \geq 
    \dprimeS{h}{\mathcal{X}_s}(0,p_{i_{s^\prime}})
    \geq 
    \dprimeS{h}{\mathcal{X}_s}(0,x_{i^*}) - 2 C_d h
    \geq 
    \min_{1\leq i \leq n_s}
    \dprimeS{h}{\mathcal{X}_s}(0,x_{i}) - 2 C_d h.
\end{align*}
With the same arguments as in this previous proof and using that $\abs{x_{i}}=s^\prime \geq h$ we then obtain
\begin{align*}
    \Exp{\dprimeS{h}{\mathcal{X}_s}(0,s e_1)}
    \geq 
    \Exp{\dprimeS{h}{\mathcal{X}_s}(0,s^\prime e_1)}
    -
    C_1 h
    -
    C_2 \sqrt{\frac{\delta_s^2}{h}s}\log(s),
\end{align*}
where $C_2$ originates from an application of \cref{lem:concentration_to_max}.
Combining both cases completes the proof.
\end{proof}

\subsection{Ratio convergence rates}

We can use the previous superadditivity results to prove a convergence rate of the ratios of two distance functions:
\begin{align*}
    \frac{\Exp{\dprimeS{h}{\mathcal{X}_s}(0,se_1)}}{\Exp{\dprimeS{h}{\mathcal{X}_s}(0,2se_1)}} \to \frac{1}{2},\quad s\to\infty.
\end{align*}
Note that, in contrast to \cref{prop:convergence_exp_T_s'}, the limiting constant $\sigma$ does not appear in this ratio convergence.
\begin{proposition}\label[proposition]{prop:ratio_convergence_exp}
Under the conditions of \cref{thm:concentration_T_s',prop:superadditivity} there exist dimensional constants $C_1,C_2>0$ such that it holds for all $s>1$ sufficiently large with $\delta_s\leq h \leq s$ that
\begin{align*}
    \abs{\frac{\Exp{\dprimeS{h}{\mathcal{X}_s}(0,se_1)}}{\Exp{\dprimeS{h}{\mathcal{X}_s}(0,2se_1)}} - \frac{1}{2}} \leq 
    C_1\frac{h}{s}
    +
    C_2 \sqrt{\frac{\delta_s^2}{h}}\frac{\log(s)}{\sqrt{s}}.
\end{align*}
\end{proposition}
\begin{remark}
In the case that $h=\delta_s$ we can again subsume the convergence rate into one term and for some dimensional constant $C_1>0$ we have for $s,t>1$ with $s\geq h$ sufficiently large:
\begin{align*}
    \abs{\frac{\Exp{\dprimeS{h}{\mathcal{X}_s}(0,se_1)}}{\Exp{\dprimeS{h}{\mathcal{X}_s}(0,2se_1)}} - \frac{1}{2}} \leq 
    C_1 \frac{\rev\log(s)^{\frac{1+2d}{2d}}\nc}{\sqrt{s}}.
\end{align*}
\end{remark}
\begin{proof}
The approximate triangle inequality from \cref{lem:triangle_ineq} implies that
\begin{align*}
    \dprimeS{h}{\mathcal{X}_s}(0,2se_1) \leq \dprimeS{h}{\mathcal{X}_s}(0,se_1)
    +
    \dprimeS{h}{\mathcal{X}_s}(se_1,2se_1) \rev + h\nc,
\end{align*}
where we remark that all three distance functions are defined on the same set of points $\mathcal{X}_s$.
Taking expectations and using the approximate translation invariance from \cref{lem:invariance} with $M=1$ yields
\begin{align*}
    \Exp{\dprimeS{h}{\mathcal{X}_s}(0,2se_1)} \leq 2 \Exp{\dprimeS{h}{\mathcal{X}_s}(0,se_1)} + C_1 h,
\end{align*}
where we used the scaling assumption and \labelcref{ineq:error_invariance} to estimate the error term by $C_1 h$.
Using also \cref{prop:superadditivity} we get
\begin{align*}
    \frac{1}{2}
    -
    \frac{C_1 h}{\Exp{\dprimeS{h}{\mathcal{X}_s}(0,2se_1)}}
    &\leq 
    \frac{\Exp{\dprimeS{h}{\mathcal{X}_s}(0,se_1)}}{\Exp{\dprimeS{h}{\mathcal{X}_s}(0,2se_1)}}
    \leq
    \frac{\Exp{\dprimeS{h}{\mathcal{X}_s}(0,se_1)}}{2\Exp{\dprimeS{h}{\mathcal{X}_s}(0,se_1)} - C_1 h - C_2 \sqrt{\frac{\delta_s^2}{h}s}\log(s)}
    \\
    &=
    \left(
    2 - C_1 \frac{h}{\Exp{\dprimeS{h}{\mathcal{X}_s}(0,se_1)}} - C_2 \frac{\sqrt{\frac{\delta_s^2}{h}s}\log(s)}{\Exp{\dprimeS{h}{\mathcal{X}_s}(0,se_1)}}
    \right)^{-1}
    \\
    &\leq 
    \left(
    2 - C_1 \frac{h}{s-h} - C_2 \frac{\sqrt{\frac{\delta_s^2}{h}s}\log(s)}{s-h}
    \right)^{-1}
    \\
    &\leq 
    \left(
    2 - C_1 \frac{h}{s(1-h/s)} - C_2 \frac{\sqrt{\frac{\delta_s^2}{h}}\log(s)}{\sqrt{s}(1-h/s)}
    \right)^{-1}.
\end{align*}
For $s>1$ sufficiently large we can assume that the two negative terms are smaller than $\frac{3}{2}$ and we can use the elementary inequality $\frac{1}{2-x}\leq\frac{1}{2}+x$ for $0\leq x \leq \frac{3}{2}$ to obtain
\begin{align*}
    \abs{\frac{\Exp{\dprimeS{h}{\mathcal{X}_s}(0,se_1)}}{\Exp{\dprimeS{h}{\mathcal{X}_s}(0,2se_1)}}
    -\frac{1}{2}}
    \leq 
    C_1\frac{h}{s}
    +C_2\sqrt{\frac{\delta_s^2}{h}}\frac{\log(s)}{\sqrt{s}},
\end{align*}    
where we used that $2s - h \leq \Exp{\dprimeS{h}{\mathcal{X}_s}(0,2se_1)}\leq 2C_d s\rev+h$ and increased the constants $C_1,C_2$ a little.
\end{proof}

% Using concentration of measure we can extend this ratio convergence to the random variables themselves.
% \begin{corollary}\label[corollary]{cor:ratio_convergence}
% There exist dimensional constants $C_1,C_2,C_3,C_4,C_5>0$ such that for all $\lambda\geq 0$ the event that\todo{check}
% \begin{align*}
%     \abs{\frac{{\dprimeS{h_s}{\mathcal{X}_s}(0,se_1)}}{{\dprimeS{h_s}{\mathcal{X}_s}(0,2se_1)}} - \frac{1}{2}} \leq 
%     C_1\frac{h_s}{s}
%     +
%     C_2 \sqrt{\frac{\delta_s^2}{h_s}}\frac{\log(s)}{\sqrt{s}}
%     +
%     C_3\lambda\sqrt{\frac{\delta_s^2}{h_s}}\frac{1}{\sqrt{s}}
% \end{align*}
% has probability greater or equal than $1-C_4\exp(-C_5\lambda)$.
% \end{corollary}
% \todo{Write the same for the quotient of $d_{h_s}$}

\section{Application to Lipschitz learning}
\label{sec:application}

In this section we discuss an application of our results to the graph infinity Laplace equation which arises in the context of graph-based semi-supervised learning. 
In particular, we will extend our previous results from \cite{bungert2022uniform} by proving uniform convergence rates for Lipschitz learning on graphs with bandwidths on the connectivity threshold. 
An alternative viewpoint of our results is that we prove that finite difference discretizations of the infinity Laplace equation on Poisson clouds converge at the percolation length scale. 
In particular, choosing large stencils---which is required for structured grids, see \cite{oberman2005convergent} but also our results in \cite{bungert2022uniform}---is \emph{not} necessary for Poisson clouds. 

For readers' convenience we first translate the results of the present paper to Poisson processes with intensity $n\gg 1$ which is the natural setting when working on graphs in bounded domains.

\subsection{Rescaling to processes with higher intensity}

Let $X_n$ be a Poisson point process with intensity $n$ in $\R^d$.
This means that 
\begin{align*}
    \Prob{{\#(A\cap X_n) = k}} = \frac{(n\abs{A})^k}{k!} e^{-n\abs{A}},\qquad
    \forall A \subset \R^d.
\end{align*}
In expectation, the number of Poisson points in a set $A$ equals $\Exp{\#(A\cap X_n)}=n\abs{A}$.
Given $x_0,x_1\in\R^d$, we define the affine map 
\begin{align*}
    \Phi(x) := n^\frac{1}{d}R(x-x_0),\quad x\in\R^d,
\end{align*}
where $R\in\R^{d\times d}$ is a suitable orthogonal matrix such that $\Phi(x_1)=n^\frac{1}{d}\abs{\rev x_1\nc-x_0}e_1$.
Using the mapping theorem for Poisson point processes \cite{kingman1992poisson} we can connect the graph distance with step size $\scale>0$ on $X_n$ with the graph distance on a unit intensity process, as studied in the previous sections.
Defining the unit intensity Poisson point process $X:=\Phi(X_n)$, the length and step size
\begin{align}
\label{eq:rescaled_distance}
s &:= n^\frac{1}{d} \abs{x-x_0},\\
\label{eq:rescaled_stepsize}
h &:= n^\frac{1}{d}\scale,
\end{align}
we have
\begin{align*}
    d_{\scale,X_n}(x_0,x)
    =
    n^{-\frac{1}{d}}d_{h,X}(0,s e_1).
\end{align*}
We also have a regularized version of the distance on $X_n$ by defining 
% \todo[inline]{How to do the notation? $X$ is unit intensity, $X_s$ enriched unit intensity, $X_n$ is intensity $n$, so what to choose for enriched intensity $n$? \red I would suggest $X_s$ for intensity $s$, with $X=X_1$, and then to use some other notation to indicate the enriched process. Say $\mathcal{X}_s$, $\overline{X}_s$, $\widehat{X}_s$, $\widetilde{X}_s$, or $X'_s$.
% \blue right, I use $\mathcal{X}_s$ now
% }
\begin{align*}
    d_{\scale}^\prime(x_0,x)
    :=
    n^{-\frac{1}{d}}\dprimeS{h}{\mathcal{X}_s}(0,s e_1),
\end{align*}
where we suppress the dependency of the enriched Poisson process for a more compact notation.
Note that for distances $\abs{x-x_0}$ of order one the choice of $h=h_s\sim\log(s)^\frac{1}{d}$ translates to
\begin{align*}
    \scale 
    = 
    \frac{\log(n^\frac{1}{d}\abs{x-x_0})^\frac{1}{d}}{n^\frac{1}{d}}
    =
    \frac{\left(\frac{1}{d}\log(n) + \log\abs{x-x_0}\right)^\frac{1}{d}}{n^\frac{1}{d}}
    \sim 
    \left(\frac{\log(n)}{n}\right)^\frac{1}{d},
\end{align*}
which is precisely the connectivity threshold for the graph $X_n$.

\begin{remark}[Change of notation]
In what follows we will let $\scale$ denote the length scales used for the distance function on $X_n$. 
Furthermore, we will also suppress the dependency of the distance function on $X_n$ and will simply write $d_{\scale}(x_0,x)$.
\end{remark}
Let us rephrase our previous results which are needed for the application to the graph infinity Laplacian in terms of the rescaled distance function.
These are the localization results \cref{lem:path_T_s_in_box,lem:T_s-T_s'}, the concentration statement \cref{thm:concentration_T_s'}, the near monotonicity from \cref{prop:monotonicity}, and the ratio convergence statement from \cref{prop:ratio_convergence_exp}.
\begin{theorem}[Properties of the distance function on $X_n$]\label{thm:properties_distance_X_n}
% \todo{Please check all rescaled quantities {\color{red} I checked them all. All look fine except for my comment on 5.}\blue should be ok, plz confirm \red All good}
% Let $x_0,x\in\R^d$ and assume that $n\mapsto\scale_n$ is non-decreasing and satisfies 
% \begin{align*}
% K\left(\frac{\log n}{n}\right)^\frac{1}{d}\leq \scale_n \leq \abs{x-x_0}   
% \end{align*}
% for $K>0$ sufficiently large.
% \begin{enumerate}
%     \item {\bf (Convergence)} There exists a dimensional constant $\sigma\in[1,C_d]$ which is independent of $x_0$ and $x$ such that
%     \begin{align*}
%         \lim_{n\to\infty} \frac{\Exp{d_{\scale_n}^\prime(x_0,x)}}{\abs{x_0-x}} = \sigma
%     \end{align*}
%     and, furthermore,
%     \begin{align*}
%         \lim_{n\to\infty} \frac{d_{\scale_n}(x_0,x)}{\abs{x_0-x}} = \sigma
%         \quad
%         \text{and}
%         \quad
%         \lim_{n\to\infty} \frac{d_{\scale_n}^\prime(x_0,x)}{\abs{x_0-x}} = \sigma,
%         \quad
%         \text{almost surely}.
%     \end{align*}
% \end{enumerate}
Let $x_0,x\in\R^d$ and assume
\begin{align*}
    K\left(\frac{\log n}{n}\right)^\frac{1}{d}\leq \scale \leq \abs{x-x_0}.
\end{align*}
Then there exist dimensional constants $C_1,C_2>0$ which are independent of $x_0$ and $x$ such that \rev for $K>0$ sufficiently large\nc:
\begin{enumerate}
    \item {\bf (Concentration)} For all $\lambda>0$ it holds
    \begin{align*}
        \Prob{\abs{d_{\scale}^\prime(x_0,x)-\Exp{d_{\scale}^\prime(x_0,x)}}>\lambda K\left(\frac{\log n}{n}\right)^\frac{1}{d}\sqrt{\frac{\abs{x-x_0}}{\scale}}} \leq C_1\exp(-C_2 \lambda).
    \end{align*}
    \item {\bf (Near monotonicity)} For $n$ sufficiently large, $x_0=0$, and $x\in\R^d$ such that\\ $(C_d+\rev 2\nc) \scale \leq \abs{x}\leq 1$ it holds for all $x^\prime\in\R^d$ with $\abs{x^\prime}\leq\abs{x}$:
    \begin{align*}
        \Exp{d_{\scale}^\prime(0,x^\prime)} 
        \leq
        \Exp{d_{\scale}^\prime(0,x)}
        +
        C_1 \scale
        +
        C_2 K \left(\frac{\log n}{n}\right)^\frac{1}{d}
        (\log n+\log\abs{x})\sqrt{\frac{\abs{x}}{\scale}}.
    \end{align*}
    \item {\bf (Ratio convergence in expectation)} For $n$ sufficiently large, $x_0=0$, and $x\in\R^d$ such that $\scale\leq \abs{x}$ it holds that
    \begin{align*}
        \abs{
        \frac{\Exp{d_{\scale}^\prime(0,x)}}{\Exp{d_{\scale}^\prime(0,2x)}}
        -\frac{1}{2}} \leq C_1 \frac{\scale}{\abs{x}} + C_2 K \left(\frac{\log n}{n}\right)^\frac{1}{d}
        \frac{\log n + \log\abs{x}}{\sqrt{\scale\abs{x}}}.
    \end{align*}
    \item\label{it:localization} {\bf (Localization)} \rev For $\abs{x-x_0}\geq 2\scale$ \nc it holds
    \begin{align*}
        &\Prob{
        \text{any optimal path of $d_\scale(x_0,x)$ lies in $B(x_0, C_d^\prime \abs{x_0-x})$}
        }
        \\
        &\qquad\geq 
        1 -  \exp\left(-C_1 n\scale^d + C_2\log( n\abs{x_0-x})\right),
        \\
        &\Prob{d_\scale(x_0,x)=d_\scale^\prime(x_0,x)} 
        \geq 
        1 - 2\exp\left(-C_1 n \scale^d + C_2
        \log( n\abs{x_0-x})\right)
        .
    \end{align*}
\end{enumerate}
\end{theorem}
\begin{proof}
One simply uses \labelcref{eq:rescaled_distance,eq:rescaled_stepsize} and observes that $\delta_s = C_d(k\log(\rev C_d^\dprime\nc s))^\frac{1}{d} = K (\log n)^\frac{1}{d}$ for a suitable constant $K=K(d)$.
\end{proof}

\subsection{Convergence rates}

We still let $X_n$ be a Poisson point process with intensity $n\in\R$ \rev on $\R^d$ \nc and let $\domain\subset\R^d$ be an open and bounded domain.
%In fact, $X_n$ can be generated by restricting a suitable process on $\R^d$ with density $n/\abs{\domain}$ to $\closure\domain$.
%Alternatively, $X_n$ can be extended to a process on $\R^d$.
% \todo[inline]{does all that make sense? 
% The intensity changes. {\color{red} No, the intensity shouldn't change. If you restrict an $n$ intensity process to a subset, it still has intensity $n$. You may be thinking of trying to ensure you have $n$ points in $\Omega$ on average, in which case you would scale like this (like I do in the dePoissonization). There is no need to do this here, and no need to restrict $n\in \N$ (indeed, I use real-valued $n/|\domain|$ in the dePoissonization). I don't think you need to say how to construct it.}}
Remember that for a bandwidth parameter $\scale>0$ and a function $u_n : X_n \to \R$ we defined the graph infinity Laplacian of $u_n$ as
\begin{align*}
    \mathcal{L}_\infty^\scale u_n(x) :=
    \sup_{y\in B(x,\scale)\cap X_n}\frac{u(y)-u(x)}{\abs{y-x}}
    +
    \inf_{y\in B(x,\scale)\cap X_n}\frac{u(y)-u(x)}{\abs{y-x}},\qquad x \in X_n.
\end{align*}
Solutions of the graph infinity Laplacian equation $\mathcal{L}_\infty^{\rev\scale\nc} u_n = 0$ satisfy a special comparison principle with the graph distance function, called \emph{comparison with cones}.
To explain this, we introduce some terminology.
For a subset $A\subset X_n$ we define its graph boundary and closure as
\begin{align*}
    \operatorname{bd}_\scale(A) &:= \Set{x\in X_n\setminus A \st \exists y \in A, \, \abs{x-y}\leq \scale}, \\
    \operatorname{cl}_\scale(A) &:= A \cup \operatorname{bd}_\scale(A).
\end{align*}
Furthermore, we refer to a subset $A\subset X_n$ as $\scale$-connected if for all points $x,y\in A$ there exists a path in $A$ which connects $x$ and $y$ and has hops of maximal size $\scale$, in other words if $d_{\scale}(x,y)[A]<\infty$.

We say that $u_n$ satisfies comparison with cones on a subset $X_n^\prime\subset X_n$ if for every subset $X_n^\dprime\subset X_n^\prime$, for all $a\geq 0$, and for all $z\in X_n^\prime\setminus X_n^\dprime$ it holds
\begin{subequations}\label{eq:CC}
\begin{align}
    \max_{\operatorname{cl}_\scale({X_n^\dprime})}\Big(u_n - a\,d_{\scale}(\cdot,z)\Big)
    &=
    \max_{\operatorname{bd}_\scale({X_n^\dprime})}\Big(u_n - a\,d_{\scale}(\cdot,z)\Big),\\
    \min_{\operatorname{cl}_\scale({X_n^\dprime})}\Big(u_n - a\,d_{\scale}(\cdot,z)\Big)
    &=
    \min_{\operatorname{bd}_\scale({X_n^\dprime})}\Big(u_n - a\,d_{\scale}(\cdot,z)\Big).
\end{align}
\end{subequations}
We have the following result:
\begin{theorem}[{\cite[Theorem 3.2]{bungert2022uniform}}]
Let $X_n^\prime\subset X_n$ be \rev an \nc $\scale$-connected subset of $X_n$ and let $u_n:X_n^\prime\to\R$ satisfy $\mathcal{L}_\infty^\scale u_n(x) = 0$ for all $x\in X_n^\prime$.
Then $u_n$ satisfies comparison with cones on $X_n^\prime$.
\end{theorem}
The goal of this section is to establish rates of convergence for solutions of $\mathcal{L}_\infty^\scale u_n = 0$ to solutions of the infinity Laplacian equation $\Delta_\infty u = 0$, where $\Delta_\infty u := \sum_{i,j=1}^d\partial_i u\partial_j u\partial_{ij}^2 u$ for smooth functions $u$.
Note that solutions to the infinity Laplacian equation are not $C^2$ in general which is why one typically uses the theory of viscosity solutions.
However, solutions can be characterized through a comparison with cones property, as well.
We refer to the seminal monograph \cite{aronsson2004tour} for this and other important properties of the infinity Laplacian equation.

For proving the rates we shall utilize the framework which we developed in \cite{bungert2022uniform} and which only relies on the comparison with cones property of the respective solutions.
The novel idea there was the introduction of a homogenized length scale $\homscale>\scale$, a corresponding extensions $u_n^\homscale$ of a graph solution $u_n$, and a homogenized infinity Laplacian operator $\Delta_\infty^\homscale$.
The general recipe for getting rates as in \cite{bungert2022uniform} is the following:
\begin{enumerate}
    \item Let $\mathcal{L}_\infty^\scale u_n=0$ and $\Delta_\infty u = 0$.
    \item Use convergence of the distance function to prove that
    \begin{align*}
        -\Delta_\infty^\homscale u_n^\homscale \lesssim \mathrm{error}(n,\scale,\homscale)
        \qquad
        \text{and}
        \qquad
        \sup\abs{u_n - u_n^\homscale} \lesssim \homscale.
    \end{align*}
    \item Perturb the continuum solution $u$ to a function $\Tilde{u}$ which satisfies
    \begin{align*}
        -\Delta_\infty^\homscale \Tilde{u}
        \gtrsim \mathrm{error}(n,\scale,\homscale)
        \qquad
        \text{and}
        \qquad
        \sup\abs{u-\Tilde{u}}
        \lesssim
        \homscale+\sqrt[3]{\mathrm{error}(n,\scale,\homscale)}.
    \end{align*}
    \item Use a comparison principle for $\Delta_\infty^\homscale$ and repeat the argument for $-u_n$ and $-u$ to get
    \begin{align*}
        \sup \abs{u_n - u}
        \lesssim 
        \homscale + \sqrt[3]{\mathrm{error}(n,\scale,\homscale)}.
    \end{align*}
    \item Optimize over $n,\scale,\homscale$ to get explicit rates.
\end{enumerate}
Note that in \cite{bungert2022uniform} a careful analysis of boundary conditions and regularity is performed in order to be able to perform the arguments above all the way up to the boundary.
Furthermore, the introduction of the homogenized operator allowed us to obtain convergence rates for arbitrary small graph bandwidths satisfying
\begin{align*}
\scale\gg\left(\frac{\log n}{n}\right)^\frac{1}{d}.    
\end{align*}
The purpose of this section is to show how our results on Euclidean first-passage percolation allows to improve the error term $\mathrm{error}(n,\scale,\homscale)$ in order to allow for length scales of the form
\begin{align*}
\scale
\sim 
\left(\frac{\log n}{n}\right)^\frac{1}{d}. 
\end{align*}
Let us now introduce the homogenized quantities.
For $\homscale>0$ we define extensions of the discrete function $u_n:X_n\to\R$ to functions $u_n^\homscale,(u_n)_\homscale:\rev\R^d\nc\to\R$ as follows:
\begin{subequations}
\begin{align}
    u_n^\homscale(x) &:= \sup_{B(x, \homscale)\cap X_n} u_n,\qquad x\in\rev\R^d\nc,\\
    (u_n)_\homscale(x) &:= \inf_{B(x, \homscale)\cap X_n} u_n,\qquad x\in\rev\R^d\nc.
\end{align}
\end{subequations}
Note that both extrema are attained if $B(x,\homscale)\cap X_n\neq\emptyset$ since this set is of finite cardinality.
We also define the nonlocal infinity Laplacian with respect to $\homscale>0$ of a function $u:\rev\R^d\nc\to\R$ as
\begin{align}
    \Delta_\infty^\homscale u(x) := \frac{1}{\homscale^2}\left(\sup_{B(x, \homscale)}u - 2u(x) +  \inf_{B(x, \homscale)}u\right),\qquad x\in\rev\R^d\nc.
\end{align}
Lastly, for a positive number $r>0$ we define inner parallel sets of $\domain$ as
\begin{align}
    \domain^r := \Set{x\in\domain\st\dist(x,\partial\domain)>r}.
\end{align}
\begin{theorem}\label{thm:consistency}
Let $\domain\subset\R^d$ be an open and bounded domain and $X_n$ be a Poisson point process \rev on $\R^d$ \nc with density $n\in\N$.
Assume that $\scale>0$ and $\homscale>0$ satisfy
\begin{align}\label{eq:scaling}
    K \left(\frac{\log n}{n}\right)^\frac{1}{d}
    \leq 
    \scale
    \rev
    \leq
    \frac{1}{K}
    \homscale
    ,\qquad
    0<\homscale<1,\nc
\end{align}
and define 
\begin{align*}
    \constr_n := \left\{x\in X_n\cap\closure\domain \st \dist(x,\partial\domain) \leq \scale\right\}.    
\end{align*}
Let \rev $g:\closure\domain\to\R$ be a Lipschitz function and \nc $u_n:X_n\to\R$ solve
\begin{align*}
    \begin{cases}
        \mathcal{L}_\infty^\scale u_n = 0\quad&\text{on }X_n\setminus \constr_n, \\
        u_n = g&\text{on }\constr_n.
    \end{cases}
\end{align*}
% Let $n\in\N$ be sufficiently large and
% \begin{align}\label{eq:scaling}
%     K \left(\frac{\log n}{n}\right)^\frac{1}{d}
%     \leq 
%     \scale
%     \leq\frac{1}{2}
%     \qquad
%     2\scale\leq \homscale \leq 1,
% \end{align}
% with a sufficiently large constant $K>0$.
% Let $u_n:X_n\to\R$ satisfy $\mathcal{L}_\infty^\scale u_n = 0$ on an $\scale$-connected subset $X_n^\prime\subset X_n$ such that $X_n^\prime\subset\domain$ for all $n\in\N$, where $\domain$ is an open and bounded domain.
Then there exist dimensional constants $C_1,C_2,C_3,C_4,C_5>0$ and $C_6>1$ such that \rev for all $\lambda\geq 0$ and for $K\geq 8$ sufficiently large \nc with probability at least 
\begin{align*}
    1 - C_1\exp\left(-C_2 K^d \log n\right) - C_3 \exp(-C_4\lambda+C_5\log n)
\end{align*}
it holds for all $x_0\in\domain^{2 C_6 \homscale}$ that
\begin{subequations}
\begin{align}
    -\Delta_\infty^\homscale u_n^\homscale(x_0) 
    &\lesssim
    \phantom{{}-}
    \rev\Lip(g)\nc
    \left(
    (\log n + \lambda)
    \left(\frac{\log n}{n}\right)^\frac{1}{d}
    \frac{1}{\sqrt{\homscale^3\scale}}
    +
    \frac{\scale}{\homscale^2}
    \right),
    \\
    -\Delta_\infty^\homscale (u_n)_\homscale(x_0) 
    &\gtrsim
    -\rev\Lip(g)\nc
    \left(
    (\log n + \lambda)
    \left(\frac{\log n}{n}\right)^\frac{1}{d}
    \frac{1}{\sqrt{\homscale^3\scale}}
    +
    \frac{\scale}{\homscale^2}
    \right).
\end{align}
\end{subequations}
\end{theorem}
\begin{remark}
Abbreviating $\delta_n := \left(\frac{\log n}{n}\right)^\frac{1}{d}$ our result translates to
\begin{align*}
    -\Delta_\infty^\homscale u_n^\homscale(x_0) 
    \lesssim
    \rev\Lip(g)\nc
    \left(
    (\log n + \lambda)
    \frac{\delta_n}{\sqrt{\homscale^3\scale}}
    +
    \frac{\scale}{\homscale^2}
    \right)
    .
\end{align*}
In particular, we can choose $\scale=\delta_n$ and the error term reduces to $\sqrt{\frac{\delta_n}{\homscale^3}}$ which goes to zero if $\homscale$ is sufficiently large compared to $\delta_n$.
In our previous work \cite[Theorem 5.13]{bungert2022uniform} we proved an analogous result for arbitrary weighted graphs (whose vertices could also be deterministic) with connectivity radius $\delta_n$, graph bandwidth $\scale$, and a free parameter $\homscale$.
There we proved that
\begin{align*}
    -\Delta_\infty^\homscale u_n^\homscale(x_0)
    \lesssim
    \rev\Lip(g)\nc
    \left(
    \frac{\delta_n}{\homscale\scale}
    +
    \frac{\scale}{\homscale^2}
    \right)
\end{align*}
and one observes that choosing $\scale=\delta_n$ is not possible since then the right hand side would diverge as $\homscale \to 0$.
\end{remark}
\begin{proof}
The proof follows very closely our earlier result \cite[Theorem 5.13]{bungert2022uniform} but involves non-trivial adaptations.

It suffices to prove the first statement since the second one follows by changing the signs of $u_n$.
Furthermore, it suffices to prove the statement for graph vertices $x_0\in X_n$ and then use \cite[Lemma 5.8]{bungert2022uniform} to extend it to continuum points, which does only incur error terms that are already present and increases the constant~$C_6$.

Let us fix $x_0\in\domain^{2 C_6\homscale}$ where for now we assume that $C_6>1$.
Utilizing that
\begin{align*}
    \sup_{B(x_0, \homscale)}u_n^\homscale 
    &=
    \sup_{x\in B(x_0, \homscale)}\sup_{B(x, \homscale)\cap X_n}u_n
    =
    \sup_{B(x_0, 2\homscale)}u_n
    =
    u_n^{2\homscale}(x_0),
    \\
    \inf_{B(x_0, \homscale)}u_n^\homscale 
    &=
    \inf_{x\in B(x_0, \homscale)}\sup_{B(x, \homscale)\cap X_n}u_n
    \geq 
    u_n(x_0),
\end{align*}
we obtain
\begin{align}\label{ineq:estimate_NL_inf_lapl}
    -\homscale^2\Delta_\infty^\homscale u_n^\homscale(x_0) 
    \leq 
    2u_n^\homscale(x_0) - u_n^{2\homscale}(x_0) - u_n(x_0).
\end{align}
To estimate this term, we turn our attention to the function $u_n$ and the fact that it satisfies comparison with cones.
For this we define the set $\mathsf B_n(x_0, 2\homscale)\subset X_n$ as
\begin{align}\label{eq:set_B}
    \mathsf B_n(x_0, 2\homscale) := \Set{x\in X_n \setminus \Set{x_0} \st d_{\scale}(x_0,x)\leq \inf_{y\in B(x_0,2\homscale-\scale)^c}d_{\scale}(x_0,y) -  \scale}.
\end{align}
% \todo[inline]{Check red {\red I checked it and ajusted some things. I think a union bound is needed in the first line. Then I think $d=d'$ on $B(x_0,C_6\tau)$ is more than enough. Maybe you can get by with less later in the proof, but nothing changes substantially in the probability by using a smaller ball.}}
% Using the localization statement \cref{it:localization} from \cref{thm:properties_distance_X_n} and a union bound we have
\rev 
We start by recording a couple of properties of the set $\mathsf B_n(x_0,2\homscale)$:

First, we observe that 
\begin{align}\label{eq:first_prop_ball}
    \mathsf B_n(x_0, 2\homscale) \subset B(x_0,2\homscale-\scale)    
\end{align}
since otherwise there would be a point $x \in \mathsf B_n(x_0, 2\homscale)$ such that $d_\scale(x_0,x)\leq d_\scale(x_0,x) - \scale$ which is a contradiction.
% Applying the localization statement \cref{it:localization} from \cref{thm:properties_distance_X_n} and a performing a union bound over all points in $\mathsf B_n(x_0,2\homscale)$ hence shows that with high probability all paths which realize $d_\scale(x_0,x)$ for $x\in \mathsf B_n(x_0, 2\homscale)$ lie in $B(x_0,2 C_6\tau)$ for a suitably enlarged constant $C_6>1$, and we obtain $d_\scale(x_0,x) = d'_\scale(x_0,x)$ for all $x\in \mathsf B_n(x_0, 2\homscale)$.

Second, we claim that
\begin{align}\label{eq:second_prop_ball}
\inf_{y\in B(x_0,2\homscale-\scale)^c}d_{\scale}(x_0,y)
 =
\inf_{y\in B(x_0,2\homscale)\setminus B(x_0,2\homscale-\scale)}d_{\scale}(x_0,y)
\end{align}
which is going to be relevant a little later.
To see this, note that the left hand side is always smaller or equal than the right hand side.
Furthermore, any feasible path from a point $y\in B(x_0,2\homscale-\scale)^c$ to $x_0$ has to contain a point in $B(x_0,2\homscale)$ and can hence be truncated to obtain a feasible path for the right side.
% Using again \cref{it:localization} and a union bound yields that \rev all paths realizing the infimum lie in $B(x_0,2 C_6\tau)$, as well.
% Therefore we can replace $d_\scale$ by $d_\scale^\prime$ in \labelcref{eq:set_B} with high probability \nc and in fact the probability is at least $1-C_1\exp\left(-C_2 K^d \log n\right)$ for some constants $C_1,C_2>0$.

\rev Third, \nc we claim that the (graph) boundary of $\mathsf B_n(x_0, 2\homscale)$ satisfies
\begin{align}\label{eq:boundary_inclusion}
\begin{split}
    \operatorname{bd}_\scale(\mathsf B_n(x_0, 2\homscale)) 
    &\subset 
    \Big\{x\in \rev X_n \cap B(x_0, 2\homscale) \nc \st
    \\
    &\qquad
    d_{\scale}(x_0,x) >
    \inf_{y\in B(x_0,2\homscale-\scale)^c}d_{\scale}(x_0,y) - \scale\Big\} \cup \Set{x_0} = : \mathsf B'
\end{split}    
\end{align}
and in particular $\mathsf B'\subset\domain$.
% To prove \labelcref{eq:boundary_inclusion} we first show that for all $x\in \mathsf B_n(x_0, 2\homscale)$ it holds $x\in B(x_0,2\homscale-\scale)$. 
% If this were not the case, \rev employing that $d_{\scale}^\prime=d_{\scale}$ would yield\nc
% %
% \begin{align*}
%     \abs{x_0-x} 
%     &= 
%     \frac{\abs{x_0-x}}{d_{\scale}^\prime(x_0,x)}d_{\scale}^\prime(x_0,x)
%     \leq 
%     \frac{\abs{x_0-x}}{\inf_{y\in B(x_0,2\homscale-\scale)^c}d_{\scale}^\prime(x_0,y)}d_{\scale}^\prime(x_0,x)
%     \\
%     &\leq 
%     \abs{x_0-x}
%     \frac{\inf_{y\in B(x_0,2\homscale-\scale)^c}d_{\scale}^\prime(x_0,y)- \scale}{\inf_{y\in B(x_0,2\homscale-\scale)^c}d_{\scale}^\prime(x_0,y)}
%     <\abs{x_0-x}
% \end{align*}
% which is a contradiction. 
By definition, for $z\in\operatorname{bd}_\scale(\mathsf B_n(x_0, 2\homscale))$ there exists $x\in \mathsf B_n(x_0, 2\homscale)$ with $\abs{x-z}\leq \scale$ and hence, \rev using also \labelcref{eq:first_prop_ball}\nc, we get
\begin{align*}
    \abs{z-x_0}\leq\abs{z-x}+\abs{x_0-x} \leq \rev \scale + 2\homscale-\scale = \nc 2\homscale,
\end{align*}
which proves \labelcref{eq:boundary_inclusion}. \rev In particular, we see by \labelcref{eq:scaling,eq:boundary_inclusion} that for $C_6>1$ sufficiently large it holds 
$\mathsf{B}^\prime \cap \mathcal{O}_n = \emptyset$. \nc
%%%%%%%%%%%%%%%%%%%%%%%%%%
% We first claim that the (graph) boundary of the set $\mathsf B_n(x_0, 2\homscale)$ defined in \labelcref{eq:set_B} satisfies
% \begin{align}\label{eq:boundary_inclusion}
%     \partial \mathsf B_n(x_0, 2\homscale) \subset \Set{x\in X_n^\prime \st d_{\scale}(x_0,x) > \inf_{y\in B(x_0,2\homscale-\scale)^c}d_{\scale}^\prime(x_0,y) - 2 \scale, \; \abs{x-x_0}\leq 2\homscale} \cup \Set{x_0} = : B'.
% \end{align}
% Since by assumption $x_0\in X_n^\prime$ and $\overline{\mathsf B_n(x_0, 2\homscale)}\subset X_n^\prime$, the only non-trivial thing to check is whether $\abs{x-x_0}\leq \inf_{y\in B(x_0,2\homscale-\scale)^c}d_{\scale}^\prime(x_0,y)$. 
% This follows from the following argument: By definition of the boundary, for $x\in\partial \mathsf B_n(x_0, 2\homscale)$ there exists $y\in \mathsf B_n(x_0, 2\homscale)$ with $\abs{x-y}\leq h$ and therefore by \cref{lem:lower_bound}
% \begin{align*}
%     \abs{x-x_0} \leq \abs{x-y} + \abs{x_0-y} \leq h + \abs{y-x_0} \leq 
%     2 \scale + \abs{x_0-y} - h
%     \leq 
%     2 \scale + d_{\scale}(x_0,y) \leq \inf_{y\in B(x_0,2\homscale-\scale)^c}d_{\scale}^\prime(x_0,y).
% \end{align*}
%%%%%%%%%%%%%%%%%%%%%%5
We have the following trivial inequality:
\begin{align*}
    u_n(x) \leq u_n(x_0) + \left(u_n^{2\homscale}(x_0) - u_n(x_0)\right) \frac{d_{\scale}(x_0,x)}{\inf_{y\in B(x_0,2\homscale-\scale)^c}d_{\scale}(x_0,y)- \rev\scale\nc},\qquad\forall x\in \mathsf B'.
\end{align*}
Indeed, if $x=x_0$ the inequality is in fact an equality, and for all $x\in B(x_0,2\homscale)$ it is also true since $u_n^{2\homscale}(x_0)\geq u_n(x)$.
Consequently, since $\operatorname{bd}_\scale(\mathsf B_n(x_0, 2\homscale)) \subset \mathsf B' \rev\subset (X_n\cap\closure\Omega)\setminus\constr_n\nc$ and $u_n$ satisfies comparison with cones \rev on this set\nc, we infer that for all $x\in \operatorname{cl}_\scale({\mathsf B_n(x_0, 2\homscale)})$ it holds
\begin{align}\label{ineq:CC_general}
    u_n(x) \leq u_n(x_0) + \left(u_n^{2\homscale}(x_0) - u_n(x_0)\right) \frac{d_{\scale}(x_0,x)}{\inf_{y\in B(x_0,2\homscale-\scale)^c}d_{\scale}(x_0,y) - \scale}.
\end{align}
Without loss of generality we can assume that $C_d\leq 3/2$ (otherwise, one can increase $K$ in the definition of $\scale$, see \cref{rem:upper_bounds}).
\rev Using \cref{lem:upper_bound} this ensures that for all $x\in B(x_0, \homscale)$ we have with probability at least $1-C_1\exp\left(-C_2 K^d \log n\right)$ for some constants $C_1,C_2>0$ that
\begin{align}\label{ineq:GD_bound1}
    d_\scale(x_0,x)
    \leq 
    C_d \homscale+\scale 
    \leq 
    \frac{3}{2}\homscale + \scale
    =
    2\homscale-3\nlscale
    + 4\scale - \frac{1}{2}\homscale.
\end{align}
On the other hand, using \cref{lem:lower_bound} we also have
\begin{align}\label{ineq:GD_bound2}
\begin{split}
    \inf_{y\in B(x_0,2\homscale-\scale)^c}d_{\scale}(x_0,y) -  \scale
    &\geq 
    \inf_{y\in B(x_0,2\homscale-\scale)^c}\abs{x_0-y} - \dist(y,\mathcal{X}_s) -  \scale
    \\
    &\geq 
    2\homscale-3\scale.
\end{split}    
\end{align}
Since $\tau \geq K\scale\geq 8\scale$ we infer from \labelcref{ineq:GD_bound1,ineq:GD_bound2} that 
\begin{align*}
    d_\scale(x_0,x)\leq 
    \inf_{y\in B(x_0,2\homscale-\scale)^c}d_{\scale}(x_0,y) -  \scale
\end{align*}
and \nc this implies $B(x_0, \homscale)\subset \mathsf B_n(x_0, 2\homscale)$.
Consequently, we can maximize both sides in \labelcref{ineq:CC_general} over $x\in B(x_0, \homscale)\cap X_n$ to get
\begin{align*}
    u_n^\homscale(x_0)
    &\leq 
    u_n(x_0) + \left(u_n^{2\homscale}(x_0) - u_n(x_0)\right) \frac{\sup_{x\in B(x_0, \homscale)\cap X_n}d_{\scale}(x_0,x)}{\inf_{y\in B(x_0,2\homscale-\scale)^c}d_{\scale}(x_0,y) - \scale}
    \\
    &\leq 
    u_n(x_0) 
    + \left(u_n^{2\homscale}(x_0) - u_n(x_0)\right) \frac{\sup_{x\in B(x_0, \homscale)\cap X_n}d_{\scale}(x_0,x)}{\inf_{y\in B(x_0,2\homscale-\scale)^c}d_{\scale}(x_0,y)}\\
    &\hspace{15em}\times\left(1 + \frac{\scale}{\inf_{y\in B(x_0,2\homscale-\scale)^c}d_{\scale}(x_0,y)}\right).
\end{align*}
In the last step we used the elementary inequality $\frac{1}{1-t}\leq 1+2 t$ for $0\leq t \leq 1/2$.
\rev
Now we argue that we can replace $d_\eps$ by $d_\eps^\prime$ in this expression with high probability:
First, we finally use property \labelcref{eq:second_prop_ball} from above which tells us that the infimum $\inf_{y\in B(x_0,2\homscale-\scale)^c}d_{\scale}(x_0,y)$ can be restricted to the annulus $B(x_0,2\homscale)\setminus B(x_0,2\homscale-\scale)$.
Hence \cref{it:localization} in \cref{thm:properties_distance_X_n} implies that $d_\scale(x_0,y)=d_\scale^\prime(x_0,y)$ for all $y\in B(x_0,2\homscale)\setminus B(x_0,2\homscale-\scale)$ with probability at least $1-C_1\exp\left(-C_2 K^d \log n\right)$ where we possibly increase $C_1$ and $C_2$.

Second, we argue for the supremum.
Possibly increasing $C_1$ and $C_2$ with probability at least $1-C_1\exp\left(-C_2 K^d \log n\right)$ it is finite and let us assume it is attained at a point $\hat x\in B(x_0,2\scale)\cap X_n$.
Then using \cref{lem:upper_bound} with the same probability we have
\begin{align*}
    2C_d\scale \geq d_\cap(x_0,\hat x) = 
    \sup_{x\in B(x_0, \homscale)\cap X_n}d_{\scale}(x_0,x)
    \geq 
    d_{\scale}(x_0,\tilde x)
    \geq 
    (K-1)\scale
\end{align*}
for every point $\tilde x \in \left(B(x_0,K\scale)\cap X_n\right)\setminus B(x_0,(K-1)\scale)$.
Note that if $K$ is sufficiently large then such a point exists with the same probability.

This is a contradiction if $K>2 C_d+1$ and so \cref{it:localization} in \cref{thm:properties_distance_X_n} again lets us replace $d_\eps(x_0,x)$ by $d_\eps^\prime(x_0,x)$ for all $x\in B(x_0, \homscale)\cap X_n$.
Hence, we obtain
\begin{align*}
    u_n^\homscale(x_0)
    &\leq 
    u_n(x_0) 
    + \left(u_n^{2\homscale}(x_0) - u_n(x_0)\right) \frac{\sup_{x\in B(x_0, \homscale)\cap X_n}d_{\scale}^\prime(x_0,x)}{\inf_{y\in B(x_0,2\homscale-\scale)^c}d_{\scale}^\prime(x_0,y)}\\
    &\hspace{15em}\times \left(1 + \frac{\scale}{\inf_{y\in B(x_0,2\homscale-\scale)^c}d_{\scale}^\prime(x_0,y)}\right)
\end{align*}
with probability at least $1-C_1\exp\left(-C_2 K^d \log n\right)$.
\nc 
Introducing the shortcut notations
\begin{subequations}\label{eq:d_bar_under_r}
\begin{align}
    \overline{d}_\homscale(x_0) &:= \sup_{x\in B(x_0, \homscale)\cap X_n}d_{\scale}^\prime(x_0,x), \\
    \underline{d}_{2\homscale}(x_0) &:= \inf_{y\in B(x_0,2\homscale-\scale)^c}d_{\scale}^\prime(x_0,y),\\
    r_\homscale(x_0) &:= \frac{\overline{d}_\homscale(x_0)}{\underline{d}_{2\homscale}(x_0)} - \frac{1}{2},
\end{align}
\end{subequations}
we can rewrite and continue the previous estimate as follows:
\begin{align*}
    u_n^\homscale(x_0) 
    &\leq 
    u_n(x_0) + 
    \left(u_n^{2\homscale}(x_0) - u_n(x_0)\right) \left(r_\homscale(x_0) + \frac{1}{2}\right)
    \left(1 + \frac{\scale}{\inf_{y\in B(x_0,2\homscale-\scale)^c}d_{\scale}^\prime(x_0,y)}\right)
    \\
    &\leq 
    \left(u_n^{2\homscale}(x_0) - u_n(x_0)\right) r_\homscale(x_0)
    +
    \frac{1}{2}
    \left(u_n(x_0)
    +
    u_n^{2\homscale}(x_0)\right)
    \\
    &\qquad
    +
    \left(u_n^{2\homscale}(x_0) - u_n(x_0)\right) \left(r_\homscale(x_0) + \frac{1}{2}\right)
    \frac{\scale}{\inf_{y\in B(x_0,2\homscale-\scale)^c}d_{\scale}^\prime(x_0,y)}.
\end{align*}
Returning to \labelcref{ineq:estimate_NL_inf_lapl} we obtain
\begin{align*}
    -\homscale^2\Delta_\infty^\homscale u_n^\homscale(x_0) 
    &\leq 
    2\left(u_n^{2\homscale}(x_0) - u_n(x_0)\right) r_\homscale(x_0)
    \\
    &\qquad 
    +
    2\left(u_n^{2\homscale}(x_0) - u_n(x_0)\right) \left(r_\homscale(x_0) + \frac{1}{2}\right)
    \frac{\scale}{\inf_{y\in B(x_0,2\homscale-\scale)^c}d_{\scale}^\prime(x_0,y)}
    \\
    &\leq 
    2
    \Lip_n(u_n)
    d_\scale(x_0,x_0^*)
    r_\homscale(x_0) 
    \\
    &\qquad
    + 
    2
    \Lip_n(u_n)
    d_\scale(x_0,x_0^*)
    \left(
    r_\homscale(x_0) + \frac{1}{2}
    \right)
    \frac{\scale}{\inf_{y\in B(x_0,2\homscale-\scale)^c}d_{\scale}^\prime(x_0,y)},
\end{align*}
where we let $x_0^*\in B(x_0,2\homscale)\cap X_n$ be a point which realizes $u_n^{2\homscale}(x_0)$ \rev and define the graph Lipschitz constant 
\begin{align}
    \Lip_n(u_n) := \max_{x,y\in X_n}\frac{\abs{u_n(x)-u_n(y)}}{d_\scale(x,y)}.
\end{align}
Since $u_n$ solves the graph infinity Laplace equation it holds $\Lip_n(u_n) = \Lip_n(g)$ by \cite[Proposition 3.8]{bungert2022uniform} and using \cref{lem:lower_bound} we get 
\begin{align*}
    \Lip_n(g) = \max_{x,y\in X_n}\frac{\abs{g(x)-g(y)}}{d_\scale(x,y)} \leq \max_{x,y\in X_n}\frac{\abs{g(x)-g(y)}}{\abs{x-y}}
    \leq \Lip(g).
\end{align*}
\nc 
We have the estimates $d_\scale(x_0,x_0^*)\leq 2 C_d \homscale$ with high probability and $\inf_{y\in B(x_0,2\homscale-\scale)^c}d_{\scale}^\prime(x_0,y)\geq 2\homscale-\rev 2\rev\scale$ which imply
\begin{align*}
    -\homscale^2\Delta_\infty^\homscale u_n^\homscale(x_0) 
    &\leq 
    2C_d\Lip(g)\homscale\ r_\homscale(x_0)
    \\
    &\qquad
    +
    4C_d^2\Lip(g)\homscale\frac{\scale}{\inf_{y\in B(x_0,2\homscale-\scale)^c}d_{\scale}^\prime(x_0,y)}\left(r_\homscale(x_0)+\frac{1}{2}\right)
    \\
    &\lesssim
    \Lip(g)
    \left(\homscale\ r_\homscale(x_0)
    +
    \frac{\homscale\scale}{\homscale-\scale}
    \right)
    \\
    &\lesssim
    \Lip(g)
    \left(\homscale\ r_\homscale(x_0)
    +
    {\scale}
    \right).
\end{align*}
In the second inequality we used trivial estimates on $\homscale$, $\scale$, and $r_\homscale(x_0)$ to absorb the second term into the first one, and we absorbed dimensional constants into the $\lesssim$ symbol.
\rev In the third inequality we used that $\homscale\geq 8\scale$ to simplify $\tfrac{\homscale\scale}{\homscale-\scale}=\scale\tfrac{1}{1-\tfrac{\scale}{\homscale}}\leq\tfrac{8}{7}\scale\lesssim\scale$. \nc
Dividing by $\homscale^2$ we obtain
\begin{align*}
    -\Delta_\infty^\homscale u_n^\homscale(x_0) 
    \lesssim
    \Lip(g)
    \left(
    \frac{r_\homscale(x_0)}{\homscale}
    +
    \frac{\scale}{\homscale^2}
    \right).
\end{align*}
By \cref{lem:estimate_r} in the appendix and a union bound there exist constants $C_3,C_4,C_5>0$ such that \rev for all $\lambda\geq 0$ \nc with probability at least $1 - C_3 \exp(-C_4\lambda+\log(\homscale/\nlscale)+C_5\log n)$ it holds
%\todo{check this union bound {\red Yes, it will always be polynomial in $n$, so this term cannot be wrong.}} 
\begin{align*}
    r_\homscale(x_0) 
    \lesssim 
    (\log n + \lambda)
    \left(\frac{\log n}{n}\right)^\frac{1}{d}
    \frac{\log n}{\sqrt{\homscale\scale}}.
\end{align*}
Plugging this in we obtain
\begin{align*}
    -\Delta_\infty^\homscale u_n^\homscale(x_0) 
    \lesssim
    \Lip(g)
    \left(
    (\log n + \lambda)
    \left(\frac{\log n}{n}\right)^\frac{1}{d}
    \frac{1}{\sqrt{\homscale^3\scale}}
    +
    \frac{\scale}{\homscale^2}
    \right).
\end{align*}
We conclude the proof, noting that the last probability can be simplified using \labelcref{eq:scaling}:
\begin{align*}
    \log(\homscale/\scale) + C_5 \log n
    &\leq 
    \log\homscale - \log\scale + C_5 \log n
    \\
    &\leq 
    -\log K - (1/d)\log\log n + (1/d)\log n + C_5 \log n
    \leq 
    C_5 \log n
\end{align*}
by changing the value of $C_5>0$ and choosing $K\geq 1$ and $n\geq 3$.
Hence the last probability can be simplified to $1-C_3\exp(-C_4\lambda+C_5\log n)$ and the final result is establish with another union bound.
\end{proof}

The proof of \cref{thm:gen_rates} is now identical to the one presented in our previous paper with the essential ingredient being \cref{thm:consistency}.

\begin{proof}[Proof sketch of \cref{thm:gen_rates}]
The proof works as in \cite[Section 5.3.3]{bungert2022uniform} replacing $\epsilon$ there with $C_6 \homscale$.\rev{} For completeness we sketch the proof below.

From \cref{thm:consistency} we obtain
\begin{align*}
-\Delta_\infty^\homscale u_n^\homscale
&\leq C
\Lip(g)
\left(
(\log n + \lambda)
\left(\frac{\log n}{n}\right)^\frac{1}{d}
\frac{1}{\sqrt{\homscale^3\scale}}
+
\frac{\scale}{\homscale^2}
\right) =: C_{n,\homscale}
\quad\text{ in }\,\domain^{2C_6\homscale},
\end{align*}
for some constant $C>0$. 
The proof strategy is to perturb $u$ to a strict supersolution associated to the operator $-\Delta_\infty^\homscale$.
For this we use \cite[Lemma 4.8, Lemma 4.9]{bungert2022uniform} as in the proof of \cite[Proposition 5.16]{bungert2022uniform} which allows us to choose $w:\domain^{2C_6 \homscale}\to\R$ such that
\begin{align*}
-\Delta_\infty^\homscale w \geq C_{n,\homscale}\quad\text{ in }\,\domain^{2 C_6\homscale}, \qquad \norm{w - (u)_\homscale}_{L^\infty(\domain^{2 C_6\homscale})} \lesssim 
\sqrt[3]{C_{n,\homscale}}
\end{align*}
Since we now have $-\Delta_\infty^\homscale u_n^\homscale\leq C_{n,\homscale}\leq -\Delta_\infty^\homscale w$ we can invoke the comparison principle for the operator $-\Delta_\infty^\homscale$, see \cite[Corollary 3.3]{armstrong2012finite}, to obtain that
\begin{align*}
\sup_{\domain^{(2C_6-1)\homscale}}(u_n^\homscale - (u)_\homscale)&\lesssim 
\sup_{\domain^{(2C_6-1)\homscale}}(u_n^\homscale - w) +  \sqrt[3]{C_{n,\homscale}}
=
\sup_{\domain^{(2C_6-1)\homscale}\setminus\domain^{2C_6\homscale}}(u_n^\homscale - w) +  \sqrt[3]{C_{n,\homscale}}\\
&\lesssim
\sup_{\domain^{(2C_6-1)\homscale}\setminus\domain^{2C_6\homscale}}
(u_n^\homscale - (u)_\homscale) + 2\sqrt[3]{C_{n,\homscale}}
\end{align*}
where we also used the triangle inequality twice. Analogously, we obtain 
\begin{align*}
\sup_{\domain^{(2C_6-1)\homscale}}\left(u^\homscale - (u_n)_\homscale\right)\lesssim
\sup_{\domain^{(2C_6-1)\homscale}\setminus\domain^{2C_6\homscale}}
\left(u^\homscale - (u_n)_\homscale\right)+ 2\sqrt[3]{C_{n,\homscale}}
\end{align*}
The next steps consists in getting rid of the extension operators at the scale of $\homscale$, for which we employ (approximate) Lipschitzness of $u$ (and $u_n$). Utilizing \cite[Lemma 5.9, Lemma 5.10, Lemma 5.11]{bungert2022uniform} this can be done at the cost of an additive error of order $\homscale$, for which we obtain
\begin{align*}
\sup_{X_n\cap \domain^{(2C_6-1)\homscale}} \abs{u - u_n} \lesssim 
\homscale + \sqrt[3]{C_{n,\homscale}}.
\end{align*}
Finally, we extend this result to $X_n\cap \closure\domain$ using again Lipschitzness of $u$ and the data $g$. Namely take $x\in X_n\cap\domain$ and $\tilde x\in X_n\cap (\domain\setminus \domain^{(2C_6-1)\homscale})$ such that $\abs{x-\tilde x}\lesssim \homscale$ which yields
\begin{align*}
\abs{u(x) - u_n(x)} &\leq \abs{u(x) - u(\tilde x)} + \abs{u(\tilde x) - u_n(\tilde x)} + \abs{u_n(\tilde x) - u_n(x)}\\
&\lesssim
\Lip(g)\homscale  + \tau + \sqrt[3]{C_{n,\homscale}},
\end{align*}
where we used that $u_n$ satisfies an approximate Lipschitz estimate of the form
\begin{align*}
    \abs{u_n(x)-u_n(y)} \leq  \Lip_n(u_n)d_\scale(x,y)
    \lesssim
    \Lip_n(g)\left(\abs{x-y}+\scale\right)
    \lesssim \Lip(g)\homscale.
\end{align*}
Hence, we have showed
\begin{align*}
\sup_{X_n\cap \domain} \abs{u - u_n} \lesssim 
\Lip(g)\homscale + \sqrt[3]{C_{n,\homscale}}
\end{align*}
which concludes the proof sketch.
\nc
\end{proof}

%%%%%%%%%%%%%%%%%%%%%%%%%%%%%%%%%%%%%%%%%%%%%%
%% Single Appendix:                         %%
%%%%%%%%%%%%%%%%%%%%%%%%%%%%%%%%%%%%%%%%%%%%%%
%\begin{appendix}
%\section*{???}%% if no title is needed, leave empty \section*{}.
%\end{appendix}
%%%%%%%%%%%%%%%%%%%%%%%%%%%%%%%%%%%%%%%%%%%%%%
%% Multiple Appendixes:                     %%
%%%%%%%%%%%%%%%%%%%%%%%%%%%%%%%%%%%%%%%%%%%%%%
\begin{appendix}

\section{Sub- and superadditivity}
\label{sec:sub_super_additivity}

\begin{lemma}\label[lemma]{lem:fekete}
Let $f:[0,\infty)\to\R$ satisfy
\begin{align}\label{ineq:subadd}
    f(s+t) \lesseqgtr f(s) + f(t),\quad\forall s,t\geq 0.
\end{align}
Then the limit $c:=\lim_{t\to\infty}\frac{f(t)}{t}$ exists in $[-\infty,\infty]$ and it holds $f(t)\gtreqless c t$ for all $t\geq 0$.
\end{lemma}
\begin{proof}
The proof works just like the proof of Fekete's lemma \cite{fekete1923verteilung}.
We just present it in the subhomogeneous case, i.e. $\leq$ in \labelcref{ineq:subadd}.

We first note that by induction \labelcref{ineq:subadd} implies
\begin{align}\label{ineq:f_inductive}
    f(ms) \leq mf(s),\quad\forall s\geq 0,\,m\in\N.
\end{align}
Define $c:=\liminf_{t\to\infty}\frac{f(t)}{t}$ and choose $\eps>0$.
We first show $f(t) \geq ct$ for all $t\geq 0$.
We can choose $s>0$ such that $f(t)/t \geq c - \eps$ for all $t\geq s$.
Then, \labelcref{ineq:f_inductive} implies that
\begin{align*}
  \frac{f(t)}{t} \geq \frac{f(mt)}{mt} \geq c - \eps  
\end{align*}
for all $t\geq 0$ and all sufficiently large $m\in\N$.
Since $\eps>0$ was arbitrary, this establishes the claim.

Now we prove existence of the limit.
Let again $\eps>0$ be arbitrary.
By definition of the limes inferior there exists $s>0$ that $\frac{f(s)}{s} < c + \eps$.
Let $t>s$ and write $t=ms+\tau$ for $0\leq\tau\leq s$.
Using \labelcref{ineq:subadd,ineq:f_inductive} it holds
\begin{align*}
    f(t) = f(ms+\tau) 
    \leq 
    f(ms) + f(\tau) 
    \leq 
    m f(s) + f(\tau).
\end{align*}
Dividing by $t=ms+\tau$ it holds
\begin{align*}
    \frac{f(t)}{t} 
    &\leq 
    m\frac{f(s)}{t} +\frac{f(\tau)}{t}
    =
    \frac{ms}{t}\frac{f(s)}{s} + \frac{f(\tau)}{t}
    \leq 
    \frac{f(s)}{s} + \frac{f(\tau)}{t}
    < c + \eps + \frac{f(\tau)}{t}.
\end{align*}
Sending $t\to\infty$ and using $\tau\leq s$ we obtain
\begin{align*}
    \limsup_{t\to\infty}\frac{f(t)}{t} \leq c+\eps.
\end{align*}
Since $\eps>0$ was arbitrary we obtain the assertion.
\end{proof}

\begin{lemma}\label[lemma]{lem:debruijn}
Let $\mu>1$, assume that $g:[0,\infty)\to[0,\infty)$ is non-decreasing and there exists $z_0>0$ such that $\int_{z_0}^\infty {g(z)}{z^{-2}}\de z<\infty$, and that $f:[0,\infty)\to\R$ satisfies
\begin{align}\label{ineq:sub_or_supadd}
    f(s) + f(t) \lesseqgtr f(s+t) \pm g(s+t)
\end{align}
for all $z_0 \leq s\leq t \leq \mu s$.
Then the limit $c := \lim_{t\to\infty}\frac{f(t)}{t}$ exists in $[-\infty,\infty]$.
\end{lemma}
\begin{proof}
The proof works as the proof of the classical de~Bruijn--Erd\H os theorem \cite[Theorem 23]{de1952some} stated for subhomogeneous sequences, see also \cite{furedi2020nearly}.
\end{proof}

As shown in \cite{furedi2020nearly} the condition $\int_{z_0}^\infty {g(z)}{z^{-2}}\de z<\infty$ cannot be dropped.
They also showed that $g(z)=o(z)$ is necessary but not sufficient, with the counterexample being $g(z) = \frac{z}{\log(z)}$.

If a function is near sub- \emph{and} superadditive, one can get a convergence rate.
We prove the following generalization of a classical result by P\'{o}lya and Szeg\H{o} from 1924, see \cite{polya1972problems}.

\begin{lemma}\label[lemma]{lem:polya}
Assume that $f:[0,\infty)\to\R$ and $g:[0,\infty)\to[0,\infty)$ satisfy
\begin{align}\label{ineq:sub_and_supadd}
    f(s) + f(t) -g(s+t) \leq f(s+t) \leq f(s) + f(t) + g(s+t),\quad \forall 0\leq s \leq t \leq 2s,
\end{align}
and assume that $\sigma := \lim_{s\to\infty}\frac{f(s)}{s}$ exists in $(0,\infty)$.
Then it holds that
\begin{align*}
    \abs{\frac{f(s)}{s} - \sigma} \leq \frac{1}{s} \sum_{n=0}^\infty \frac{g(2^{n+1}s)}{2^{n+1}}.
\end{align*}
\end{lemma}
\begin{proof}
By assumption we have for every $s>0$ that
\begin{align*}
    2f(s) - g(2s) \leq f(2s) \leq 2f(s) + g(2s).
\end{align*}
Dividing by $2s>0$ yields
\begin{align*}
    \frac{f(s)}{s} - \frac{g(2s)}{2s} \leq \frac{f(2s)}{2s} \leq \frac{f(s)}{s} + \frac{g(2s)}{2s}
\end{align*}
and hence
\begin{align}\label{ineq:dyadic}
\abs{\frac{f(2s)}{2s} - \frac{f(s)}{s}} \leq \frac{g(2s)}{2s},\quad\forall s>0.
\end{align}
Furthermore, we can express the limit $\sigma$ as
\begin{align*}
    \sigma = \lim_{n\to\infty} \frac{f(2^{n}s)}{2^n s} = \sum_{n=0}^\infty
    \left(
    \frac{f(2^{n+1}s)}{2^{n+1}s}-\frac{f(2^{n}s)}{2^{n}s}
    \right)
    + \frac{f(s)}{s}.
\end{align*}
Utilizing \labelcref{ineq:dyadic} we get that
\begin{align*}
    \abs{\frac{f(s)}{s}-\sigma} 
    \leq 
    \sum_{n=0}^\infty\abs{    \frac{f(2^{n+1}s)}{2^{n+1}s}-\frac{f(2^{n}s)}{2^{n}s}}
    \leq 
    \sum_{n=0}^\infty
    \frac{g(2^{n+1}s)}{2^{n+1}s},
\end{align*}
as desired.
\end{proof}

\section{Concentration of measure}

The following is a simplified version of a martingale concentration inequality due to Kesten \cite[Theorem 3]{kesten1993speed}.
The original result is pretty general but Kesten's proof is technical and not self-contained.

In our situation it suffices to assume a certain bound to hold true almost surely.
In this situation the statement becomes slightly stronger and we can give an entirely self-contained proof.
 
\begin{lemma}[Simplification of {\cite[Theorem 3]{kesten1993speed}}]\label[lemma]{lem:bounded_difference}
\label[lemma]{lem:abstract_concentration}
Let $\mathbb{F}:=\{\mathcal{F}_k\}_{k \in \N_0}$ be a filtration with $\F_k \uparrow \F$ as $k\to\infty$, 
and let $\{U_k\}_{k\in\N}$ be a sequence of $\F$-measurable positive random variables.
Let $\{M_k\}_{k\in\N_0}$ be a martingale with respect to $\mathbb{F}$.
Assume that for all $k\in\N$ the increments $\Delta_k := M_k-M_{k-1}$ satisfy
\begin{align*}
    \abs{\Delta_k}&\leq c \quad\text{for some $c>0$}, \\
    \Exp{\Delta_k^2 \vert \mathcal{F}_{k-1}} &\leq \Exp{U_k \vert \mathcal{F}_{k-1}}.
\end{align*}
Assume further that there exists a constant $\lambda_0\geq \frac{c^2}{4e}$ such that for all $K\in\N$ the random variable $S_K := \sum_{k=1}^K U_k$ satisfies
\begin{align*}
    S_K \leq \lambda_0\quad\text{almost surely}.
\end{align*}
Then the limit $M:=\lim_{K\to\infty}M_K$ exists almost surely.
Furthermore, there is a universal constant $C>0$ (not depending on $c$ or $\lambda_0$) such that
\begin{align}\label{eq:concentration_M}
    \Prob{M-M_0>\eps} 
    \leq  
    C\exp\left(-\frac{1}{2 \sqrt{e \lambda_0}}\eps\right)\quad\forall \eps\geq 0.
\end{align}
\end{lemma}
\begin{proof}
We first prove that for every $K\in\N$ it holds
\begin{align}\label{eq:concentration_M_K}
    \Prob{M_K-M_0>\eps} 
    \leq  
    C\exp\left(-\frac{1}{2 \sqrt{e \lambda_0}}\eps\right)\quad\forall \eps\geq 0.
\end{align}
We use the Chernoff bounding trick and Markov's inequality to compute for arbitrary $t>0$ and $\eps\geq 0$:
\begin{align*}
    \Prob{M_K - M_0>\eps}
    &=
    \Prob{\exp(t(M_K-M_0))\geq \exp(t\eps)}
    \\
    &\leq 
    \exp(-t\eps)\Exp{\exp(t(M_K-M_0))}
    \\
    &=
    \exp(-t\eps)\Exp{\exp\left(t\sum_{k=1}^K\Delta_k\right)}
    \\
    &=
    \exp(-t\eps)\Exp{\prod_{k=1}^K\exp\left(t\Delta_k\right)}
    \\
    &=
    \exp(-t\eps)\Exp{\prod_{k=1}^K\Exp{\exp\left(t\Delta_k\right)\g\mathcal{F}_{k-1}}}.
\end{align*}
The last equality follows from an iterated application of the law of total expectation.
We proceed by estimating the factors in this product where we use the elementary inequality $\exp(x)\leq 1+x+\frac12 x^2\exp(\abs{x})$ for $x\in\R$.
Using that $\Exp{\Delta_k\vert\mathcal{F}_{k-1}}=0$ and $\abs{\Delta_k}\leq c$ we get
\begin{align*}
    \Exp{\exp\left(t\Delta_k\right)\gm\mathcal{F}_{k-1}}
    &\leq 
    \Exp{1+t\Delta_k + \frac{t^2}{2} \Delta_k^2\exp(t\abs{\Delta_k})\gm\mathcal{F}_{k-1}}
    \\
    &=
    1+\frac{t^2}{2}\exp(t c)\Exp{\Delta_k^2\vert\mathcal{F}_{k-1}}.
\end{align*}
Using this estimate and the elementary inequality $\log(1+x)\leq x$ for $x\geq 0$ and the assumption that $\Exp{\Delta_k^2\vert\mathcal{F}_{k-1}}\leq \Exp{U_k\vert\mathcal{F}_{k-1}}$, we obtain
\begin{align}
    \Prob{M_K-M_0>\eps}
    &\leq 
    \exp(-t\eps)
    \Exp{\prod_{k=1}^K\left(1+\frac{t^2}{2}\exp(tc)\Exp{\Delta_k^2\vert\mathcal{F}_{k-1}}\right)}
    \nonumber
    \\
    &= 
    \exp(-t\eps)
    \Exp{\prod_{k=1}^K\exp\left(\log\left(1+\frac{t^2}{2}\exp(tc)\Exp{\Delta_k^2\vert\mathcal{F}_{k-1}}\right)\right)}
    \nonumber
    \\
    &\leq  
    \exp(-t\eps)
    \Exp{\prod_{k=1}^K\exp\left(\frac{t^2}{2}\exp(tc)\Exp{\Delta_k^2\vert\mathcal{F}_{k-1}}\right)}
    \nonumber
    \\
    &=
    \exp(-t\eps)
    \Exp{\exp\left(\frac{t^2}{2}\exp(tc)\sum_{k=1}^K\Exp{\Delta_k^2\vert\mathcal{F}_{k-1}}\right)}
    \nonumber
    \\
    \label{ineq:bound_1}
    &\leq
    \exp(-t\eps)
    \Exp{\exp\left(\frac{t^2}{2}\exp(tc)\sum_{k=1}^K\Exp{U_k\vert\mathcal{F}_{k-1}}\right)}.
\end{align}
Let us abbreviate
\begin{align}
    A_K := \sum_{k=1}^K\Exp{U_k\vert\mathcal{F}_{k-1}}.
\end{align}
We claim that there exists a universal constant $C>0$ such that for all $t>0$ we have
\begin{equation}\label{eq:tail_bound}
\Prob{A_K \geq t} \leq C \exp\left(-\frac{t}{4\lambda_0}\right).
\end{equation}
For $\ell=0,\dots,K$ let us define
\[Z_\ell = \sum_{k=\ell+1}^K \E[ U_k \g \F_\ell].\]
We note that since $U_k$ are non-negative and $S_K \leq \lambda_0$ almost surely, we have
\begin{equation}\label{eq:Zlbound}
Z_\ell \leq \sum_{k=1}^K \E[ U_k \, |\, \F_\ell] = \E \left[\sum_{k=1}^K U_k \g \F_\ell\right]
=
\E \left[S_K \g \F_\ell\right]
\leq \lambda_0
\end{equation}
almost surely. We will obtain the tail bound on $A_K$ through a moment bound
\begin{equation}\label{eq:moment_bound}
\Prob{A_K\geq t} = \Prob{A_K^r \geq t^r} \leq t^{-r}\E[A_K^r],
\end{equation}
for a particular choice of $r>0$. So we need to estimate the moments of $A_K$.

We compute
\begin{align*}
\E[A_K^r] 
&= 
\E\left[ \left( \sum_{k=1}^K \E[U_k \g \F_{k-1}]\right)^r\right]\\
&=\E \left[\sum_{k_1,k_2,\dots,k_r=1}^K \prod_{i=1}^r \E[U_{k_i} \g \F_{k_i-1}] \right]\\
&=\sum_{k_1,k_2,\dots,k_r=1}^K\E \left[ \prod_{i=1}^r \E[U_{k_i} \g \F_{k_i-1}] \right]\\
&\leq  r!\sum_{1\leq k_1 \leq k_2 \leq \dots \leq k_r\leq K}\E \left[ \prod_{i=1}^r \E[U_{k_i} \g \F_{k_i-1}] \right].
\end{align*}
Following Kesten \cite{kesten1993speed}, we introduce the abbreviation
\[\Gamma_r = \sum_{1\leq k_1 \leq k_2 \leq \dots \leq k_r\leq K}\E \left[ \prod_{i=1}^r \E[U_{k_i} \g \F_{k_i-1}] \right].\]
Then we have
\begin{align*}
\Gamma_r &= 
\sum_{1\leq k_1 \leq k_2 \leq \dots \leq k_r\leq K}
\E \left[ \E \left[ \left(\prod_{i=1}^r \E[U_{k_i} \g \F_{k_i-1}]\right) \gm \F_{k_{r-1}-1} \right] \right]\\
&= \sum_{1\leq k_1 \leq k_2 \leq \dots \leq k_{r-1}\leq K}\E \left[ \E \left[ \sum_{k_r=k_{r-1}}^K\left(\prod_{i=1}^r \E[U_{k_i} \g \F_{k_i-1}]\right) \gm \F_{k_{r-1}-1} \right] \right]\\
&= \sum_{1\leq k_1 \leq k_2 \leq \dots \leq k_{r-1}\leq K}\E \left[ \E \left[ \left(\prod_{i=1}^{r-1} \E[U_{k_i} \g \F_{k_i-1}]\right) \sum_{k_r=k_{r-1}}^K\E[U_{k_r} \g \F_{k_r-1}]\gm \F_{k_{r-1}-1} \right] \right]\\
&= \sum_{1\leq k_1 \leq k_2 \leq \dots \leq k_{r-1}\leq K}\E \left[ \left(\prod_{i=1}^{r-1} \E[U_{k_i} \g \F_{k_i-1}]\right)  \E \left[\sum_{k_r=k_{r-1}}^K\E[U_{k_r} \g \F_{k_r-1}]\gm \F_{k_{r-1}-1} \right] \right]\\
&= \sum_{1\leq k_1 \leq k_2 \leq \dots \leq k_{r-1}\leq K}\E \left[ \left(\prod_{i=1}^{r-1} \E[U_{k_i} \g \F_{k_i-1}]\right) \sum_{k_r=k_{r-1}}^K \E \left[\E[U_{k_r} \g \F_{k_r-1}]\gm \F_{k_{r-1}-1} \right] \right]\\
&= \sum_{1\leq k_1 \leq k_2 \leq \dots \leq k_{r-1}\leq K}\E \left[ \left(\prod_{i=1}^{r-1} \E[U_{k_i} \g \F_{k_i-1}]\right) \sum_{k_r=k_{r-1}}^K \E[U_{k_r} \g \F_{k_{r-1}-1}] \right]\\
&= \sum_{1\leq k_1 \leq k_2 \leq \dots \leq k_{r-1}\leq K}\E \left[ \left(\prod_{i=1}^{r-1} \E[U_{k_i} \g \F_{k_i-1}]\right) Z_{k_{r-1}-1} \right].
\end{align*}
Using the bound \labelcref{eq:Zlbound} we have
\[\Gamma_r \leq \lambda_0 \Gamma_{r-1},\]
and therefore
\[\Gamma_r \leq \lambda_0^{r-1}\Gamma_1 = \lambda_0^{r-1} \E\left[ \sum_{k=1}^K U_k\right] \leq \lambda_0^r.\]
Therefore 
\[\E[A_K^r] \leq r!\,\lambda_0^r\]
and so by  \labelcref{eq:moment_bound} we have
\[\P(A_K \geq t) \leq r! \left(\frac{\lambda_0}{t}\right)^r,\]
for any $r\geq 1$. By Stirling's formula, for $r \geq 1$ we have
\[r! \leq  \sqrt{2\pi r}\left(\frac{r}{e}\right)^re^{\frac{1}{12r}} \leq C r^{r+\frac{1}{2}}e^{-r},\]
and hence
\begin{align*}
    \Prob{A_K\geq t} \leq C r^{r+\frac{1}{2}}e^{-r} \left(\frac{\lambda_0}{t}\right)^r,
\end{align*}
where $C =  \sqrt{2\pi}e^{\frac{1}{12}}$. Now, let $r = \left\lfloor \tfrac{t}{\lambda_0}\right\rfloor$. Then for $t \geq \lambda_0$, so that $r\geq 1$, we have
\begin{align*}
\P(A_K \geq t) 
&\leq 
C \left\lfloor \frac{t}{\lambda_0}\right\rfloor^{\left\lfloor \frac{t}{\lambda_0}\right\rfloor+\frac{1}{2}}
e^{-\left\lfloor \frac{t}{\lambda_0}\right\rfloor}
\left(\frac{\lambda_0}{t}\right)^{\left\lfloor \frac{t}{\lambda_0}\right\rfloor}\\
&\leq 
C \left\lfloor \frac{t}{\lambda_0}\right\rfloor^{\frac{1}{2}} e^{-\left\lfloor \frac{t}{\lambda_0}\right\rfloor}.
\end{align*}
Note that $\lfloor x \rfloor \geq x - 1 \geq \tfrac{x}{2}$ when $x\geq 2$. When $1\leq x < 2$ we have $\lfloor x\rfloor = 1 \geq \tfrac{x}{2}$. Thus $\lfloor x\rfloor \geq \tfrac{x}{2}$ for all $x\geq 1$.  It follows that for $t \geq \lambda_0$ we have
\begin{align*}
\P(A_K \geq t) &\leq C \left( \frac{t}{\lambda_0}\right)^{\frac{1}{2}} e^{- \frac{t}{2\lambda_0}}\\
&= C \left( \frac{t}{\lambda_0}\right)^{\frac{1}{2}} e^{- \frac{t}{4\lambda_0}}e^{- \frac{t}{4\lambda_0}}\\
&\leq C \widetilde{C}e^{- \frac{t}{4\lambda_0}},
\end{align*}
where 
\[\widetilde{C} =\sup_{x\geq 1} \sqrt{x}e^{-\frac{x}{4}} < \infty.\]
Now, if $0 < t \leq \lambda_0$, then we have 
\[e^{- \frac{t}{4\lambda_0}} \geq e^{-\frac{1}{4}} =:c.\]
Making the constant $C$ in \labelcref{eq:tail_bound} larger, if necessary, so that $C\geq c^{-1}$, we can ensure that the right hand side of \labelcref{eq:tail_bound} is larger than one when $t \leq \lambda_0$, so that \labelcref{eq:tail_bound} trivially holds. This completes the proof of \labelcref{eq:tail_bound}.

To see how we can complete the concentration inequality, note from \labelcref{ineq:bound_1} above we have
\[ \Prob{M_K - M_0>\eps} \leq \exp\left(-t\eps\right)\Exp{\exp\left(\frac{t^2}{2}\exp(tc)A_K\right)}.\]
For notational simplicity let us write 
\[\tau = 2t^2\exp(tc),\]
so that
\[ \Prob{%T_s^\dprime-\Exp{T_s^\dprime}
\rev M_k - M_0 \nc >\eps} \leq \exp\left(-t\eps\right)\Exp{\exp\left(\frac{1}{4}\tau A_K\right)}.\]
We now use that for a nonnegative random variable $X$ and a differentiable function $g$ \rev with 
\begin{align}\label{eq:abstract_tail_bound}
    \lim_{x\to\infty}g(x)\Prob{X\geq x}=0  \nc  
\end{align}
we have \rev the following consequence of integration by parts:\nc
\[\E[g(X)] = g(0) + \int_0^\infty g'(x) \P(X \geq x) \de x.\]
\rev 
Taking into account \labelcref{eq:tail_bound}, the choice $g(x) = \exp\left(\frac{1}{4}\tau x\right)$ for $\tau\leq\frac{1}{2\lambda_0}$ and $X=A_K$ satisfies \labelcref{eq:abstract_tail_bound} and hence
\nc 
\[ \Prob{M_K - M_0>\eps} \leq \exp\left(-t\eps\right)\left(1 + \frac{\tau}{4}\int_0^\infty\exp\left(\frac{1}{4}\tau x\right)\P(A_K \geq x) \de x\right).\]
Using \labelcref{eq:tail_bound} we have
\[ \Prob{M_K-M_
0>\eps} \leq \exp\left(-t\eps\right)\left(1 + C\tau\int_0^\infty\exp\left(-\frac{1}{4}\left( \frac{1}{\lambda_0} - \tau\right)x\right)\de x\right).\]
Let us choose
\begin{align}\label{eq:def_of_t}
    t := \min\left(\frac{1}{c},\frac{1}{\sqrt{4e\lambda_0}}\right)
\end{align}
which satisfies $t^2\exp(tc) \leq \frac{1}{4\lambda_0}$ or equivalently $\tau \leq \frac{1}{2\lambda_0}$.
This yields
\begin{align}
\Prob{M_K-M_0>\eps} 
&\leq 
\exp\left(-t\eps\right)\left(1 + \frac{C}{\lambda_0}\int_0^\infty\exp\left(-\frac{x}{8\lambda_0}\right)\de x\right)
=
\exp(-t\eps)(1+8C)
\nonumber
\\
\label{eq:almost_final_conc}
&= C\exp\left(-t\eps\right),
\end{align}
where the constant $C$ changed in the final line. 

We conclude the proof by using $\lambda_0\geq \frac{c^2}{4e}$ to obtain from \labelcref{eq:def_of_t} that:
\begin{align*}
    t = \min\left(\frac{1}{c},\frac{1}{\sqrt{4e\lambda_0}}\right) = \frac{1}{\sqrt{4e\lambda_0}}.
\end{align*}
Plugging this into \labelcref{eq:almost_final_conc} we get the \labelcref{eq:concentration_M_K}.

It remains to be shown that the concentration \labelcref{eq:concentration_M_K} extends to the limiting martingale.
First, we use Doob's martingale convergence theorem to argue that the limit $M:=\lim_{K\to\infty}$ exists.
Then, we show it satisfies the concentration inequality \labelcref{eq:concentration_M}.

Regarding existence of the limit:
Replacing $M_K - M_0$ \rev by \nc $M_0 - M_K$ we also get \labelcref{eq:concentration_M_K} and hence
\begin{align*}
    \Prob{\abs{M_K-M_0}>\eps}
    \leq  
    2C\exp\left(-\frac{1}{2 \sqrt{e \lambda_0}}\eps\right)\quad\forall \eps\geq 0.
\end{align*}
This allows us to bound the expectation of $\abs{M_K-M_0}$ as follows:
\begin{align*}
    \sup_{K\in\N}\Exp{\abs{M_K-M_0}}
    &=
    \sup_{K\in\N}
    \int_0^\infty\Prob{\abs{M_K-M_0}> \eps}\d\eps 
    \\
    &\leq 
    \sup_{K\in\N}
    2C
    \int_0^\infty
    \exp\left(-\frac{1}{2 \sqrt{e \lambda_0}}\eps\right)\d\eps
    = 
    4C\sqrt{e \lambda_0} < \infty.
\end{align*}
Using this and the triangle inequality we obtain
\begin{align*}
    \sup_{K\in\N}\Exp{\abs{M_K}} \leq \sup_{K\in\N}\Exp{\abs{M_K-M_0}} + \Exp{\abs{M_0}}
    < \infty.
\end{align*}
Doob's martingale convergence theorem \cite{doob1953stochastic} then implies that $M := \lim_{K\to\infty}M_K$ exists almost surely.

To show the concentration, note that for any $\lambda\in(0,1)$ we get using \labelcref{eq:concentration_M_K}:
\begin{align*}
    \Prob{{M - M_0}>\eps} 
    &\leq 
    \Prob{{M - M_K} + {M_K - M_0} > \eps}
    \\
    &\leq 
    \Prob{{M - M_K}> (1-\lambda)\eps} + 
    \Prob{{M_K - M_0}>\lambda\eps }
    \\
    &\leq 
    \Prob{\abs{M - M_K}> (1-\lambda)\eps} + 
    C\exp\left(-\frac{1}{2 \sqrt{e \lambda_0}}\lambda\eps\right).
\end{align*}
Since almost sure convergence of $M_K$ to $M$ implies convergence in probability, we can send $K\to\infty$ and the first term goes to zero.
Hence, we obtain
\begin{align*}
    \Prob{{M - M_0}>\eps} 
    \leq 
    C\exp\left(-\frac{\lambda}{2 \sqrt{e \lambda_0}}\eps\right).
\end{align*}
for any $\lambda\in(0,1)$.
Sending $\lambda\to 1$ we finally obtain \labelcref{eq:concentration_M}.
\end{proof}

The following is a useful technical statement which is similar to \cite[Lemma 4.3]{howard2001geodesics}.
\begin{lemma}\label[lemma]{lem:concentration_to_max}
For $s\geq s_0>1$ and $1\leq i\leq n_s$ let $Y_i^{(s)}$ be non-negative random variables on the probability space $(\probSpace, \sigAlg, \probMeasure)$ such that for constants $C_0,C_1,C_2,C_3> 0$ and exponents $\alpha_0,\alpha_1>0$ it holds
\begin{align}
\Exp{Y_i^{(s)}} &\leq C_0 s^{\alpha_0},\label{eq:assLemA}\\
n_s &\leq C_1 s^{\alpha_1}\label{eq:assLemB}\\
\Prob{\abs{Y_i^{(s)} - \Exp{Y_i^{(s)}}} > t} &\leq C_2 \exp(-C_3 t),\qquad\forall t>0,\label{eq:assLemC}
\end{align}
then for a constant $C_4>0$ we have that
\begin{align*}
\Exp{\max_{1\leq i\leq n_s} \Big(\Exp{Y_i^{(s)}} - Y_i^{(s)}\Big)} \leq %
C_4 \log(s),\qquad\forall s\geq s_0.
\end{align*}
If instead of \labelcref{eq:assLemA} one assumes that 
\begin{align}
{Y_i^{(s)}} &\leq C_0 s^{\alpha_0},\label{eq:assLemA'}
\end{align}
then it holds
\begin{align*}
\Exp{\max_{1\leq i\leq n_s} \Big(Y_i^{(s)}-\Exp{Y_i^{(s)}}\Big)} \leq %
C_4 \log(s),\qquad\forall s\geq s_0.
\end{align*}
\end{lemma}
\begin{proof}
We consider the random variable $M:=\max_{1\leq i\leq n_s} \Big(\Exp{Y_i^{(s)}} - Y_i^{(s)}\Big)$ for which non-negativity and the assumption \labelcref{eq:assLemA} yield
\begin{align*}
M \leq \max_{1\leq i\leq n_s} \Exp{Y_i^{(s)}} \leq C_0 s^{\alpha_0}.
\end{align*}
We define $f(s):= \frac{\alpha_0 + \alpha_1}{C_3} \log(s)$ for which we have
\begin{align*}
\Exp{M} 
&=
\Exp{M \vert M \leq f(s)} \Prob{M \leq f(s)} 
+
\Exp{M \vert M > f(s)} \Prob{M > f(s)}
\\
&\leq 
f(s)\ \Prob{M \leq f(s)} + C_0\ s^{\alpha_0}\ \Prob{M \geq f(s)}\\
&\leq
f(s) + C_0\ s^{\alpha_0}\ \sum_{i=1}^{n_s} \Prob{\Exp{Y_i^{(s)}} - Y_i^{(s)}\geq f(s)}\\
&\overset{\labelcref{eq:assLemC}}{\leq}
f(s) + C_0 C_2\ s^{\alpha_0}\ n_s \exp(-C_3 f(s))\\
&\overset{\labelcref{eq:assLemB}}{\leq}
f(s) + C_0 C_1 C_2\ s^{\alpha_0 + \alpha_1} \exp(-(\alpha_0 + \alpha_1)\log(s))\\
&=
f(s) + C_1\\
&\leq C_4 \log(s),\qquad\forall s\geq s_0,
\end{align*}
where we choose $C_4>0$ sufficiently large.
For proving the second statement one repeats the proof verbatim for the random variable $M:=\max_{1\leq i\leq n_s} \Big(Y_i^{(s)}-\Exp{Y_i^{(s)}}\Big)$ for which non-negativity and assumption \labelcref{eq:assLemA'} yield
\begin{align*}
    M \leq \max_{1\leq i\leq n_n}Y_i^{(s)} \leq C_0 s^{\alpha_0}.
\end{align*}
\end{proof}

\section{Estimates for ratio convergence}
\label{sec:ratio_cvgc}

In this section we provide high probability estimates for $\overline{d}_\homscale(x_0)$, $\underline{d}_{2\homscale}(x_0)$, and $r_\homscale(x_0)$, defined in \labelcref{eq:d_bar_under_r}.

\begin{lemma}\label[lemma]{lem:d_bar}
Under the conditions of \cref{thm:properties_distance_X_n} there exist constants $C_1,C_2,C_3,C_4>0$ such that for every $\lambda\geq 0$ with probability at least $1 - C_1 \exp(-C_2\lambda+\log(\homscale/\nlscale))$ it holds that
\begin{align}
    \overline{d}_\homscale(x_0) \leq \Exp{d_{\scale}^\prime(0,\homscale e_1)} 
    + \lambda K\left(\frac{\log n}{n}\right)^\frac{1}{d}\sqrt{\frac{\homscale}{\scale}}
    +
    C_3 \left(\frac{\log n}{n}\right)^\frac{1}{d}\log n\sqrt{\frac{\homscale}{\scale}}
    +
    C_4 \scale.
\end{align}
\end{lemma}
\begin{proof}
We perform a covering argument similar to the proof of \cref{prop:superadditivity}.
For this we cover $B(x_0,\homscale)$ with deterministic points $\{x_i \st i=1,\dots,n_\homscale\}$ where $n_\homscale$ is of order $(\homscale/\scale)^d$, making sure that for every $x\in B(x_0,\homscale)$ there exists $1\leq i\leq n_\homscale$ such that $\abs{x-x_i}\leq\scale$.
Since the supremum in the definition of $\overline{d}_\homscale(x_0)$ is taken over finitely many points, it is achieved for some $x_\homscale\in B(x_0,\homscale)\cap X_n$.
Furthermore, by definition of the covering there exists $1\leq i^* \leq n_\homscale$ such that $\abs{x_\homscale-x_{i^*}}\leq\scale$.
Then it follows
\begin{align*}
    \overline{d}_\homscale(x_0)
    &=
    \max_{x\in B(x_0, \homscale)\cap X_n} d_{\scale}^\prime(x_0,x)
    =
    d_{\scale}^\prime(x_0, x_{\homscale})
    \leq 
    d_{\scale}^\prime(x_0, x_{i^*})
    +
    d_{\scale}^\prime(x_{i^*}, x_\homscale)
    \\
    &\leq 
    \max_{1\leq i \leq n_\homscale} d_{\scale}^\prime(x_0,x_i) + (C_d\rev+1\nc)\scale.
\end{align*}
By \cref{thm:properties_distance_X_n} we can estimate the probability of the event
\begin{align*}
    A_i := \Set{d_\scale^\prime(x_0,x_i) \leq \Exp{d_\scale^\prime(x_0,x_i)} + \lambda K \left(\frac{\log n}{n}\right)^\frac{1}{d}\sqrt{\frac{\abs{x_0-x_i}}{\scale}}}
\end{align*}
as follows:
\begin{align*}
    \Prob{A_i} \geq 1 - C_1 \exp(-C_2\lambda)\qquad\forall \lambda\geq 0.
\end{align*}
Using a union bound it holds
\begin{align*}
    \Prob{\bigcap_{i=1}^{n_\homscale}A_i} \geq 1 - n_\homscale C_1 \exp(-C_2\lambda).
\end{align*}
Changing the constants $C_1,C_2$, it holds with probability at least $1 - C_1 \exp(-C_2\lambda+\log(\homscale/\nlscale))$ that
\begin{align*}
    \overline{d}_\homscale(x_0)
    &\leq 
    \max_{1\leq i \leq n_\homscale} d_{\scale}^\prime(x_0,x_i) + (C_d\rev+1\nc)\scale
    \\
    &\leq 
    \max_{1\leq i \leq n_\homscale}
    \Exp{d_{\scale}^\prime(x_0,x_i)}
    +
    \lambda K\left(\frac{\log n}{n}\right)^\frac{1}{d}\sqrt{\frac{\homscale}{\scale}}
    +
    (C_d\rev+1\nc)\scale.
\end{align*}
Using $\abs{x_i}\leq\homscale$ and the almost monotonicity of the expectation from \cref{thm:properties_distance_X_n} concludes the proof.
\end{proof}
\begin{lemma}\label[lemma]{lem:d_underbar}
Under the conditions of \cref{thm:properties_distance_X_n} there exist constants $C_1,C_2,C_3,C_4>0$ such that for every $\lambda\geq 0$ with probability at least $1 - C_1 \exp(-C_2\lambda+\log(\homscale/\nlscale))$ it holds that
\begin{align}
    \underline{d}_{2\homscale}(x_0) \geq \Exp{d_{\scale}^\prime(0,2\homscale e_1)} 
    - \lambda K\left(\frac{\log n}{n}\right)^\frac{1}{d}\sqrt{\frac{\homscale}{\scale}}
    -
    C_3 \left(\frac{\log n}{n}\right)^\frac{1}{d}\log n\sqrt{\frac{\homscale}{\scale}}
    -
    C_4 \scale.
\end{align}
\end{lemma}
\begin{proof}
We start by observing that
\begin{align*}
    \underline{d}_{2\homscale}(x_0)
    &=
    \inf_{y\in B(x_0, 2\homscale-\scale)^c} d_{\scale}^\prime(x_0,y)
    \geq
    \inf_{y\in B(x_0, 2\homscale-2\scale)^c\cap X_n} d_{\scale}^\prime(x_0,y)
\end{align*}
holds true, since one can shorten any path realizing $d_\scale(x_0,y)$ for $y\in B(x_0,2\homscale-\scale)^c$ by removing the last hop from a graph point to $y$.
Furthermore, it holds that
\begin{align*}
    \underline{d}_{2\homscale}(x_0)
    \geq
    \inf_{y\in B(x_0, 2\homscale-2\scale)^c\cap X_n} d_{\scale}^\prime(x_0,y)
    \geq 
    \inf_{y\in B(x_0, 2\homscale)^c\cap X_n} d_{\scale}^\prime(x_0,y)
    - 2 C_d \scale.
\end{align*}
Thanks to \cref{lem:path_T_s'_in_box} the paths of $d_\scale^\prime(x_0,y)$ for $y\in B(x_0, 2\homscale)^c\cap X_n$ are confined in large enough ball with radius $2 C_d^\prime\homscale$.
Covering the annulus between $B(x_0,2\homscale)$ and $B(x_0,2 C_d^\prime \homscale)$ with order $(\homscale/\scale)^d$ deterministic points $\{x_i\st 1\leq i \leq n_\homscale\}$ similar to the proof of \cref{lem:d_bar} one obtains
\begin{align*}
    \underline{d}_{2\homscale}(x_0)
    \geq
    \inf_{y\in B(x_0, 2\homscale)^c\cap X_n} d_{\scale}^\prime(x_0,y)
    - 2 C_d \scale
    \geq 
    \min_{1\leq i \leq n_\homscale} d_{\scale}^\prime(x_0,x_i) - 3 C_d\scale.
\end{align*}
Similar to before we define the event
\begin{align*}
    A_i := \Set{d_\scale^\prime(x_0,x_i) \geq \Exp{d_\scale^\prime(x_0,x_i)} - \lambda K \left(\frac{\log n}{n}\right)^\frac{1}{d}\sqrt{\frac{\abs{x_0-x_i}}{\scale}}},
\end{align*}
whose probability, according to \cref{thm:properties_distance_X_n}, is at most
\begin{align}
    \Prob{A_i} \geq 1 - C_1 \exp(-C_2\lambda)\qquad\forall\lambda\geq 0.
\end{align}
Using a union bound it holds
\begin{align*}
    \Prob{\bigcap_{i=1}^{n_\homscale}A_i} \geq 1 - n_\homscale C_1 \exp(-C_2\lambda).
\end{align*}
Changing the constants $C_1,C_2$ it holds that with probability at least $1 - C_1 \exp(-C_2\lambda+\log(\homscale/\nlscale))$ it holds
\begin{align*}
    \underline{d}_{2\homscale}(x_0)
    &\geq 
    \min_{1\leq i \leq n_\homscale} d_{\scale}^\prime(x_0,x_i) - 3 C_d \scale
    \\
    &\geq
    \min_{1\leq i \leq n_\homscale}
    \Exp{d_{\scale}^\prime(x_0,x_i)}
    -
    \lambda K\left(\frac{\log n}{n}\right)^\frac{1}{d}\sqrt{\frac{2 C_d^\prime\homscale}{\scale}}
    -
    3 C_d \scale.
\end{align*}
Using that $\abs{x_i-x_0}\geq 2\homscale$, utilizing the almost monotonicity from \cref{thm:properties_distance_X_n}, and defining suitable constants $C_3,C_4$ concludes the proof.
\end{proof}
Now we can estimate $r_\homscale(x_0)$ as follows:
\begin{lemma}\label[lemma]{lem:estimate_r}
Under the conditions of \cref{thm:properties_distance_X_n} there exist constants $C_1,\dots,C_5>0$ such that for every $\lambda\geq 0$ with probability at least $1 - C_1 \exp(-C_2\lambda+\log(\homscale/\nlscale))$ it holds that
\begin{align}
    r_\homscale(x_0)
    \leq 
    C_3 \frac{\scale}{\homscale} 
    + 
    C_4 \left(\frac{\log n}{n}\right)^\frac{1}{d}
    \frac{\log n}{\sqrt{\homscale\scale}}
    +
    \lambda C_5\left(\frac{\log n}{n}\right)^\frac{1}{d}\frac{1}{\sqrt{\homscale \scale}}.
\end{align}
\end{lemma}
\begin{proof}
The result follows from combining \cref{thm:properties_distance_X_n,lem:d_bar,lem:d_underbar}
\end{proof}
\begin{remark}
The leading order of this expression is $(\log n/n)^\frac{1}{d}/\sqrt{\homscale\scale}$ since $\scale\gtrsim (\log n/n)^\frac{1}{d}$.
In our prior work \cite{bungert2022uniform} (in particular, Lemma 5.5 therein) we had $r_\homscale(x_0)\lesssim (\log n/n)^\frac{1}{d}/\scale$ independently of $\homscale$.
However, for $\homscale\geq \scale$ our results here are better.
\end{remark}

\section{Numerical examples}

In the following examples we want to numerically examine the behaviour of graph distance functions. The code for the experiments can be found at \url{https://github.com/TimRoith/PercolationConvergenceRates}.
\begin{figure}[t!]
\begin{subfigure}[b]{0.33\textwidth}%
\centering
\includegraphics[width=\textwidth]{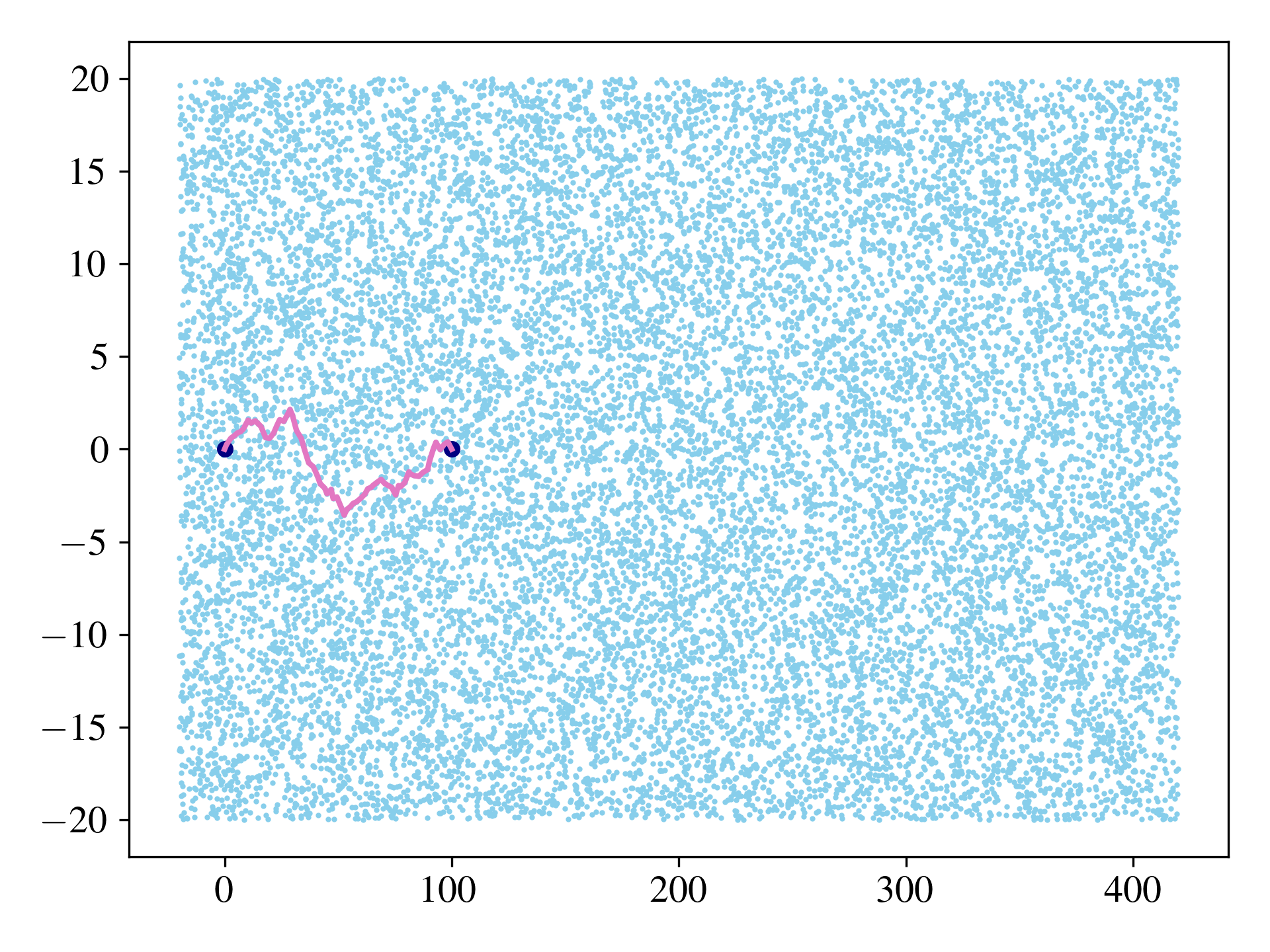}
\caption{$s=100$}
\end{subfigure}%
\hfill%
\begin{subfigure}[b]{0.33\textwidth}
\centering
\includegraphics[width=\textwidth]{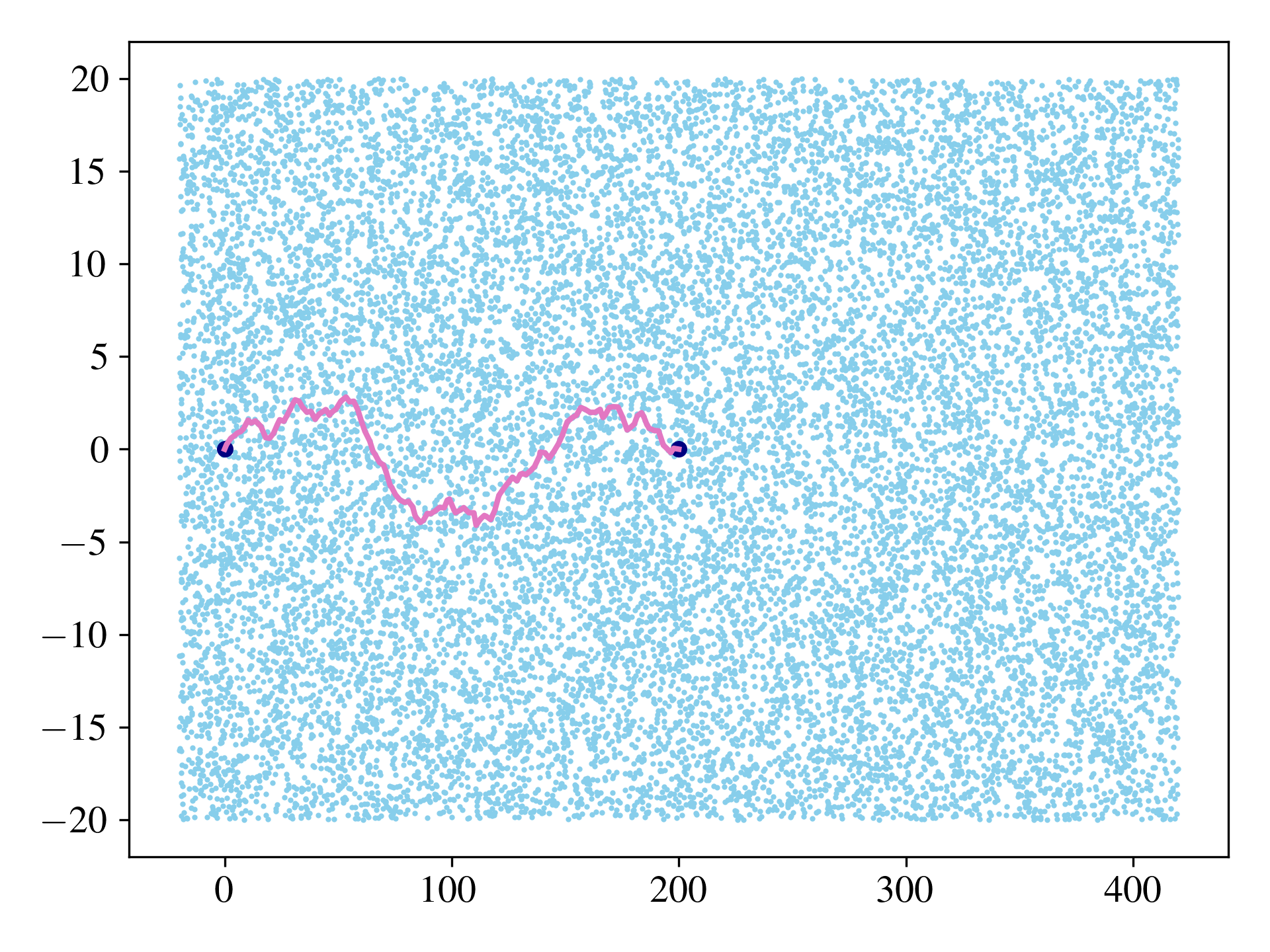}
\caption{$s = 200$}
\end{subfigure}
\hfill%
\begin{subfigure}[b]{0.33\textwidth}
\centering
\includegraphics[width=\textwidth]{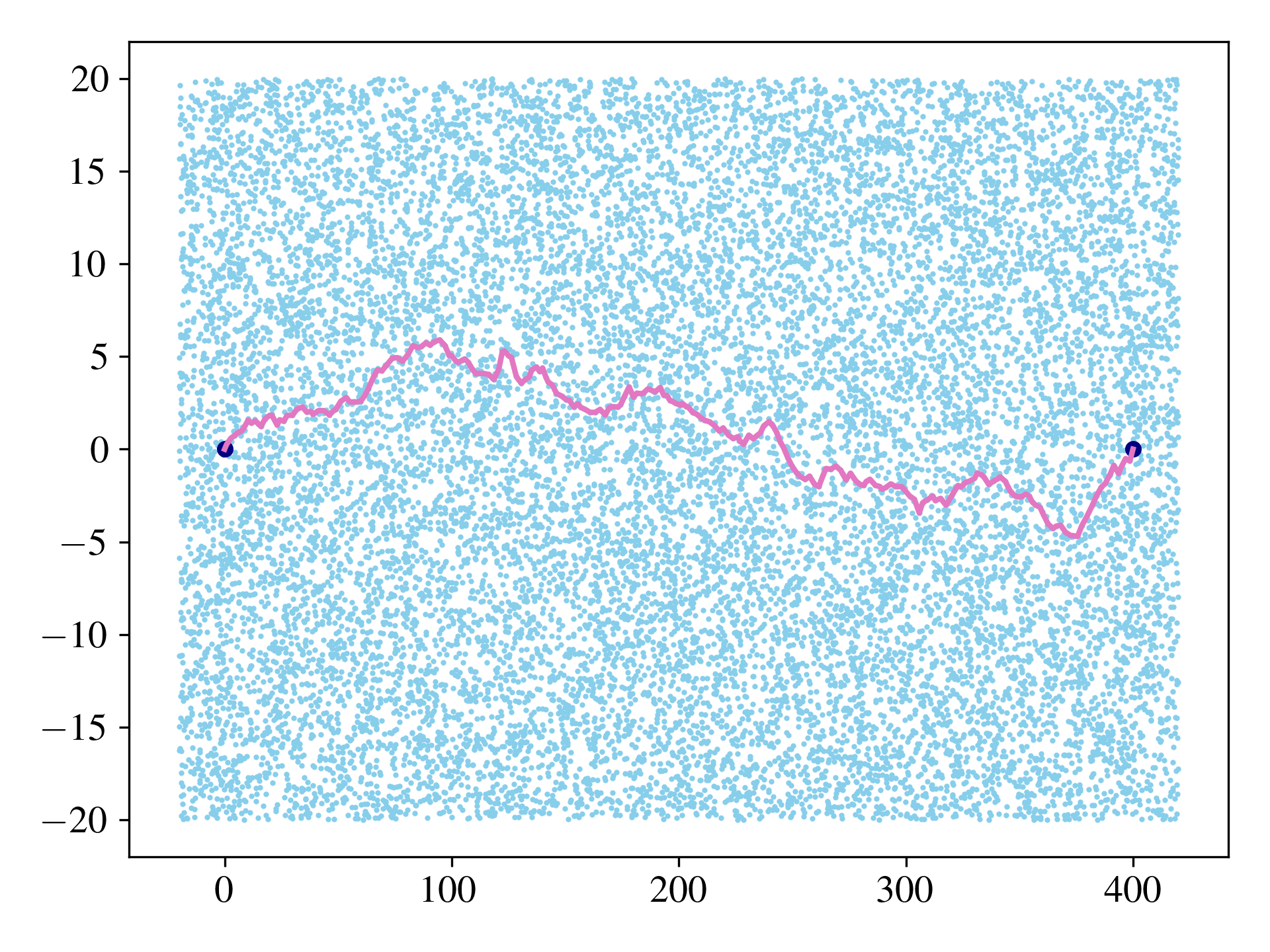}
\caption{$s = 400$}
\end{subfigure}
\caption{Visualizations of optimal paths for $d=2$.\label{fig:2D}}
% \end{figure}
% %
% %
% %
% %
% \begin{figure}[htb]
\begin{subfigure}[b]{0.33\textwidth}%
\centering
\includegraphics[width=\textwidth,trim=70pt 10pt 40pt 10pt,clip]{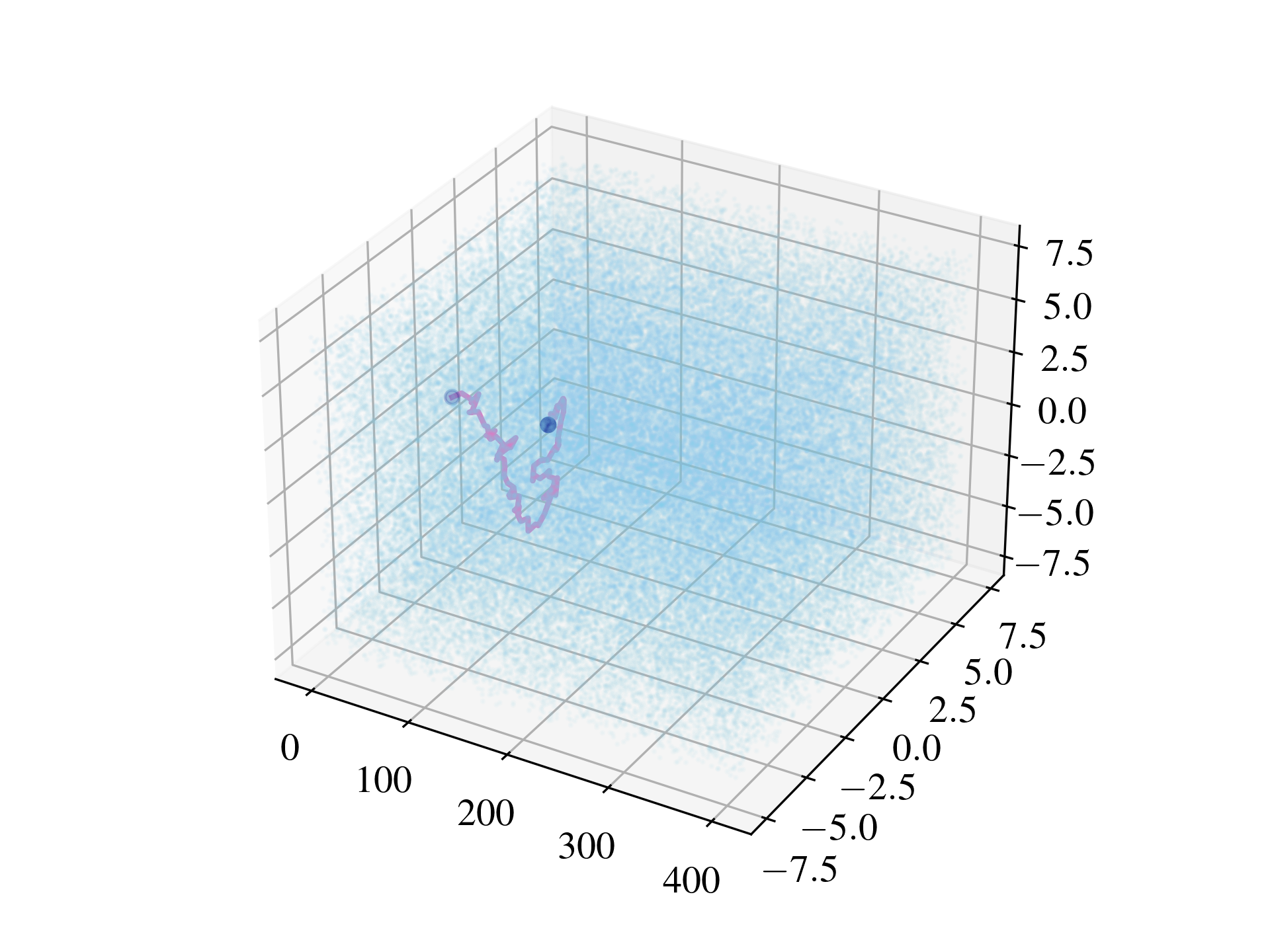}
\caption{$s = 100$}
\end{subfigure}%
\hfill%
\begin{subfigure}[b]{0.33\textwidth}
\centering
\includegraphics[width=\textwidth,trim=70pt 10pt 40pt 10pt,clip]{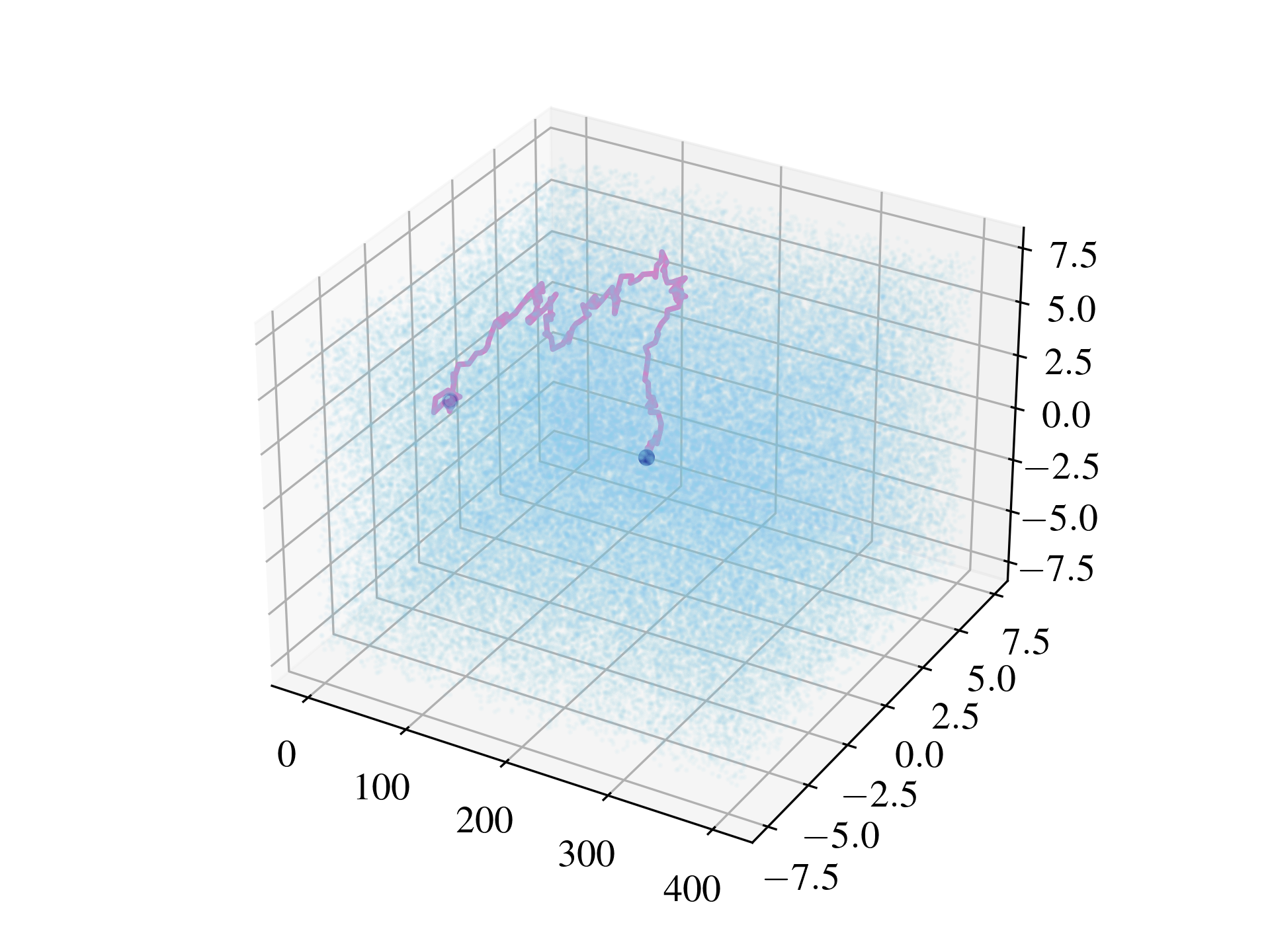}
\caption{$s = 200$}
\end{subfigure}
\hfill%
\begin{subfigure}[b]{0.33\textwidth}
\centering
\includegraphics[width=\textwidth,trim=70pt 10pt 40pt 10pt,clip]{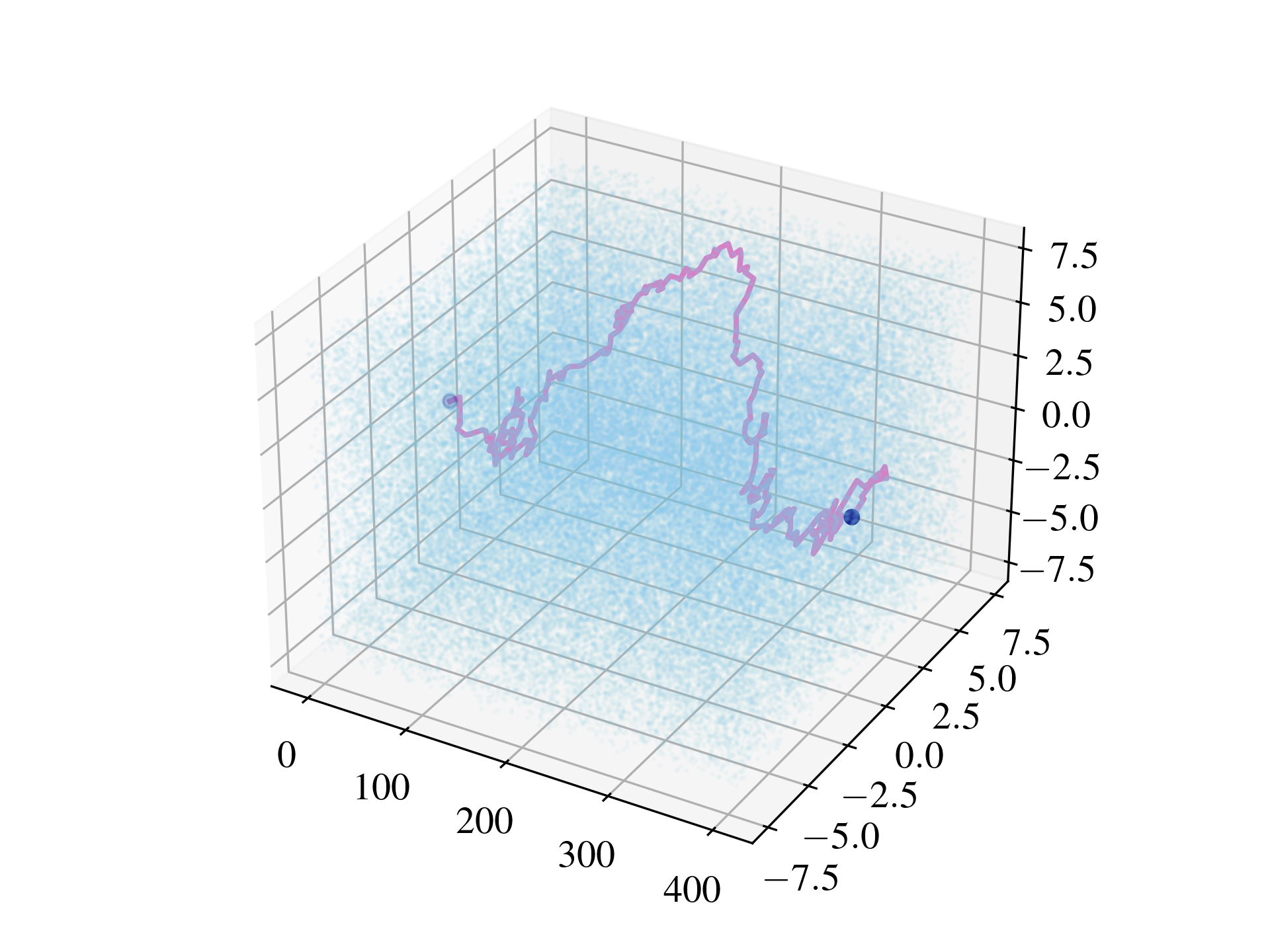}
\caption{$s = 400$}
\end{subfigure}
\caption{Visualizations of optimal paths for $d=3$.\label{fig:3D}}
\end{figure}
Numerical convergence studies for the graph infinity Laplacian equation at percolation length scales can be found in \cite{bungert2022uniform}.
Since we are interested in distances along the first dimension, we consider domains of the form
\begin{align*}
\Omega_s = [-s^{1/d}, s+s^{1/d}] \times \left[-s^{1/d}, s^{1/d}\right]^{d-1}
\end{align*}
for different values of $s>0$ and dimensions $d=2$ and $d=3$. 
In order to observe the limiting behavior for $s\to\infty$ we evaluate distances at $s_i = 100\cdot 2^i$ for $i=1,\ldots, N\in\N$. For each distance we perform $K\in\N$ different trials, where in each trial $k=1,\ldots,K$ we sample a Poisson point process $P_{i,k}\subset\Omega_{s_i}$ with unit density and then set
\begin{align*}
\overline{\mathrm{T}}_{i}:= \frac{1}{K}\sum_{k=1}^K d_{h_{s_i}, P_{i,k}}(0, s_i e_1).
\end{align*}
Here, the scaling is chosen as
\begin{align}\label{eq:logscaling}
h_s := a\log(s)^{1/d},
\end{align}
where $a>0$ is a factor.
We visualize some paths in \cref{fig:2D,fig:3D}.
Note that all axes but the $x$-axis are scaled for visualization purposes which makes the paths look less straight than they actually are.
\begin{figure}[t]
\centering
\begin{subfigure}[b]{\textwidth}
\includegraphics[width=.99\textwidth]{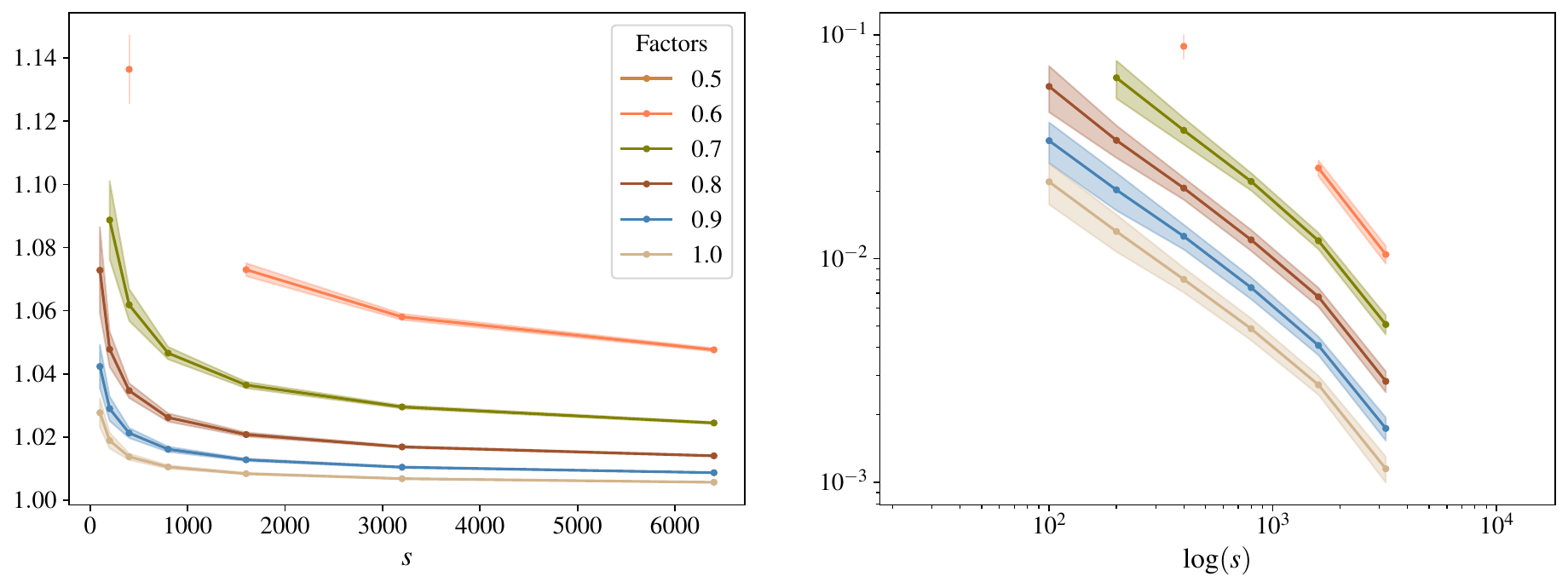}
\caption{Dimension $d=2$.}
\end{subfigure}
\begin{subfigure}[b]{\textwidth}
\includegraphics[width=.99\textwidth]{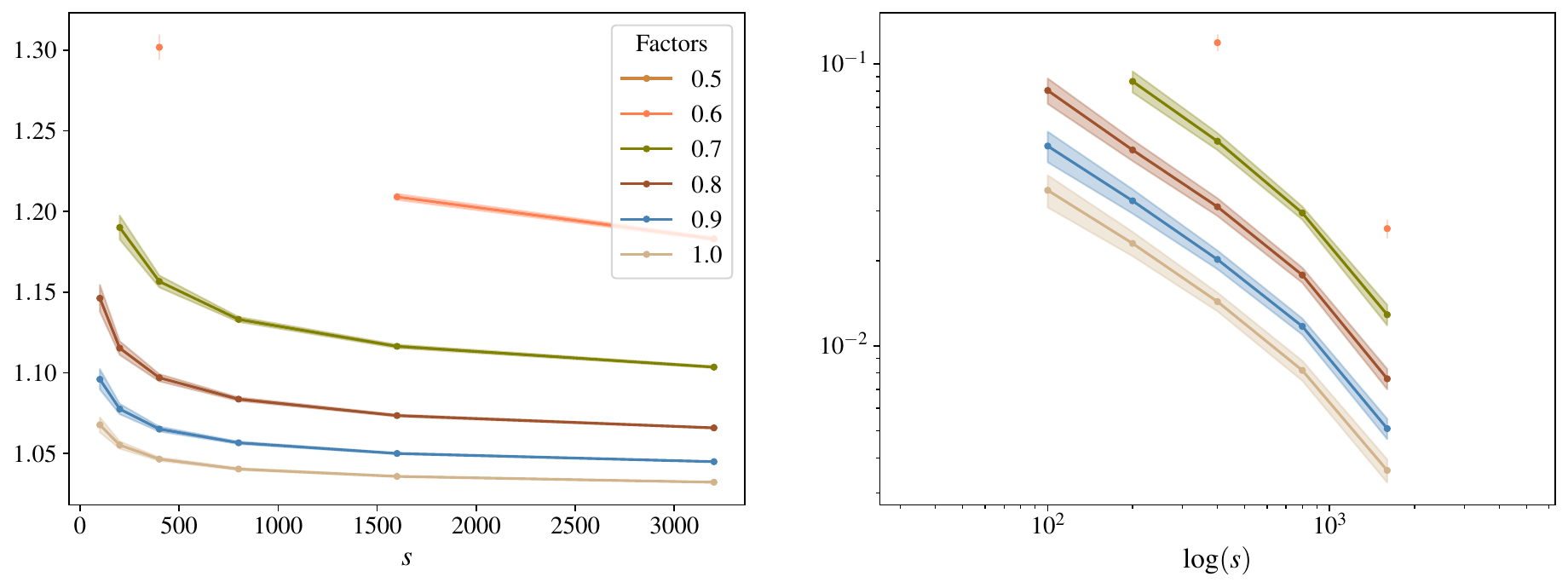}
\caption{Dimension $d=3$.}
\end{subfigure}
\caption{The averaged distance from $0$ to $s$ as defined in \labelcref{eq:sample_ratio}. Each curve corresponds to a different choice of the factor $a>0$ in front of the log scaling \labelcref{eq:logscaling}.
The left plots have standard axes, the right ones are log-log plots of the error to the last iterate. 
}\label{fig:cvg_rate_1}
\end{figure}

\noindent
In \cref{fig:cvg_rate_1} the values
\begin{align}\label{eq:sample_ratio}
\frac{\overline{\mathrm{T}}_i}{s_i},\quad i=1,\ldots, N
\end{align}
including one standard deviation error bands are visualized for $N=6$, averaged over $K=100$ trials.
The figure on the right is a log-log plot of the values
\begin{align*}
\abs{\frac{\overline{\mathrm{T}}_i}{s_i} - \sigma},\qquad i=1,\dots,N-1,
\end{align*}
is visualized, where the limiting constant $\sigma$ from \cref{thm:convergence}, whose analytic value is not known, is approximated by $\sigma = \frac{\overline{\mathrm{T}}_N}{s_N}$.
Numerically we observe rates close to linear which seems to indicate that indeed the strong form of superadditivity \labelcref{ineq:superadditivity_conj} holds true.
Note, however, that these results must be taken with a grain of salt since the limit $\sigma$ is approximated with the numerical limit which positively influences the observed rates, as can be observed in the right figure.

\begin{figure}[t]
\centering
\begin{subfigure}[b]{\textwidth}
\includegraphics[width=.99\textwidth]{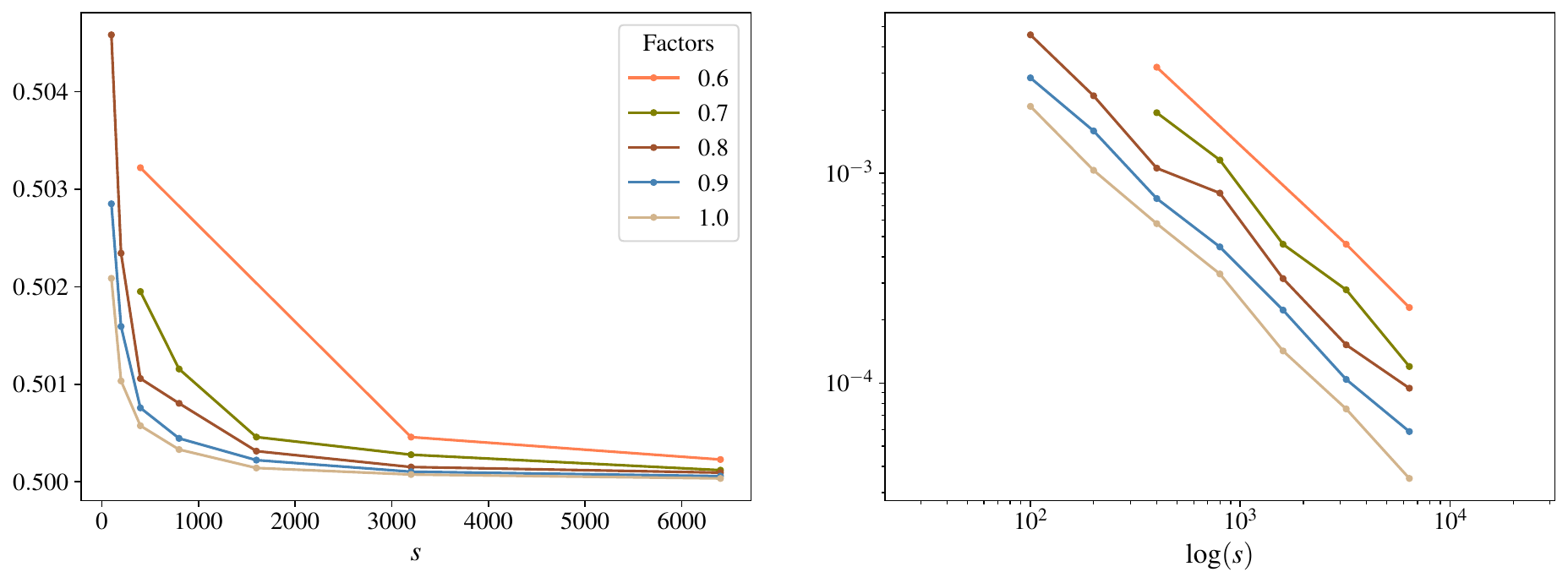}
\caption{Dimension $d=2$.}
\end{subfigure}
\begin{subfigure}[b]{\textwidth}
\includegraphics[width=.99\textwidth]{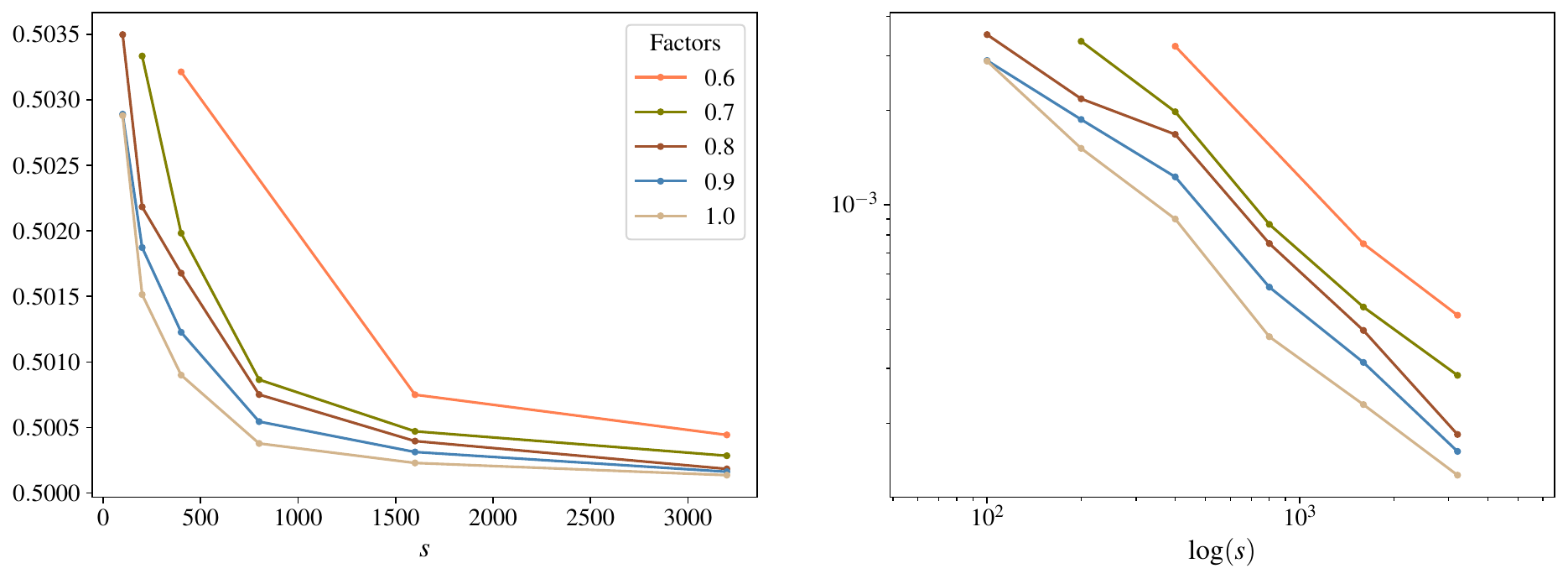}
\caption{Dimension $d=3$.}
\end{subfigure}
\caption{Illustration of the ratio convergence for different scaling parameters. 
The left plots shows the ratios which converge to $1/2$. The log-log plots of the errors on the right illustrate an algebraic convergence rate which appears to be close to $1/s$.\label{fig:cvg_rate_2}}
\end{figure}

\noindent
Additionally, we want to evaluate the ratio convergence which we proved in \cref{prop:ratio_convergence_exp}. 
Therefore, we also compute
\begin{align*}
\overline{\mathrm{T}}_{i,1/2}:= \frac{1}{K}\sum_{k=1}^K d_{h_{s_i}, P_{i,k}}\left(0, \frac{s_i}{2} e_1\right)
\end{align*}
and visualize the ratios and errors
\begin{align*}
\frac{\overline{\mathrm{T}}_{i}}{\overline{\mathrm{T}}_{i,1/2}},
\qquad
\abs{\frac{\overline{\mathrm{T}}_{i}}{\overline{\mathrm{T}}_{i,1/2}} - \frac{1}{2}}
\end{align*}
for $i=1,\ldots, N$, in \cref{fig:cvg_rate_2}.
The rates that we numerically observe appear to be of order ${1}/{s}$ whereas our theoretical result from \cref{prop:ratio_convergence_exp} assures ${1}/{\sqrt{s}}$ up to log factors.

\end{appendix}

%%%%%%%%%%%%%%%%%%%%%%%%%%%%%%%%%%%%%%%%%%%%%%
%% Support information, if any,             %%
%% should be provided in the                %%
%% Acknowledgements section.                %%
%%%%%%%%%%%%%%%%%%%%%%%%%%%%%%%%%%%%%%%%%%%%%%
% \begin{acks}[Acknowledgments]
% The authors would like to thank ...
% \end{acks}
%%%%%%%%%%%%%%%%%%%%%%%%%%%%%%%%%%%%%%%%%%%%%%
%% Funding information, if any,             %%
%% should be provided in the                %%
%% funding section.                         %%
%%%%%%%%%%%%%%%%%%%%%%%%%%%%%%%%%%%%%%%%%%%%%%
\begin{funding}
Part of this work was also done while LB and TR were in residence at Institut Mittag-Leffler in Djursholm, Sweden during the semester on \textit{Geometric Aspects of Nonlinear Partial Differential Equations} in 2022, supported by the Swedish Research Council under grant no. 2016-06596.
LB also acknowledges funding by the Deutsche Forschungsgemeinschaft (DFG, German Research Foundation) under Germany's Excellence Strategy - GZ 2047/1, Projekt-ID 390685813.
Most of this study was carried out while LB was affiliated with the Hausdorff Center for Mathematics at the University of Bonn.
JC acknowledges funding from NSF grant DMS:1944925, the Alfred P.~Sloan foundation, and a McKnight Presidential Fellowship. 
TR acknowledges support from DESY (Hamburg, Germany), a member of the Helmholtz Association HGF, by the German Ministry of Science and Technology (BMBF) under grant agreement No. 05M2020 (DELETO) and the European Unions Horizon 2020 research and innovation programme under the Marie Sk{\l}odowska-Curie grant agreement No 777826 (NoMADS). Most of this study was carried out while TR was affiliated with the Friedrich-Alexander-Universität Erlangen-Nürnberg.
\end{funding}

%%%%%%%%%%%%%%%%%%%%%%%%%%%%%%%%%%%%%%%%%%%%%%
%% Supplementary Material, including data   %%
%% sets and code, should be provided in     %%
%% {supplement} environment with title      %%
%% and short description. It cannot be      %%
%% available exclusively as external link.  %%
%% All Supplementary Material must be       %%
%% available to the reader on Project       %%
%% Euclid with the published article.       %%
%%%%%%%%%%%%%%%%%%%%%%%%%%%%%%%%%%%%%%%%%%%%%%
%\begin{supplement}
%\stitle{???}
%\sdescription{???.}
%\end{supplement}

%%  MR numbers will be added by VTeX.                      %%
%%                                                         %%
%%  Use \cite{...} to cite references in text.             %%
%%                                                         %%
%%%%%%%%%%%%%%%%%%%%%%%%%%%%%%%%%%%%%%%%%%%%%%%%%%%%%%%%%%%%%

%% if your bibliography is in bibtex format, uncomment commands:
\bibliographystyle{imsart-number} % Style BST file (imsart-number.bst or imsart-nameyear.bst)
\bibliography{bibliography}       % Bibliography file (usually '*.bib')

%% or include bibliography directly:
% \begin{thebibliography}{}
% \bibitem{b1}
% \end{thebibliography}

\end{document}